\documentclass[reqno]{amsart}
\usepackage[margin=1in]{geometry}
\usepackage[utf8]{inputenc}
\usepackage[T1]{fontenc}
\usepackage{graphicx}
\usepackage{xcolor}
\usepackage{mathrsfs}
\usepackage{amsthm}
\usepackage{amsmath}
\allowdisplaybreaks
\usepackage{amsfonts}
\usepackage{amssymb}
\usepackage{mathtools}
\usepackage{enumitem}
\usepackage[hidelinks]{hyperref}
\usepackage{tikz-cd}
\usepackage{tikz}
\usetikzlibrary{tqft}
\usetikzlibrary{positioning}
\usetikzlibrary{matrix,quotes,arrows.meta,bending}
\usetikzlibrary{decorations.markings}
\usetikzlibrary{calc}

\theoremstyle{plain}
\newtheorem{Thm}{Theorem}[section]
\newtheorem{Prop}[Thm]{Proposition}
\newtheorem{Lem}[Thm]{Lemma}

\theoremstyle{definition}
\newtheorem{Def}[Thm]{Definition}
\newtheorem{Rem}[Thm]{Remark}
\newtheorem{Exp}[Thm]{Example}

\newcommand{\QD}{\mathcal{QD}}
\newcommand{\R}{\mathbb{R}}
\newcommand{\G}{\mathcal{G}}
\renewcommand{\H}{\mathcal{H}}
\newcommand{\g}{\mathfrak{g}}
\newcommand{\h}{\mathfrak{h}}
\newcommand{\Lred}{L^{\mathrm{red}}}
\newcommand{\too}{\longrightarrow}
\newcommand{\mtoo}{\longmapsto}
\newcommand{\Ad}{\operatorname{Ad}}
\newcommand{\rk}{\operatorname{rk}}
\newcommand{\pr}{\mathrm{pr}}
\newcommand{\C}{\mathcal{C}}
\renewcommand{\S}{\mathcal{S}}
\renewcommand{\O}{\mathcal{O}}
\newcommand{\M}{\mathcal{M}}
\newcommand{\N}{\mathcal{N}}
\renewcommand{\t}{\mathfrak{t}}
\newcommand{\imp}{\text{imp}}
\renewcommand{\L}{\mathcal{L}}

\newcommand{\jmapb}{\jmath}

\newcommand{\Z}{\mathcal{Z}}
\newcommand{\Phib}{\widehat{\Phi}}
\newcommand{\X}{\mathcal{X}}
\newcommand{\D}{\mathcal{D}}
\newcommand{\gammab}{\widehat{\gamma}}
\newcommand{\A}{\mathcal{A}}
\newcommand{\Kt}{K_t}

\newcommand{\Nx}{N_x(\Kt)}
\newcommand{\expt}{\exp_t}

\newcommand{\Q}{\mathcal{Q}}
\newcommand{\tM}{\mathbb{T}M}
\newcommand{\tN}{\mathbb{T}N}
\newcommand{\KKS}{\text{KKS}}
\newcommand{\Ker}{\text{Ker}}
\newcommand{\K}{\mathcal{K}}
\newcommand{\prTs}{\mathrm{pr}_{T^\ast}}
\newcommand{\PP}{\mathcal{P}}

\newcommand{\Lie}{\operatorname{Lie}}
\newcommand{\ad}{\operatorname{ad}}

\title{Dirac geometry, deformation theory and shifted symplectic geometry}
\date{\today}
\author{Mohamed Moussadek Maiza}

\begin{document}

\maketitle

\begin{abstract}
We develop a Dirac deformation theory that interpolates between twisted Dirac geometry and Poisson geometry, and prove that this deformation is compatible with the principal structural operations of Dirac geometry: reduction along strong Dirac maps, integration to quasi-symplectic groupoids, and Morita equivalence of Lie algebroids.
The fundamental example is the deformation of the Cartan--Dirac structure $L_G$ on a compact Lie group~$G$
to the Kirillov--Kostant--Souriau Poisson structure on $\mathfrak{g}^*$, and its lift to a deformation
of the quasi-symplectic groupoid $D(G)\rightrightarrows G$ to the symplectic groupoid $T^*G\rightrightarrows\mathfrak{g}^*$.
As applications, we obtain a uniform deformation theory recovering, as special cases, the deformation of quasi-Hamiltonian to Hamiltonian reduction along conjugacy classes, the Steinberg and Sevostyanov slices to their additive (Kostant, Slodowy) counterparts, the multiplicative parabolic and unipotent reductions, the quasi-Hamiltonian implosion to symplectic implosion, and the
multiplicative Moore--Tachikawa varieties to their additive analogues. In the language of quasi-symplectic groupoid presentations of $1$-shifted symplectic stacks, our main reduction theorem yields a smooth deformation of the corresponding reduced spaces.
\end{abstract}

\section{Introduction}
\subsection*{Motivation.}
Quasi-Hamiltonian and Hamiltonian geometry are linked by an old and recurring analogy.
Conjugacy classes in a compact Lie group~$G$ behave like coadjoint orbits in $\mathfrak{g}^*$;
the multiplicative moment map $\mu\colon M\to G$ of Alekseev--Malkin--Meinrenken behaves
like an ordinary moment map; the double $D(G)$ behaves like the cotangent bundle $T^*G$.
This analogy is more than formal: in many concrete settings, the multiplicative object
\emph{deforms} smoothly to its additive counterpart. The deformation of quasi-Hamiltonian to Hamiltonian $G$-spaces of \cite{burelle2026deformations} and the recent deformation of Hamiltonian quasi-Poisson $G$-manifolds to Hamiltonian Poisson $G$-manifolds and the deformation of group valued implosion to symplectic implosion of \cite{maiza2026multiplicative} all instantiate the same principle: multiplicative geometric structure on a manifold $M$ can be rescaled and degenerated, in a controlled way, to the corresponding additive structure on a manifold $N$.\\

What has been missing is a single theory in which all of these deformations are special cases of one structure, and in which the standard operations of Poisson and quasi-Poisson
geometry -- reduction, groupoid integration, Morita equivalence -- are visibly preserved along the deformation. The natural setting is Dirac geometry, since twisted Dirac structures
specialize on one side to quasi-Hamiltonian moment maps via the Bursztyn--Crainic
construction \cite{bursztyn2005dirac,bursztyn2007dirac}, and on the other side to
Poisson manifolds. The strong Dirac maps \cite{bursztyn2005dirac,bualibanu2026reduction}
provide a unified reduction theory that already covers, in the non-deforming setting, all of the reduction examples one wants to handle.
 
The aim of this paper is to introduce a deformation theory for twisted Dirac structures compatible with this reduction theory, and to show that it is broad enough to recover all the principal multiplicative-to-additive deformations as instances of a single theorem.
 
\subsection*{Main definition.}
A \emph{Dirac deformation} (Definition \ref{Dirac def}) from a $\lambda_M$-twisted
Dirac manifold $(M,L_M,\lambda_M)$ to a Poisson manifold $(N,L_N)$ consists of a smooth
family
\[
\pi_{\mathcal{M}}\colon \mathcal{M}\longrightarrow\mathbb{R},\qquad
\pi_{\mathcal{X}}\colon \mathcal{X}\longrightarrow\mathbb{R}
\]
of smooth manifolds together with fiberwise twisted Dirac structures $\mathcal{L}\subset T\mathcal{M}$,
$\mathcal{K}\subset T\mathcal{X}$, fiberwise closed $3$-forms $\lambda_{\mathcal{M}}$, $\lambda_{\mathcal{X}}$,
and a smooth fiberwise strong Dirac map $\widehat{\phi}\colon\mathcal{M}\to\mathcal{X}$,
such that on the generic fibers $t\neq 0$ the data assembles into twisted Dirac manifolds
with a strong Dirac map between them, while at the central fiber $t=0$ the twisting vanishes
and one recovers Poisson manifolds with a Poisson map. We require $(\mathcal{M}_1,\mathcal{L}_1)\cong(M,L_M)$
and $(\mathcal{M}_0,\mathcal{L}_0)\cong(N,L_N)$.
 
The fundamental example is the deformation
\[
L_G \;\rightsquigarrow\; L_{\mathfrak{g}^*}
\]
of the Cartan--Dirac structure on $G$ to the Kirillov--Kostant--Souriau structure on $\mathfrak{g}^*$,
constructed in Proposition \ref{Dirac G} using the local coordinates of the deformation space $\mathcal{G}=D(G,\{1\})$ of \cite{burelle2026deformations}. The rescaled
Cartan $1$-form $\sigma(x)/t$ extends smoothly across $t=0$ to $x^\vee\in T^*\mathfrak{g}$,
and the rescaled Cartan $3$-form $\eta_G/t$ extends smoothly to zero. 
 
\subsection*{Main results.}
The paper has four principal theorems.
 
\smallskip
\noindent\textbf{(1) The Dirac theory recovers Poisson deformation of quasi-Hamiltonian spaces.}
The Poisson deformation of Hamiltonian quasi-Poisson manifolds in the sense of \cite{maiza2026multiplicative}
gives rise to a Dirac deformation in our sense (Theorem \ref{embed thm}). The construction
is the natural one: the Bursztyn--Crainic Dirac structure associated to a quasi-Hamiltonian space
extends smoothly across the central fiber. Examples \ref{ex:UncontrolledKernel}
and \ref{ex:CP2} show that the converse fails: there exist Dirac deformations that do not arise
from any Hamiltonian quasi-Poisson manifold, so the Dirac deformation theory is strictly more general.
 
\smallskip
\noindent\textbf{(2) Compatibility with reduction along strong Dirac maps.}
Theorem~\ref{thm:main} shows that under clean-intersection and rank hypotheses on the fiberwise stabilizer subalgebroids, the reduced spaces $Q_t=\widehat{\phi}_t^{-1}(S_t)/A_{S_t,\widehat{\gamma}_t,\varphi_t}$
of B\u{a}libanu--Mayrand~\cite[Theorem 2.21]{bualibanu2026reduction} assemble into a smooth Dirac deformation.
Specializing the reduction level produces, in a unified way:
\begin{itemize}
\item the deformation of quasi-Poisson reduction at a point to Marsden--Weinstein reduction,
\item the deformation of reduction along conjugacy classes to reduction along coadjoint orbits,
\item the deformation of the Steinberg cross-section reduction to the Kostant slice reduction,
\item the deformation of multiplicative parabolic / unipotent reductions to their additive analogues,
\item the deformation of quasi-Hamiltonian implosion to symplectic implosion (recovering \cite[Theorem 5.3]{maiza2026multiplicative} from a Dirac-geometric perspective),
\item the deformation of Sevostyanov-slice / Whittaker reduction to Slodowy-slice reduction,
\item the deformation of multiplicative Moore--Tachikawa varieties to the additive
Moore--Tachikawa varieties of~\cite{crooks2024moore}.
\end{itemize}
Each of these is a single example in Section \ref{F example}, recovered as the application
of Theorem \ref{thm:main} to a specific deformation of reduction levels.
 
\smallskip
\noindent\textbf{(3) Integration and Morita equivalence.}
Under a uniform integrability hypothesis on the fiberwise monodromy groups, Theorem \ref{groupoid integra} shows that a smooth family of twisted Dirac manifolds integrates to a smooth family of quasi-symplectic groupoids, with the central fiber being
the symplectic groupoid integrating the Poisson limit. The fundamental example is the
deformation of $D(G)\rightrightarrows G$ to $T^*G\rightrightarrows\mathfrak{g}^*$
(Proposition \ref{prop:fund}). Building on the dual-pair formalism of Xu~\cite{xu2004momentum},
Proposition~\ref{prop:morita-LA} establishes that infinitesimal Morita equivalence of Lie algebroids
is preserved across $t=0$ when witnessed by a smooth family of dual pairs satisfying transversality
and rank-constancy hypotheses.
 
\smallskip
\noindent\textbf{(4) Reduced spaces and shifted-symplectic presentations.}
Section \ref{Section 7} adopts the language of $1$-shifted symplectic stacks
\cite{pantev2013shifted, calaque2021derived} to organize the reduction theorem in its natural conceptual frame.
We work entirely with the presentation model: a smooth family of quasi-symplectic groupoids
$\mathcal{G}\rightrightarrows\mathcal{X}$ \emph{presents} a smooth family of $1$-shifted symplectic
stacks $[\mathcal{X}_t/\mathcal{G}_t]$, and a smooth family of Hamiltonian $\mathcal{G}_t$-spaces
presents a smooth family of Lagrangian morphisms. We do not develop derived algebraic geometry
over $\mathbb{R}$ in the sense of \cite{pantev2013shifted, calaque2021derived}; that is beyond the scope of this paper.
The main statement, Theorem \ref{thm:shifted}, is that under clean-intersection and local-freeness hypotheses the fiber product of Lagrangian morphisms produces a smooth family
of $0$-shifted symplectic manifolds, deforming a smooth family of symplectic manifolds.
 
\subsection*{Relation to existing work.}
The deformation $D(G)\rightsquigarrow T^*G$ at the level of groupoids and forms is implicit
in \cite{burelle2026deformations}; what is new here is the recognition that it is the integration
of a deformation of Dirac structures, and that this perspective extends to all the reduction
examples above. The compatibility of strong-Dirac reduction with deformation appears to be new
even for the most studied examples, such as Steinberg-to-Kostant. The Morita compatibility
result builds on Xu \cite{xu2004momentum} but is, to our knowledge, the first deformation-theoretic statement in that theory.
 
The scheme of Section \ref{Section 7} is closely related in spirit to the shifted-symplectic perspective on Moore--Tachikawa varieties developed in \cite{crooks2024moore}, but specialized to the deformation context. We do not claim a derived statement; the content is a classical-shadow theorem about smooth families of
reduced spaces, expressed in the stack setting because that language renders the Lagrangian-correspondence structure transparent.
 
\subsection*{Organization.}
Section~\ref{section 1} recalls the necessary background on Lie algebroids, twisted
Dirac structures, strong Dirac maps, and reduction, following~\cite{bursztyn2005dirac,bursztyn2007dirac, bualibanu2026reduction}.
Section \ref{section 3} introduces the Dirac deformation theory, proves the
fundamental example $L_G\rightsquigarrow L_{\mathfrak{g}^*}$, and shows that the theory
strictly generalizes the Poisson deformation theory of \cite{maiza2026multiplicative}.
Section \ref{section 4} establishes compatibility with strong-Dirac reduction and applies it
to the principal examples in Section \ref{F example}.
Section \ref{section 6} develops the integration to quasi-symplectic groupoids and the
Morita compatibility. Section \ref{Section 7} reformulates the main reduction theorem in the language of presented $1$-shifted symplectic stacks.
 
\subsection*{Conventions.}
Throughout, $G$ denotes a (typically compact, occasionally complex reductive) Lie group with Lie
algebra $\mathfrak{g}$ equipped with an Ad-invariant scalar product $\langle\,,\,\rangle_{\mathfrak{g}}$.
The corresponding Cartan $3$-form on $G$ is denoted $\eta_G\in\Omega^3(G)$, and the corresponding
trivector on $\mathfrak{g}$ is $\chi\in\bigwedge^3\mathfrak{g}$. We write $\mathbb{T}M=TM\oplus T^*M$
for the generalized tangent bundle, with the standard pairing $\langle(v,\alpha),(w,\beta)\rangle=\alpha(w)+\beta(v)$.
The deformation parameter $t\in\mathbb{R}$ has $t=1$ for the multiplicative
endpoint and $t=0$ for the additive endpoint.

\subsection*{Acknowledgments} We thank Jean-Philippe Burelle, Maxence Mayrand and Yassine Ait-Mohamed for useful discussions. The corresponding author thanks the Institut des sciences mathématiques ISM and the Ernest-Monga scholarship for their financial support.

\section{Preliminaries} \label{section 1}
In this section, we recall the basic facts about Dirac geometry with some fundamental examples. The material of this part is based on \cite{bursztyn2005dirac,bursztyn2007dirac}.

\subsection{Lie algebroids.} Let $M$ be a smooth manifold. a Lie algebroid over $M$ consists of a vector bundle $ \pi: \A \too M$ together with a bundle map $\rho: \A \too TM$ called \textbf{anchor map} and a Lie algebra bracket on the space of sections $\Gamma(\A)$ of $\A$ satisfying the Leibniz rule \begin{equation}
    [\sigma,f\eta] = f[\sigma,\eta] + \L_{\rho(\sigma)}(f)\eta, \quad \forall \sigma,\eta \in \Gamma(\A), f \in C^{\infty}(M).
\end{equation}

The basic example is when $M$ is equipped with an infinitesimal action of a lie algebra $\g$: if $\rho_{\g}: \g \too \Gamma(TM)$ denote this infinitesimal action, then $\A$ is the trivial bundle $M\times \g$, such that the anchor map $\rho$ is given by \[ \rho: M\times \g \too TM, \quad (m,x) \mtoo \rho_{\g}(x)(m)\] and the bracket $[\cdot,\cdot]$ on the set of sections of $\Gamma(\A) = C^{\infty}(M,\g)$ is \begin{equation}
    [u,v](p) = [u(p),v(p)]_{\g} + d_{p}v(\rho_{\g}(u(p))) - d_{p}u(\rho_{\g}(v(p))), \quad \forall u,v \in \Gamma(\A), p\in M.
\end{equation}

We denote this lie algebroid by $M \rtimes \g.$\\ 

Another example comes from Poisson geometry: if $(M,\pi)$ is a Poisson manifold, then the cotangent bundle $T^{*}M$ is a Lie algebroid with anchor map $\pi^{\sharp}: T^{*}M \too TM, \quad \alpha \mtoo \pi^{\sharp}(\alpha)$ and the bracket on the sections of $T^{*}M$ is given by \begin{equation}
[\alpha,\beta] = \L_{\pi^{\sharp}(\alpha)}\beta - \L_{\pi^{\sharp}(\beta)}\alpha - d\pi(\alpha,\beta),    
\end{equation}

uniquely characterized by $[df,dg] = d\{f,g\}$ and the Leibniz rule. Here $\{f,g\} = \pi(df,dg)$ is the Poisson bracket.\\ The last example found its origins in quasi-Poisson geometry: Let $G$ be a compact Lie group and $\g$ the associated Lie algebra equipped with an $\Ad$-invariant scalar product $\langle \cdot , \cdot \rangle_{\g}.$ Let $\eta_G \in \Omega^{3}(G)$ be the invariant Cartan 3-form on $G$ and $\chi \in \bigwedge^{3} \g$ the corresponding tri-vector on $\g.$ If an orthonormal basis $(e_a)$ is fixed, then for any $u,v,w \in \g,$ the expression of $\chi$ is \begin{equation}
    \chi = \dfrac{1}{12}\sum f_{abc} e_{a}\wedge e_{b}\wedge e_{c}, \quad f_{abc} = \langle e_{a}, [e_b,e_c]\rangle_{\g}.
\end{equation}

Recall that a $G$-quasi-Poisson manifold consists of a manifold $M$ together with a $G$-invariant bi-vector $\pi$ such that \[ [\pi,\pi] = \rho_{M}(\eta_G),
\]

where $\rho_{M}: \g \too \Gamma(TM)$ is the infinitesimal action of $\g$ and $[,]$ represents the Schouten brackets. There exists also a notion of  quasi-Poisson $\g$-manifold which is related to the original one by a procedure of integration on infinitesimal actions. Now, consider a quasi-Poisson $\g$-manifold $(M,\pi)$. We define a vector bundle $\A$ over $M$ as follows: Take $\A = T^{*}M \oplus \g$ with the map $\rho: T^{*}M \oplus \g \too TM, \quad (\alpha,x) = \pi^{\sharp}(\alpha) + \rho_{M}(v)$ and the bracket $[,]$ such that: \begin{enumerate}
    \item $[(\alpha,0),(\beta,0)] = ([\alpha,\beta], \chi(\rho^{*}_{M}(\alpha \wedge \beta),\cdot)).$
    \item $[(0,v),(0,w)] = (0,[v,w]_{\g}),$
    \item $[(0,v),(\alpha,0)] = (\L_{\rho_{M}(v)}\alpha, 0),$
\end{enumerate}
for all $\alpha,\beta \in \Omega^{1}(M)$ and $v,w \in \g.$ Then, by \cite[\text{Theorem 2.5}]{bursztyn2005dirac}, $(T^{*}M \oplus \g,\rho,[\cdot,\cdot])$ is a Lie algebroid.

\subsection{Dirac manifolds} For any smooth manifold $M$, we denote \textbf{the generalized tangent bundle} over $M$ by $\mathbb{T}M = TM \oplus T^{*}M$. Moreover, let $\langle,\rangle$ be the pairing \[
\langle(v,\alpha),(w,\beta)\rangle = \alpha(w) + \beta(v), \quad \forall (v,\alpha),(w,\beta) \in \tM,
 \] and let $\lambda \in \Omega^{3}(M)$ be a closed 3-form on $M.$ Finally, consider the $\lambda$-twisted bracket \[
 [(v,\alpha),(w,\beta)]_{\lambda} = ([v,w], \L_{v}\beta - i_{w}d\alpha + i_{v\wedge w}\lambda).  
 \]

\begin{Def}(\cite{bualibanu2026reduction})
  A $\lambda$-\textbf{twisted Dirac structure} on $M$ consists of a vector sub-bundle $L$ of $\tM$ such that \begin{enumerate}
    \item $L$ is Lagrangian with respect to the pairing $\langle, \rangle.$
    \item The set of sections $\Gamma(L)$ of $L$ is invariant under the twisted bracket $[,]_{\lambda}.$
\end{enumerate}
\end{Def}

$L$ is in fact a Lie algebroid: The restriction of $[,]_{\lambda}$ to $\Gamma(L)$ is a Lie bracket and the anchor map $\rho: L \too TM$ is simply the projection $(v,\alpha) \mtoo v.$\\ 

The basic examples of twisted Dirac structures come from $2$-forms and bi-vectors: Let $\omega \in \Omega^{2}(M)$ be a smooth two-form. Then, we define the $(-d\omega)$-twisted Dirac structure $L_{\omega}$ as follows: \[
L_{\omega} \coloneqq \{(v,\omega^{\flat}(v)), v \in TM\}.
\]  

Here, $\omega^{\flat}: TM \too T^{*}M$ is the map $v \too i_{v}\omega.$ Moreover, a $\lambda$-twisted Dirac structure is induced by a two-form if and only if $L \cap T^{*}M = 0.$ In general, the image $\rho(L)$ of the twisted Dirac structure is an integrable generalized distribution. This defines a foliation of $M$ where each leaf $\mathcal{O}$ is equipped with a \textbf{pre-symplectic} form $\omega_{\O}$ such that $d\omega_{\O} = -\lambda_{|\O}.$\\

The second standard example is twisted Dirac structures induced from bi-vectors: for any bi-vector $\pi$ on $M,$ we define the Lagrangian sub-bundle $L_{\pi} \coloneqq  \{(v,\pi^{\sharp}(v)), v\in TM\}.$ Then, $L_{\pi}$ is $\lambda$-twisted Dirac if and only if $[\pi,\pi] = \pi^{\sharp}(\lambda).$ In this case, $L_{\pi}$ is called a \textbf{twisted Poisson Dirac structure.} Conversely, a twisted Dirac structure is induced by a bi-vector $\pi$ if and only if the kernel of $L$ defined by $\text{Ker}(L) = L\cap TM$ is reduced to $0.$ As a consequence, there exists a correspondence between Poisson bi-vectors and non-twisted Dirac structures with trivial kernels. In this case, the induced foliation on $M$ is the one whose non-degenerate leaves are symplectic.\\

The third example comes from a quasi-Poisson perspective: the special case of Poisson brackets in the construction of non-twisted Dirac structures extends also to quasi-Poisson settings: Suppose that $(M,P,\mu,G)$ is a Hamiltonian $G$-quasi-Poisson manifold \cite[\text{Definitions 2.1,2.2}]{alekseev2002quasi}. Let $\sigma: \g \too T^{*}G$ be the map $x\mtoo \dfrac{1}{2}(x^{L} + x^{R})^{\vee}$ when $x^{\vee}$ is the dual of $x\in TG$ under the identification $T^{*}G \cong TG$ using the scalar product $\langle,  \rangle_{\g}$ and \[
C = \mathrm{Id} - \dfrac{1}{2}e_{aM}\otimes \mu^{*}(\theta^{L}_{a} - \theta^{R}_{a}).
\]

Here, $B = (e_{a})$ is a fixed orthonormal basis of $\g$ and $\theta^{L}_{a}$ (respectively $\theta^{R}_{a}$) are the coordinates of the left-invariant Maurer-Cartan form (resp. the coordinates of the right-invariant Maurer-Cartan form) in $B$. Then, the quasi-Poisson manifold $M$ induces a $-\mu^{*}\phi$ twisted Dirac structure \cite[\text{Theorem 3.16}]{bursztyn2007dirac} defined by \[ \label{Dirac qp}
L = \{(\pi^{\sharp}(\alpha) + x_{M}, C^{*}(\alpha) + \mu^{*}\sigma(x)); \alpha \in T^{*}M, x\in \g\}.
\]

The anchor map $\rho: L\too TM$ is the image of the map \begin{align*}
    T^{*}M \oplus \g \too TM, \quad (\alpha,x) \mtoo \pi^{\sharp}(\alpha) + x_{M}.
\end{align*}

The leaves of the induced foliation on $M$ are quasi-Hamiltonian: for any non-degenerate leaf $\O$, there exists a two-form $\omega_{\O}$ such that $d\omega_{\O} = \mu^{*}\phi_{|\O}.$

The first example that appears naturally in the quasi-Poisson category is the group $G$ itself considered as a Hamiltonian $G$-quasi-Poisson manifold with respect to the conjugation action and the identity $\mathrm{Id}: G\too G$ as a moment map. The corresponding Dirac-Cartan structure is the following sub-bundle $L_{G} = \{(x^{L} - x^{R}, \sigma(x)); x\in \g\}.$ In particular, the leaves of the induced foliation on $G$ are conjugacy classes and there exists an isomorphism $L_{G}\cong G \rtimes \g$ \cite[\text{Example 3.4}]{bursztyn2005dirac}.

\subsection{f-Dirac and b-Dirac maps} Let $(M,L_{M},\lambda),(N,L_{N},\eta)$ be two twisted Dirac manifolds and $f: M \too N$ a smooth map. Then, the pullback $f^{*}L_{N}$ is the Lagrangian sub-bundle \[
f^{*}L_{N} = \{(v,f^{*}\alpha) \in \tM; (df(v),\alpha)\in L_{N}\}.
\] 
Then, if $f^{*}L_{N}$ is of a constant rank, $f^{*}L_{N}$ is a $f^{*}\eta$-twisted Dirac structure. Moreover, $f$ is said to be a \textbf{b-Dirac} if $L_{M} = f^{*}L_{N}$ (\cite[Subsection 1.3]{bualibanu2026reduction})\\

The push-forward $f_{*}L_{M}$ is the Lagrangian sub-bundle \[
f_{*}L_{M,p} = \{(df(v),\alpha) \in \tN_{f(p)}, (v,f^{*}\alpha) \in L_{M,p}.\}
\]

This distribution is not well defined unless we impose the condition  $f_{*}L_{M,p} = f_{*}L_{M,q}$ such that $f(p) = f(q).$ In this case, the map $f$ is \textbf{f-Dirac} if $f_{*}L_{M} = L_{N}.$ This notion generalizes the push-forward of vector fields. When $f$ is a diffeomorphism, $f$ is $b$-Dirac if and only it is $f$-Dirac, called a Dirac diffeomorphism.\\

\subsection{Strong Dirac maps} The notion of Strong Dirac maps \cite{bualibanu2026reduction}, or Dirac realizations \cite{bursztyn2005dirac} is a generalization of Poisson moment map: Consider a Poisson manifold $(M,\pi)$ and $\pi_{\text{KKS}}$ the linear Poisson structure on the dual of the Lie algebra $\g^{*}.$ Any Poisson map $\mu: (M,\pi_{M})\too (\g^{*},\pi_\KKS)$ induces a Hamiltonian infinitesimal action $\rho_{M}:\g \too TM,\quad x\mtoo \pi_{M}^{\sharp}(\mu^{*}(x)).$ Then, by The Marsden-Weinstein reduction theorem \cite{marsden1974reduction},  if $x$ is a regular value of $\mu,$ the quotient $\mu^{-1}(x)/\g_{x}$ inherits a Poisson structure. Now, if we replace $(M,\pi_M)$ and  $(\g^{*},\pi_{\KKS})$ by Dirac manifolds $(M,L_M)$ and $(N,L_N)$ respectively, a $f$-map $\mu: (M,L_M)\too (N,L_N)$ is called a \textbf{strong Dirac maps} \cite{bualibanu2026reduction} if \[
\mu^{*}\lambda_{N} = \lambda_{M} \text{ and } \text{Ker}(d\mu) \cap \text{Ker}(L_{M}) = 0.
\] 

In other terms, $\mu$ is a strong Dirac map if and only if $\mu$ induces an isomorphism between $\text{Ker}(L_M)$ and $\text{Ker}(L_{N}).$ The last condition means that for any $(w,\alpha)$ in $TN,$ there exists a unique $v \in TM$ such that $w = d\mu(v)$ and $(v,\mu^{*}\alpha)\in L_{M}.$\\

We define a Lie algebroid action of $L_N$ on $M$ as follow: \[
\rho_{M}: L_{N} \times_{N} M \too TM, \quad (w,\alpha)\mtoo v.
\]

Thus, if $x$ is a regular value of the map $\mu,$ the induced Dirac structure on $\mu^{-1}(x)$ descends to a Poisson structure on the reduced space $\mu^{-1}(x)/L_{N,x}.$ 

Finally, for a $\lambda$-twisted Dirac structure $(M,L_M)$ and $\gamma \in \Omega^{2}(M),$ the gauge transformation of $L_M$ by $\gamma$ is the $(\lambda - d\gamma)$-twisted Dirac structure $\tau_{\gamma}L_M = \{(v,\alpha + \gamma^{\flat}(v)), (v,\alpha) \in L_{M}\}$ and we say that $L_{M}$ and $\tau_{\gamma}L_M$ are gauge equivalent. For any smooth map $f: M\too N$ between two Dirac manifolds $(M,L_M)$ and $(N,L_N)$ and $\gamma \in \Omega^{2}(N),$ we have \cite{bualibanu2026reduction} \begin{equation}
    f^{*}\tau_{\gamma}L_{N} = \tau_{f^{*}\gamma}f^{*}L_{N} \text{ and } \tau_{\gamma}f_{*}L_M = f_{*}\tau_{f^{*}\gamma}L_M.
\end{equation}
\\

\paragraph{\textbf{Quotient of Dirac structures.}}
Consider a $\lambda$-twisted Dirac structure $(M,L_M).$ Given a vector field $X$, we say that $X$ is \textbf{Dirac} if its flow is a local Dirac diffeomorphism. This happens if and only if the set of the sections $\Gamma(L_M)$ is preserved by the the Lie derivative $\L_{Z}$ and $\L_{Z}\eta_M$ is equal to zero on each non-degenerate leaf. Given a distribution $D\subset TM,$ $D$ is said to be \textbf{Dirac} if every vector field on $D$ extends to a local Dirac vector field: for every $p\in M$ and $v\in T_{p}M,$ there exists a vector field $Z$ defined near $p$ such that \begin{itemize}
    \item $Z_p = v.$
    \item $Z$ is tangent to $D$(i.e., $Z_q \in D_q$ for any $q$ in the domain).
    \item $D$ is a Dirac vector field.
\end{itemize}

We want to construct  a Dirac structure on the reduced space (called also the leaf space) $M/D.$ In general, this quotient is not well defined and does not admit a smooth structure. However, if we suppose that $D$ is an involutive regular Dirac distribution contained in the kernel $\text{Ker}(\eta_M)$ and the quotient $ Q = M/D$ is smooth, Proposition \cite[\text{Proposition 1.12}]{bualibanu2026reduction} states that the Dirac structure on $M$ descends to a Dirac structure on the leaf space. In other terms, $(Q,\pi_{*}L_{M})$ is a $\eta_{Q}$-twisted Dirac manifold where $\pi: M\too Q$ is the canonical projection and $\lambda_{Q}$ is the 3-form determined by the condition $\pi^{*}\lambda_Q = \lambda_M.$

\section{Deformation of twisted-Dirac structures} \label{section 3}

In this section, we treat the general theory of deformation of twisted Dirac structures to Poisson structures. Under specific conditions, it generalizes the deformation theories studied in the case of Hamiltonian quasi-Poisson manifolds and the imploded sections of Hamiltonian and quasi-Hamiltonian manifolds \cite{maiza2026multiplicative}.\\ During this section, we denote by
\begin{enumerate}
\item $(M,L_M,\lambda_M),(X,L_X,\lambda_X)$ are twisted Dirac manifolds.
\item $(N,L_{\pi_{N}}),(Y,L_{\pi_{Y}})$ are Poisson manifolds.
\item $\phi: M \too X$ is a strong Dirac map.
\item $\varphi: N \too Y$ is a Poisson map.
\end{enumerate}

\begin{Def} \label{Dirac def}
\textbf{A Dirac deformation} from the $\lambda_{M}$-twisted Dirac $M$ to the Poisson manifold $N$ is the data of smooth manifolds $\M,\X$, each equipped with a smooth surjective submersion $\pi_\M: \M \too \R$ (respectively $\pi_\X: \X \too \R$); smooth Lie algebroids $\L \subset \mathbb{T}\M$ and $\K \subset \mathbb{T}\X$; a smooth section $\lambda_{\M}$ of $\bigwedge^3 \Ker(d\pi_\M)^*$ closed fiberwise (resp. a smooth section $\lambda_{\X}$ of $\bigwedge^{3} \Ker(d\pi_\X)^*$) and a  smooth map $\widehat{\phi}: \M \too \X$ such that for each $t\in \R$, the restrictions $\M_t \coloneqq \M_{|\pi_{\M}^{-1}(t)}, \X_t \coloneqq \X_{|\pi_{\X}^{-1}(t)},\lambda_{\M,t} \coloneqq (\lambda_{\M})_{|\pi_{\M}^{-1}(t)}, \lambda_{\X,t} \coloneqq (\lambda_{\X})_{|\pi_{\X}^{-1}(t)}, \L_t \coloneqq \L_{|\pi_{\M}^{-1}(t)}$ and $\K_t \coloneqq \K_{|\pi_{\X}^{-1}(t)}$ 
verify the following  \begin{enumerate}
    \item $(\M_t,\L_t,\lambda_{\M,t})$ is a $\lambda_{\M,t}$-twisted Dirac manifold (respectively $(\X_t,\K_t,\lambda_{\K,t})$ a $\lambda_{\X,t}$-twisted Dirac manifold), together with a strong Dirac map $\widehat{\phi}_t: \M_t \too \X_t$ for every $t\neq 0$ and,
    \item $(\M_0,\L_0)$ is a Poisson manifold (resp. $(\X_0,\K_0)$) and a Poisson map $\widehat{\phi}_0: \M_0 \too \X_0$ such that \begin{itemize}
        \item $(\M_1,\L_1) \cong (M,L_M)$ as twisted Dirac manifolds.
        \item $(\M_0,\L_0) \cong (N,L_N)$ as Poisson manifolds.
    \end{itemize}    
\end{enumerate}
\end{Def}

For any compact Lie group $G$, let $ \G = \D(G,\{1\})$ be the deformation space studied in \cite{burelle2026deformations} and the local coordinates of the central fiber \cite[\text{Lemma 2.1}]{burelle2026deformations} given by the map \[
\psi : \g \times \R \too \G, \quad
(x, t) \mtoo
\begin{cases}
(\exp(tx), t) & \text{if } t \ne 0 \\
(x, 0) & \text{if } t = 0
\end{cases}
\]

We consider the fiberwise distribution $\L  \subset \mathbb{T}\G$ defined by \begin{align*}
    \L_{t} &\coloneqq \{((x^{L} - x^{R}), \dfrac{\sigma(x)}{t}), x\in \g\}, \quad t \neq 0,\\
    \L_{0} &\coloneqq L_{\g^{*}}, \quad t = 0.
\end{align*}  

Also, let $\widehat{\eta}_\G$ defined on $G\times \R^{*}$ \begin{align*}
    \eta_{\G,t} = \dfrac{\eta_G}{t}, \quad t\neq 0.
\end{align*}

\begin{Prop} \label{Dirac G}
$\L$ is smooth fiberwise Lagrangian sub-bundle of $\mathbb{T}\G,$ and $\eta_{\G}$ extends to a smooth three-form on $\G$ such that $\eta_{\G,0} = 0.$
\end{Prop}

The following technical lemmas are needed for the proof of the proposition. For any $y\in \g$ and $t$ sufficiently small, we denote by $A = \mathrm{ad}_y$ and $\psi_{t}(y) = \exp(ty).$ Finally, let $U$ be the open set on which the $\psi_{|U}$ is a diffeomorphism

\begin{Lem} \label{vect}
For any $t$ sufficiently small (including 0) and for all $x\in \g$, we have $(d_{y}\psi_{t})^{-1}(x^{L}-x^R) = [y,x].$  
\end{Lem}

\begin{proof}
The differential of the exponential map gives \[ d_{y}\psi_{t}(v) = d_{y}\exp(tv) = dL_{\exp(ty)} \circ \dfrac{1- e^{-tA}}{tA}(tv) = tdL_{\exp(ty)}\circ \dfrac{1 - e^{-tA}}{tA}(v),\] yielding \begin{equation} \label{eq 1}
    (d_{y}\psi_{t})^{-1}(w) = \dfrac{1}{t} (\dfrac{1 - e^{-tA}}{tA})^{-1} \circ dL_{\exp(-ty)}(w).
\end{equation}
Using Left trivialization, we get $(x^{L} - x^{R})(g) = dL_{g}(x - \Ad_{g^{-1}}x).$ At $g = \exp(ty):$ \[
(x^{L} - x^{R})(g) = dL_{\exp(ty)}((1 - e^{-tA})(x)). 
\]

Substituting this in \ref{eq 1}, one gets \[
(d_{y}\psi_{t})^{-1}(x^L - x^R) = \dfrac{1}{t} (\dfrac{1 - e^{-tA}}{tA})^{-1} (1 - e^{-tA})(x) = \dfrac{1}{t}tA(x) = A(x) = [y,x],
\]
since $(\dfrac{1 - e^{-tA}}{tA})^{-1}(1 - e^{-tA}) = tA$ is an identity in $\text{End}(\g).$ As a result, $[y,x]$ is independent of $t$ and polynomial in $y$; it is smooth on $U.$
\end{proof}

\begin{Lem}[1-form component]\label{lem:cov}
Define, for each $x \in \g$,
\[
\alpha_x(y,t) := \frac{1}{t}\,\psi_t^*\sigma(x)\big|_y \in T_y^*\g, \qquad t \neq 0.
\]
Then $\alpha_x$ extends smoothly to $t = 0$, with
\[
\alpha_x(y,t)(v) = \left\langle x,\, \frac{\sinh(tA)}{tA}(v)\right\rangle_\g
\]
for all $(y,t) \in \g \times \R$ and $v \in T_y\g \cong \g$. In particular, $\alpha_x(y,0)(v) = \langle x, v\rangle_\g = x^\vee(v)$.
\end{Lem}
 
\begin{proof}
At $g = \exp(ty)$, the 1-form $\sigma(x)$ evaluated on a tangent vector $w \in T_gG$ is (using left-trivialization)
\[
\sigma(x)_g(w) = \tfrac{1}{2}\big\langle x + \Ad_{g^{-1}} x,\; dL_{g^{-1}} w \big\rangle_\g = \tfrac{1}{2}\big\langle (1 + e^{-tA})x,\; dL_{\exp(-ty)} w \big\rangle_\g.
\]
Pulling back via $\psi_t$ with $d\psi_t|_y(v) = t\, dL_{\exp(ty)} \circ \frac{1-e^{-tA}}{tA}(v)$:
\[
\psi_t^*\sigma(x)\big|_y(v) = \tfrac{1}{2}\Big\langle (1 + e^{-tA})x,\;\frac{1-e^{-tA}}{A}(v)\Big\rangle_\g,
\]
since $dL_{\exp(-ty)} \circ dL_{\exp(ty)} = \mathrm{Id}$ and $t \cdot \frac{1-e^{-tA}}{tA} = \frac{1-e^{-tA}}{A}$.
 
Using the $\Ad$-invariance of $\langle\cdot,\cdot\rangle_\g$ and the skewness of $\ad_y$, namely $\langle e^{-tA}x, w\rangle = \langle x, e^{tA}w\rangle$:
\begin{align*}
\psi_t^*\sigma(x)\big|_y(v)
&= \tfrac{1}{2}\Big\langle x,\; (1 + e^{tA})\frac{1-e^{-tA}}{A}(v)\Big\rangle_\g.
\end{align*}
Now observe the identity
\[
(1 + e^{tA})(1 - e^{-tA}) = e^{tA} - e^{-tA} = 2\sinh(tA),
\]
so that
\[
\psi_t^*\sigma(x)\big|_y(v) = \Big\langle x,\;\frac{\sinh(tA)}{A}(v)\Big\rangle_\g.
\]
Dividing by $t$:
\[
\alpha_x(y,t)(v) = \frac{1}{t}\,\psi_t^*\sigma(x)\big|_y(v) = \Big\langle x,\;\frac{\sinh(tA)}{tA}(v)\Big\rangle_\g.
\]
The function $z \mapsto \frac{\sinh(z)}{z} = \sum_{k=0}^\infty \frac{z^{2k}}{(2k+1)!}$ is entire, so $\frac{\sinh(tA)}{tA} = \sum_{k=0}^\infty \frac{t^{2k} A^{2k}}{(2k+1)!}$ is a smooth (in fact analytic) function of $(y,t)$ with values in $\operatorname{End}(\g)$. At $t=0$:
\[
\frac{\sinh(tA)}{tA}\Big|_{t=0} = \mathrm{Id}_\g,
\]
and therefore $\alpha_x(y,0)(v) = \langle x, v\rangle_\g = x^\vee(v)$.
\end{proof}
 
To prove that $\L$ is smooth, for each $x\in \g$, define the section of the fiber-wise generalized tangent bundle $\mathbb{T}_{\mathrm{vert}}\mathcal{G}|_{(y,t)}= T_y\g \oplus T_y^*\g$:
\begin{equation}\label{eq:spanning}
s_x(y,t) := \Big([y,x],\;\Big\langle x, \frac{\sinh(t\,\ad_y)}{t\,\ad_y}(\cdot)\Big\rangle_\g\Big) \in T_y\g \oplus T_y^*\g.
\end{equation}

\begin{Lem} \label{sections s}
The sections $s_x$ are smooth on $g\times \R.$ Moreover, $\{s_x\}_{x \in \g}$ are everywhere linearly independent as sections of $\mathbb{T}_{\mathrm{vert}}\mathcal{G}$, ensuring constant rank $n = \dim \g$. 
\end{Lem}

\medskip

\begin{proof}

By Lemmas \ref{vect} and \ref{lem:cov}, $s_x$ is smooth on $\g \times \R.$
At $t \neq 0,$ the sections $\{s_x\}_{x \in \g}$ are exactly the pullbacks of $\{(x^L - x^R, \sigma(x)/t)\}_{x \in \g}$ via $\psi_t$, which span the (rescaled) Cartan--Dirac structure $\mathcal{L}_t$ on $\mathcal{G}_t$. For $t = 0,$ we have $s_x(y,0) = ([y,x], x^\vee)$. As $x$ ranges over $\g$, these sections span precisely the Poisson--Dirac structure $L_{\g^*} = \{(\pi_{\mathrm{KKS}}^\sharp(\beta),\beta) : \beta \in T^*\g\}$, since $\pi_{\mathrm{KKS}}^\sharp(x^\vee)\big|_y = [y,x]$ under the identification $\g \cong \g^*$ via $\langle\cdot,\cdot\rangle_\g$.\\

It remains to verify that $\{s_x\}_{x \in \g}$ are everywhere linearly independent as sections of $\mathbb{T}_{\mathrm{vert}}\mathcal{G}$, ensuring constant rank $n = \dim \g$. In other words, for every $(y,t) \in \g \times \R$, the map $x \mapsto s_x(y,t)$ is injective. Indeed, suppose $s_x(y,t) = 0.$ We have $[y,x] = 0$ i.e., $x \in \ker(A) = \mathfrak{z}_{\g}(y)$ the centralizer of $y\in \g$ and $\langle x, \frac{\sinh(t\ad_y)}{t\ad_y}(v)\rangle = 0$ for all $v \in \g$. Write $\g = \Ker(A)\oplus \Ker(A)^{\perp}.$ This decomposition is preserved by $\frac{\sinh(t\ad_y)}{t\ad_y}(v)$ since $A = \ad_y$ does. Restricted to $\Ker(A),$ the operator $\frac{\sinh(t\ad_y)}{t\ad_y}(v)$ acts exactly as the identity (since $\frac{sinh(z)}{z}$ goes to 1 as $z$ goes to $0$). Then $\langle x,v_{\Ker(A)} \rangle_{\g} = 0$ for any $v.$ In particular, for $v_{\Ker(A)} = x,$ one gets $\langle x,x \rangle_{\g} = 0.$ Thus $x = 0.$\\

We also verify the Lagrangian condition. The natural pairing on $T_y\g \oplus T_y^*\g$ is $\langle(v,\alpha),(w,\beta)\rangle = \alpha(w) + \beta(v)$. For $x, x' \in \g$:
\begin{align*}
\langle s_x, s_{x'}\rangle
&= \Big\langle x, \frac{\sinh(tA)}{tA}([y,x'])\Big\rangle + \Big\langle x', \frac{\sinh(tA)}{tA}([y,x])\Big\rangle.
\end{align*}
Using that $\frac{\sinh(tA)}{tA} \circ A = \frac{\sinh(tA)}{t}$ and that $\langle x, \sinh(tA)(x')\rangle = -\langle \sinh(tA)(x), x'\rangle$ (since $A = \ad_y$ is skew-symmetric and $\sinh$ is odd), we get
\begin{align*}
\langle s_x, s_{x'}\rangle
&= \frac{1}{t}\Big[\langle x, \sinh(tA)(x')\rangle + \langle x', \sinh(tA)(x)\rangle\Big] \\
&= \frac{1}{t}\Big[\langle x, \sinh(tA)(x')\rangle - \langle x, \sinh(tA)(x')\rangle\Big] = 0.
\end{align*}

At $t = 0,$ $s_{x}(y,0) = \langle [y,x],x^\vee \rangle_{\g}.$ Then \[
\langle s_x,s_{x^{'}}\rangle_{|t = 0} = x^{\vee}([y,x^{'}]) + (x^{'})^{\vee}([y,x]) = \langle x,[y,x^{'}] \rangle_{\g} + \langle x^{'},[y,x]\rangle_{\g} = \langle [x,y],x^{'} \rangle_{\g} - \langle x^{'},[x,y]\rangle_{\g} = 0.
\]

Hence $\mathcal{L}$ is isotropic. Since $\operatorname{rank}(\mathcal{L}) = n = \frac{1}{2}\dim(T_y\g \oplus T_y^*\g)$, it is Lagrangian.

\end{proof}
 
\begin{Lem}(Smoothness of $\eta_{\G}$) \label{smooth eta}
The 3-form $\eta_{\G}$ extends smoothly to the deformation space $\G$ such that $\eta_{\G,0} = 0$
\end{Lem}

\begin{proof}
The Cartan 3-form $\eta_G$ can be expressed in left-trivialization as
\[
(\eta_{G})_g(u_1, u_2, u_3) = \tfrac{1}{12}\big\langle [dL_{g^{-1}}u_1, dL_{g^{-1}}u_2], dL_{g^{-1}}u_3\big\rangle_\g.
\]
Writing $d\psi_t|_y(v) = t\,dL_{\exp(ty)} \circ \frac{1 - e^{-tA}}{tA}(v)$ and denoting $S_t := \frac{1-e^{-tA}}{tA}$, we compute:
\begin{align}
\psi_t^*\eta_{G}\big|_y(v_1,v_2,v_3)
&= \varphi_{\exp(ty)}\big(d\psi_t(v_1), d\psi_t(v_2), d\psi_t(v_3)\big) \notag\\
&= \frac{t^3}{12}\big\langle [S_t(v_1), S_t(v_2)], S_t(v_3)\big\rangle_\g. \label{eq:pullback-phi}
\end{align}
Since $S_t = \mathrm{Id} - \frac{tA}{2} + \frac{t^2A^2}{6} - \cdots$ is smooth in $(y,t)$, the pullback $\psi_t^*\varphi = O(t^3)$ as $t \to 0$. In the deformation space, the fiberwise 3-form is
\[
\eta_{\mathcal{G},t} := \frac{1}{t}\,\psi_t^*\eta_{G}\big|_y(v_1,v_2,v_3) = \frac{t^2}{12}\big\langle[S_t(v_1),S_t(v_2)],S_t(v_3)\big\rangle_\g.
\]
This is manifestly smooth in $(y,t)$ (the factor $t^3/t = t^2$ removes the singularity), and at $t=0$:
\[
\eta_{\mathcal{G},0} = \frac{0}{12}\big\langle [v_1, v_2], v_3 \big\rangle_\g = 0.
\]
This completes the proof that $\eta_\mathcal{G}$ extends smoothly with $\eta_{\mathcal{G},0} = 0$, consistent with the fact that $(\g^*, \pi_{\mathrm{KKS}})$ is an untwisted Dirac manifold with trivial kernel, hence a Poisson manifold.
\end{proof}

\begin{proof}(of Proposition \ref{Dirac G})
This follows immediately from Lemmas \ref{vect}, \ref{lem:cov} and \ref{smooth eta}.
\end{proof}

\subsection{Dirac deformation as a generalization of Poisson deformation of quasi-Poisson manifolds} In this subsection, we prove that the Dirac deformation theory generalizes strictly the Poisson deformation of Hamiltonian quasi-Poisson manifolds. We show that the Poisson deformation in the sense of \cite[\text{Definition 2.1}]{maiza2026multiplicative} give rise to a Dirac deformation in the sense of \ref{Dirac def}, and that there exist Dirac deformations outside the range of quasi-Poisson manifolds.

\begin{Def}\cite[Definition 2.1]{maiza2026multiplicative} \label{Poisson def}
A \emph{Poisson deformation} of a Hamiltonian quasi-Poisson $G$--manifold $(M,P,\phi_{M},\mu)$ to a Hamiltonian Poisson $G$--manifold $(N,P_{0},\nu)$ consists of a smooth manifold $\mathcal{QD}$ equipped with:
\begin{enumerate}
\item a $G$--invariant smooth submersion $\pi:\mathcal{QD}\to[0,1]\subset \R$,
\item a smooth $G$--action on $\mathcal{QD}$,
\item a smooth $G$--equivariant map $\tilde{\mu}:\mathcal{QD}\to\mathcal{G}$ satisfying $\pi_{\mathcal{G}}\circ\tilde{\mu}=\pi$,
\item a smooth section $\tilde{P}$ of $\bigwedge^{2}(\ker d\pi)\to\mathcal{D}$,
\item a smooth section $\phi_{\mathcal{QD}}$ of $\bigwedge^{3}(\ker d\pi)\to\mathcal{D}$.
\end{enumerate}
Writing $\mathcal{QD}_{t}=\pi^{-1}(t)$ for the fiber over $t\in(0,1]$, with $\tilde{P}_{t}$, $\phi_{t}$, $\tilde{\mu}_{t}$ denoting the restricted data, we require:
\begin{itemize}
\item $(\mathcal{QD}_{t},\tilde{P}_{t},\phi_{t},\tilde{\mu}_{t})\cong (M,tP,t^{2}\phi_{M},\mu)$ as Hamiltonian quasi-Poisson $G$--manifolds,
\item $(\mathcal{QD}_{0},\tilde{P}_{0},\nu)\cong (N,P_{0},\mu_{0})$ as Hamiltonian Poisson $G$--manifolds, with $\phi_{0}=0$.
\end{itemize}
In particular, $\mathcal{QD}_{1}\cong M$ and $\mathcal{QD}_{0}\cong N$.
\end{Def}

\begin{Thm}(The embedding theorem) \label{embed thm}
Let $(M,P,\phi_M,\mu)$ be a Hamiltonian $G$ quasi-Poisson manifold that admits a Poisson deformation $(\QD,\tilde{\mu},\tilde{P},\phi_{\QD},\pi)$ in the sense of Definition \ref{Poisson def} to a Hamiltonian $G$-Poisson manifold $(N,P_0,\nu).$ Then, there exists a Dirac deformation $(\M,\G,\L,\K,\lambda_{\M},\lambda_{\G}, \widehat{\phi})$ in the sense of Definition \ref{Dirac def}, with $\M = \QD, \X = \G$ and $\widehat{\phi} = \tilde{\mu}.$
\end{Thm}

The following lemma will be useful for the proof of the embedding theorem \ref{embed thm}.

\begin{Lem} \label{embd lem}
For each $t \in (0,1],$ the Hamiltonian quasi-Poisson isomorphism $(\QD_t,\tilde{P}_t,\phi_t,\tilde{\mu}_t) \cong (M,tP,t^{2}\phi_M,\mu)$ induces a $\tilde{\mu}_{t}^{*}\eta_{\G,t}$-twisted Dirac structure $\L_t$ via the construction \ref{Dirac qp} that extends to a smooth Lagrangian sub-bundle $\L \subset \mathbb{T}^{\text{vect}}\QD$ with $\L_0 = L_{P_0}$
\end{Lem}

\begin{proof}

The proof is organized through the following steps: \\

\noindent\textbf{Step 1: Source Dirac structure.} For $t\neq 0$, the quasi Poisson structure $(M,tP,\mu)$ yields by the construction \ref{Dirac qp} the Dirac structure \begin{align*}
    \L_t = \{((tP)^{\sharp} + x_{\QD_t}, C_{t}^{*}(\alpha_t) + \tilde{\mu}_t^{*}\sigma_t(x)), \quad \alpha_t \in T^{*}\QD_t, x\in \g\},
\end{align*}
where \begin{align*}
    C^{*}_{t} &= \mathrm{Id} - \dfrac{t}{2} \sum_{i} (e_i)_{M}\otimes \tilde{\mu}_{t}^{*}(\theta^{L}_i - \theta^{R}_i), \quad \sigma_{t}(x) = \dfrac{1}{t} \sigma(x).
\end{align*}
As $\tilde{P}$ is a smooth section of $\bigwedge^{2} \Ker(d\pi),$ the map $\alpha \mtoo (\tilde{P}_{t})^{\sharp}(\alpha)$ varies smoothly on $t$, with $\tilde{P}_{0}^{\sharp} = P_0^{\sharp}.$ Moreover, the vector field $x_{\QD_t}$ is smooth for every $t$ since the action of $G$ on $\QD$ is smooth. Thus, the first component is smooth for any $t.$ On the other side, using the $\psi$-coordinates of the central fiber in $\G,$ the Maurer-Cartan forms satisfy $\tilde{\mu}_{t}^{*}(\theta^{L}_i - \theta^{R}_i) = O(t)$ since $\theta^{L}_i(e) = \theta^{R}_i(e).$ Then, $C^{*}_{t} = \mathrm{Id} + O(t)$ near the central fiber, and $C^{*}_0 = \mathrm{Id}$.  As the map $\tilde{\mu}: \QD \too \G$ is smooth, $C^{*}_{t}(\alpha_t)$ varies smoothly in $t$, with $C^{*}_{0}(\alpha_0) = \alpha_{0}.$ Similarly, $\tilde{\mu}^{*}_{t}\sigma_{t}(x)$ extends smoothly to $\nu^{*}x$ (Lemma \ref{lem:cov}). Finally, as $t$ goes to 0, we get \[
\L_{0} = \{(P_{0}^{\sharp}(\alpha_0) + x_N, \alpha_0 + \nu^{*}x), \quad \alpha_0 \in T^{*}N, x\in \g\}.
\]
Consider the variable change $\beta = \alpha_0 + \nu^{*}x.$ Then \[  
P_{0}^{\sharp}(\alpha_0) + x_N = P_{0}^{\sharp}(\alpha_0) + P_{0}^{\sharp}(\nu^{*}x) = P_{0}^{\sharp}(\alpha_0 + \nu^{*}x) = P^{\sharp}_{0}(\beta). 
\]
Thus, $\L_{0} = \{(P^{\sharp}_{0}(\beta),\beta), \quad \beta \in T^{*}N\} = \L_{P_0}.$\\

\paragraph{\textbf{The Lagrangian condition.}} For $t\neq 0,$ $\L_t$ is Lagrangian and $\text{rank}(\L_t) = \text{dim}(\QD_t)$ by \cite[\text{Theorem 3.16}]{bursztyn2005dirac}, and as the graph of Poisson bi-vector for $t = 0.$ Since $\pi: \QD \too \R$ is a smooth submersion, the fiber dimension is constant, hence so is the rank of $\L_t.$\\

\paragraph{\textbf{The twisting 3-form on the source.}} We define the 3-form $\lambda_{\QD}$ on $\QD$ to be the pull-back of $\lambda_{\G}$ by the map $\tilde{\mu}.$ By Lemma \ref{smooth eta}, $\lambda_{\G}$ extends smoothly across $t = 0$ with $\lambda_{\G,0} = 0.$ Combined with the smoothness of $\tilde{\mu},$ this gives a smooth 3-form on $\QD.$ Fiber-wise, for $t\neq 0,$ $\lambda_{\QD,t} = \tilde{\mu}_{t}^{*}\lambda_{\G,t} = \dfrac{1}{t}\mu^{*}\eta_{G},$ and $\lambda_{\QD,0} = \tilde{\mu}_{0}^{*}\lambda_{\G,0} = \nu^{*}(0) = 0.$\\

\noindent\textbf{Strong Dirac condition.} For any $t\neq 0,$ the map $\tilde{\mu}_t: (\QD_t,\L_t,\lambda_{\QD,t}) \too (\G_t,\L_{\G,t},\lambda_{\G,t})$ is a strong Dirac map. Indeed, $\tilde{\mu}_t$ is $f$-Dirac, and $\mu^{*}\lambda_{\G,t} = \lambda_{\M,t}.$ The kernel transversality,i.e., $\Ker(d\tilde{\mu}_t)\cap \Ker(\L_t) = 0, \text{ for all }  t$ follows immediately from \cite[\text{Theorem 3.16}]{bursztyn2005dirac}.
\end{proof}

\begin{proof}(proof of Theorem \ref{embed thm})
We assemble the Dirac deformation data:
\begin{itemize}
    \item \textbf{Source:} $\M \coloneqq \QD,$ with projection $\pi_{\M} = \pi.$
    \item \textbf{Target:} $ \X \coloneqq \G = \D(G,\{1\})$ with projection $\pi_\X = \pi_{\G}.$
    \item \textbf{Fiber-wise Dirac structures} $\L$ and $\K$ from Lemma \ref{embd lem} and Proposition \ref{Dirac G}.
    \item \textbf{Smooth 3-forms.} $\lambda_{\M}$ from Lemma \ref{embd lem} and $\lambda_\K$ from Proposition \ref{Dirac G}.
    \item \textbf{The moment map.} $\widehat{\phi} \coloneqq \tilde{\mu}: \QD \too \G.$
\end{itemize}
The proof follows immediately from Lemma \ref{embd lem} and Proposition \ref{Dirac G}.
\end{proof}

The following examples show that the Dirac deformation theory \ref{Dirac def} strictly generalizes the Poisson deformation theory. More specifically, there exist Dirac deformations that do not arise from any Hamiltonian quasi-Poisson manifold in the sense of Theorem \ref{embed thm}.

\begin{Exp}(Uncontrolled Kernel) \label{ex:UncontrolledKernel} 
Let $M = \R^{4}$ with coordinates $(x_1,x_2,x_3,x_4).$ Consider the family of two-forms $\omega_t = dx_1\wedge dx_2 + (1-t)dx_3 \wedge dx_4, \quad t\in [0,1].$ $\omega_t$ is closed and the associated Dirac structure $\L_t \coloneqq L_{\omega_t} = \{(v,\omega^{\flat}(v))\}$ which is untwisted ($\lambda_{\M,t} = 0$ for each $t$). At $t = 0,$ $\omega_0 = dx_1\wedge dx_2 + dx_3 \wedge dx_4$ is symplectic (rank $4$) and $\L_{0} = \L_{\pi_0}$ where $\pi_0 = (\omega_{0}^{\flat})^{-1},$ is the Poisson structure. For $t\in [0,1),$ $\omega_t$ is symplectic and so $\L_t$ is again Poisson-Dirac. For $t = 1,$ $\omega_1 = dx_1 \wedge dx_2$ so that $\ker(\omega) = \textbf{span}(\partial x_3, \partial x_4),$ which gives $\Ker(L_{\omega_1}) = L_{\omega_1}\cap T\R^{4} = \ker(\omega_1).$ Taking any Poisson manifold $(Y,\pi_Y)$ (eventually $Y = {\text{pt}}$) and a Poisson map $\widehat{\phi}_0: (\R^{4},\pi_0) \too (Y,\pi_Y)$ assembles into a Dirac deformation theory from $(\R^{4},L_{\omega_1})$ to $(\R^4, L_{{\pi}_0}).$ Now, suppose that $L_{\omega_1}$ arises from a quasi-Hamiltonian manifold $(M,G,\mu)$ within the Dirac deformation theory. Then $\Ker(\omega_1)_{m} = \{x_{M}(m), \Ad_{\mu(m)}x + x = 0\}.$ This requires a compact Lie group $G$ acting on $\R^4$ such that the fundamental vector fields generate $\text{span}(\partial x_3, \partial x_4)$ at each point. Then, $G$ acts freely on the $(x_3,x_4)$ direction and trivially on $(x_1,x_2).$ In other words, $\partial x_3$ and $\partial x_4$ act as translations which generate $\R^{2}$ as a subgroup of $G.$ But, as $R^{2}$ is not compact, the integral curves of $\partial x_3$ and $\partial x_4$ are non-compact lines which contradicts the fact that the $G$-orbits of a compact Lie group are compact.       
\end{Exp}

\begin{Exp} \label{ex:CP2}  
Let $M = \mathbb{CP}^{2}$ be the complex projective space of dimension $2$. $M$ is a compact complex symplectic manifold with respect to the Fubini-Study two-form $\omega_{\text{FS}}$ that induces a Poisson-Dirac structure $L_{\pi_0},$ with $\pi_0 = (\omega_{\text{FS}}^{\flat})^{-1}.$ Consider $ X = M = \mathbb{CP}^{2}, \M = \X = \mathbb{CP}^{2} \times \R, \widehat{\phi} = \mathrm{Id}, \L_t = \K_t = \L_{\pi_{0}}$ and $\widehat{\phi} = \mathrm{Id}.$ This assembles into a Dirac deformation theory from $M$ to itself. Suppose that this deformation arises from a Hamiltonian $G$-quasi-Poisson manifold. For each $t$, $\mathbb{CP}^{2}$ is isomorphic to $G$. Since $\mathbb{CP}^{2}$ is simply connected, $\pi_{1}(CP^{2}) = 0$ and $H^{2}(\mathbb{CP}^{2},\mathbb{Z}) \cong \mathbb{Z}.$ Now, by Bott periodicity and Hurewicz, for any compact, connected, simply connected Lie group $G$, $H^{2}(G,\mathbb{Z}) = 0$ Since $H^{2}(\mathbb{CP}^{2},\mathbb{Z}) \neq H^{2}(G,\mathbb{Z})$ for any simply-connected compact Lie group, $\mathbb{CP}^2$ is not even homotopy equivalent to $G$. For non-simply connected Lie group $G$, the universal cover $\tilde{G}$ is a compact Lie group with $H^{2}(\tilde{G}) = 0$ and $\pi_1(G) \neq 0,$ which contradicts $\pi_{1}(\mathbb{CP}^2) = 0.$ 
\end{Exp}

\begin{Rem}
 The Dirac deformation theory developed here can be interpreted in terms of $L_{\infty}$ algebras. The twisted Dirac structure $(M,L_M,\lambda_M)$ determines an $L_{\infty}$-algebra $\g^{\bullet}_{L_M} = \Gamma(\bigwedge^{\bullet + 1}L_M)$ with differential $d_{L_M},$ Schouten binary bracket $l_2$ and a cubic bracket $l_3$ controlled by $\lambda_M.$ A Dirac deformation in the sense of Definition \ref{Dirac def} corresponds to a smooth family of Maurer-Cartan elements in a family of such $L_{\infty}$-algebras, degenerating at $t = 0$ to a DGLA Maurer-Cartan element in the Poisson deformation complex. The existence of a Dirac deformation is then controlled by two-obstruction theory: pointwise $L_{\infty}$-Maurer-Cartan obstructions at each $t,$ together with a degeneration obstruction controlling the contraction $l^{t}_3 \too 0$ as $t \too 0.$ A systematic obstruction theory is left to future work. 
\end{Rem}

\section{Dirac deformation and reduction along strong Dirac maps.} \label{section 4}
The purpose of this section is to prove that the Dirac deformation theory is compatible with the reduction along strong Dirac maps (see for instance \cite{bualibanu2026reduction}). We then apply this to all fundamental examples of \cite{bualibanu2026reduction}, recovering deformations from quasi-Poisson reduction to Hamiltonian reduction.

\subsection{Reduction along strong Dirac maps} We recall the definitions from \cite{bualibanu2026reduction}. Let $(X,L_X)$ be a $\lambda_X$-twisted Dirac manifold. A strong Dirac map \cite[\text{equation 1.7}]{bualibanu2026reduction} is an $f$-Dirac map $\phi: (M,L_M) \too (X,L_X)$ such that \begin{equation} \label{eq s Dirac}
    \phi^{*}\lambda_X = \lambda_M \text{ and } \Ker(d\phi) \cap \Ker(L_M) = 0.
\end{equation}

\begin{Def} \cite[Definition 2.2]{bualibanu2026reduction} \label{gen red level}
\textit{A generalized reduction level} is a triple $(S,\gamma,\phi)$ consisting of a submanifold $i: S\too X,$ a two-form $\gamma \in \Omega^{2}(S),$ and a smooth map $\varphi: S \too Z$ of constant rank to a Dirac manifold $(Z,L_Z)$ such that $d\phi(\tau_{-\gamma}i^{*}L_X) = L_Z, i^{*}\lambda_X + d\gamma = \phi^{*}\lambda_Z,$ and the distribution \begin{equation}
    A_{S,\gamma,\phi} \coloneqq L_X \cap i^{*}L_X \cap (\ker(d\varphi) \oplus T^{*}X) 
\end{equation}
has constant rank.
\end{Def}

\begin{Thm} \cite[Theorem 2.21]{bualibanu2026reduction} \label{reduction th}
Let $\Phi: (M,L_M) \too (X,L_X)$ be a strong Dirac map and $(S,\gamma,\phi)$ a generalized level of $X$ such that $S$ intersects cleanly $\Phi.$ If the quotient $Q = \Phi^{-1}(S)/A_{S,\gamma,\phi}$ is a manifold then \begin{enumerate}
    \item $Q$ inherits a Dirac structure $L_Q = \pi_{*}L_C,$ where $L_C = \tau_{-\Phi_{C}^{*}\gamma}j^{*}L_M.$
    \item The map $\phi \circ \Phi_C$ descends to a strong Dirac map $\bar{\phi}: Q \too Z.$
    \item If $M$ is non-degenerate with $L_M = L_\omega,$ then $L_Q$ is also non-degenerate with $\pi^{*}\bar{\omega} = j^{*}\omega - \Phi_{C}^{*}\gamma.$
\end{enumerate}
\end{Thm}

\subsection{Deformation of reduction levels}
We first introduce a deformation of reduction levels compatible with the Dirac deformation theory. 

\begin{Def} \label{defo level}
Let $(\X,\K,\lambda_\X, \pi_{\X})$ be a smooth family of twisted Dirac manifolds over $\R$ (as in Definition \ref{Dirac def}). Suppose that $S \subset X$ (respectively $\tilde{S} \subset Y$) is a generalized reduction level of $X$(resp. $\tilde{S}$ is a reduction level of $Y$). A deformation of reduction level on $\X$ consists of \begin{enumerate}
    \item A smooth submanifold $\S$ of $\X$ such that $\pi_{X}|_\S$ is a submersion, so that $\S_t \coloneqq \X_t \cap \S$ is a submanifold of $\X_t$ for each $t.$
    \item A smooth section $\widehat{\gamma}$ of $\bigwedge^{2}(\Ker(d\pi_{X}|_\S))^{*}.$
    \item A smooth family of Dirac manifolds $(\Z,\lambda_\Z,\pi_{\Z})$ and a smooth fiber-wise map $\varphi: \S \too \Z$ such that \begin{itemize}
        \item For each $t,$ $(\S_t,\widehat{\gamma}_t,\varphi_t)$ is a generalized reduction level of $(\X_t,\K_t,\lambda_{\X,t},)$ in the sense of Definition \ref{gen red level}.
        \item $(\S_1,\widehat{\gamma}_1,\varphi_1 \cong (S,\gamma,\varphi), (\S_0,\widehat{\gamma}_0,\varphi_0) \cong (\tilde{S},\tilde{\gamma},\tilde{\varphi})$ and $\widehat{\gamma}_0$ is closed.
        \item The fiberwise stabilizers sub-algebroids $A_{\S_t,\widehat{\gamma}_t,\varphi_t}$ assemble into a smooth family over $\R.$
    \end{itemize}
\end{enumerate}
\end{Def}

Now, we are ready to state the principal theorem of this section.
\begin{Thm}[Dirac deformation of reduced spaces]\label{thm:main}
Let $(\M, \X, \L, \K, \lambda_\M, 
\lambda_X, \widehat\Phi)$ be a Dirac deformation from $(M, L_M)$ to 
$(N, L_N)$, and let $(\S, \widehat{\gamma}, \varphi, \Z)$ be a deformation of reduction levels on $X$. Suppose that:
\begin{enumerate}
    \item[(H1)] $\S$ has a clean intersection with $\Phi$ and the restriction $\pi_{\M}|_{\C}$ is a surjective submersion, where $\C = \widehat{\Phi}^{-1}(\S).$
    \item[(H2)] The fiber-wise quotient 
    $\mathcal{Q}_t := \hat\Phi_t^{-1}(\S_t)/A_{\S_t,\widehat{\gamma}_t,\varphi_t}$ 
    is a smooth manifold for each $t$, and the family 
    $\Q := \bigsqcup_{t} \mathcal{Q}_t$ is a smooth manifold 
    over $\R$.
    \item[(H3)] The isotropy algebra of the stabilizer action has locally constant dimension on $\mathcal{C}$. Equivalently,
    $\operatorname{rk}(\mathcal{L} \cap T^{\circ}\mathcal{C})$ is locally constant on $\mathcal{C}$, where $T^{\circ}\mathcal{C} \subset T^{*}\mathcal{M}$ denotes the fiberwise conormal of $\mathcal{C}$ inside $\mathcal{M}$.
\end{enumerate}
Then:
\begin{enumerate}
    \item The family $Q$ carries a smooth fiber-wise Lagrangian bundle
    $\mathcal{L}^{\mathrm{red}}$ such that $\mathcal{L}^{\mathrm{red}}_t$ is the 
    reduced Dirac structure from Theorem \ref{reduction th} applied to 
    $\hat\Phi_t : (\M_t, \L_t) \to (\X_t, \K_t)$ 
    at the reduction level $(\S_t, \widehat{\gamma}_t, \varphi_t)$.
    \item The reduced strong Dirac maps 
    $\bar\Phi_t : (\Q_t, \L^{\mathrm{red}}_t) \to 
    (\Z_t, \K_{Z,t})$ assemble into a smooth fiberwise map 
    $\bar\Phi : \Q \to \Z$.
    \item The data $(\Q, \Z, \L^{\mathrm{red}}, 
    \K_Z, \lambda_\Q, \lambda_Z, \bar\Phi)$ is a Dirac deformation 
    from the reduced twisted Dirac manifold 
    $(\Q_1, \L^{\mathrm{red}}_1)$ to the reduced Poisson manifold 
    $(\Q_0, \L^{\mathrm{red}}_0)$.
    \item If each $\M_t$ is non-degenerate with $\L_t = L_{\omega_t}$, 
    then each $\Q_t$ is non-degenerate, and the reduced 2-forms 
    $\bar\omega_t$ on $\mathcal{Q}_t$ vary smoothly in $t$.
\end{enumerate}
\end{Thm}

Before giving the proof, we prove the following useful lemma 

\begin{Lem} \label{lem:central}
    Let $(\M_0,\L_O)$ be the central Poisson fiber, with $\L_{0} = L_{\pi_0}$ the graph of a Poisson bivector $\pi_0$ and $\widehat{\Phi}_0: \M_0 \too \X_0$ be the corresponding Poisson map. Let $(\S_0,\hat{\gamma}_0,\varphi_0)$ be the central reduction level, with $\lambda_{\X,0} = 0$ and $\hat{\gamma}_0 = 0$ being closed, and set $\C_0 = \widehat{\Phi}^{-1}_{0}(\S_0).$ Then, for each $p\in \C_0$ the anchor-plus-conormal map restricts to a linear isomorphism \[
    \A_{\S_0,\hat{\gamma}_{0},\varphi_0}|\widehat{\Phi}_{0}(p) \too \L_0 \cap T_{p}^{\circ }\C_0.
    \]
    In particular, $\rk(\L_0 \cap T_{p}\C_0) = \dim \A_{\S_0,\hat{\gamma}_0,\varphi_0}|\widetilde{\Phi}_{0}(p),$ and hence $r(p) \leq i(p)$ for every $p \in \C_0.$
\end{Lem}

\begin{proof}
Since $\L_0$ is Poisson, $\Ker(\L_0)= \L_0\cap T\M_0=0$, so the projection $\prTs: \L_0\to T^\ast M_0$, $(\pi_0^\sharp\alpha,\alpha)\mapsto\alpha$, is a
bundle isomorphism onto $T^\ast M_0$; write $\kappa=(\prTs|_{L_0})^{-1}$ for its
inverse, $\kappa(\alpha)=(\pi_0^\sharp\alpha,\alpha)$.
 
Under $\prTs$, the subspace $\L_0\cap T^{\circ}_{p}\C_0$ maps isomorphically onto
\[
  \Lambda_p:=\big\{\alpha\in T^{\circ}_{p}\C_0 :
   (\pi_0^\sharp\alpha,\alpha)\in L_0\big\}
   =\prTs\big(\L_0\cap T^{\circ}_{p}\C_0\big),
\]
the last equality because every element of $\L_0$ has the form
$(\pi_0^\sharp\alpha,\alpha)$.
 
We now identify $\Lambda_p$ with the pulled-back annihilator data downstairs. Since $\Phib_0$ is a Poisson map and $\lambda_{\X,0}=0$, the strong-Dirac transversality
$\Ker(d\Phib_0)\cap\Ker(\L_0)=0$ holds trivially, so $\Phib_0$ induces the Lie-algebroid action of \cite[Section 2]{bualibanu2026reduction}. The only place the general
($\lambda\neq0$) correspondence uses the twist is the mixed term of the Dorfman bracket; at $t=0$ this term is
\begin{equation}\label{eq:dorfman}
  \iota_{v\wedge w}\lambda_{\X,0}=0
  \qquad\text{ for all } v,w,
\end{equation}
because $\lambda_{\X,0}=0$. Thus the bracket restricts to the untwisted Dorfman bracket on $T^\circ_{p}\C_0$, and the strong-Dirac correspondence identifies
\[
  \Lambda_p \;\cong\; \Phib_0^\ast\big(\K_0\cap T^{\circ}_{\Phib_0(p)}\S_0\big),
\]
where $T^{\circ}\S_0=\ker(di_0)^\ast$ and $i_0\colon \S_0\hookrightarrow \X_0$ is the inclusion.
 
Finally, by \cite[Definition 2.2]{bualibanu2026reduction} the stabilizer subalgebroid on the central
level is
\[
  \A_{\S_0,\gammab_0,\varphi_0}
   =\K_0\cap (i_0)_{*} \K_{\gammab_0}\cap(\ker d\varphi_0\oplus T^\ast \X_0),
\]
and with $\gammab_0$ closed the gauge term drops out, so the anchor-plus-conormal map sends $\A_{\S_0,\gammab_0,\varphi_0}\big|_{\Phib_0(p)}$ isomorphically onto $\Lambda_p$. Composing with $\kappa$ gives the asserted isomorphism onto $\L_0\cap T^{\circ}_{p}\C_0$. Taking dimensions yields $r(p)=\iota(p)$, and in particular $r(p)\le\iota(p)$.
\end{proof}

\begin{proof}(of Theorem \ref{thm:main})
We proceed in three steps.\\

\noindent\textbf{Step 1: Fiberwise application of Theorem \ref{reduction th}} For each $t\neq 0,$  the map $\widehat{\Phi}_t: (\M_t,\L_t,\lambda_{\M,t}) \too (\X_t,\K_t,\lambda_{\X,t})$ is a strong Dirac map by the condition (1) of Definition \ref{Dirac def}, and $(\S_t,\widehat{\gamma}_t,\varphi_t)$ is a generalized reduction level of $\X_t$ by Definition \ref{defo level}. The clean intersection hypothesis $(H1)$ and the manifold condition $(H3)$ ensure that Theorem \ref{reduction th} applies to each fiber. Define, \[
\mathcal{C}_t = \widehat{\Phi}_{t}^{*}(\S_t), \quad L_{\C_t} \coloneqq \tau_{-\widehat{\Phi}_{\C_t}^{*}\widehat{\gamma}_t} j_{t}^{*} \L_t,
\]
where $j_{t}: \C_t \too \M_t$ is the inclusion. By Theorem \ref{reduction th}, the push-forward $L^{\text{red}}_t \coloneqq (\pi_t)_{*}L_{\C_t}$ is a Dirac structure on $\Q_t,$ and $\bar{\varphi}_t: \Q_t \too \Z_t$ is a strong Dirac map.\\ At $t = 0,$ $\widehat{\Phi}_0$ is a Poisson map (hence a strong Dirac map without twisting), and $(\S_0,\X_0,\varphi_0)$ is a reduction level of the Poisson manifold $\X_0.$ By \cite[\text{Remark 2.2(1)}]{bualibanu2026reduction}, when $\Z_0$ is a point, the reduced Dirac structure $\L^{\text{red}}_0$ is Poisson. More generally, $\bar{\varphi}_0$ is a Poisson map to $\Z_0.$\\

\noindent\textbf{Step 2: Smoothness of $\mathcal{L}^{\mathrm{red}}$ across $t=0$.}
The total preimage $\C=\Phib^{-1}(\S)$ is a smooth submanifold of $\M$ by (H1), and $\pi_\M|_\C$ is a
surjective submersion, so $\C$ fibers smoothly over $\mathbb{R}$ with fibers $\C_t$. We show that the pullback $\jmapb^{\ast}\L$ is a smooth Lagrangian subbundle of $\mathbb{T}\C$. Smoothness of $\L$ as a subbundle of $\mathbb{T}M$ holds by the Dirac
deformation hypothesis.
 
By \cite[Proposition 5.6, Example 5.7]{bursztyn2013brief}, the backward image
$\jmapb^{\ast}\L=B_{\jmapb}(\L)$ is a smooth fiberwise Lagrangian subbundle of $\mathbb{T}\C$ precisely when the clean-intersection quantity $\L\cap T^{\circ}\C$ has locally constant rank, where $T^{\circ}\C = \ker(d\jmapb_t)^\ast\subset T^\ast \M_t$ is the
fiberwise conormal of $\C_t$ in $\M_t$; being Lagrangian, $\jmapb^{\ast}\L$ then automatically has rank $\dim \C_t$. It therefore suffices to prove that
\[
  r(p):=\rk\big(\L_p\cap T^{\circ}_{p}\C\big),\qquad p\in \C,
\]
is locally constant on the \emph{total space} $\C$ -- not merely along each fiber. This is the delicate point: as an intersection of the two smooth subbundles $\L$ and
$T^{\circ}\C$ of $\mathbb{T}\M|_\C$, the function $r$ is \emph{a priori} only upper-semicontinuous, so a rank constant on every fiber may still jump upward as $t\to 0$.
 
\medskip

\noindent\textbf{(i) $r$ is upper-semicontinuous.}
The sum $\L+T^{\circ}\C=\mathrm{im}\big(\L\oplus T^{\circ}\C\to\mathbb{T}\M|_\C\big)$ is the image
of a bundle map, so its rank is lower-semicontinuous on $\C$
\cite[Section 1.1]{bursztyn2013brief}. Since $\L$ and $T^{\circ}\C$ are subbundles of constant rank
-- the former by the Dirac deformation hypothesis, the latter because $\C$ is a fixed submanifold, so its fiberwise conormal has constant rank $\operatorname{codim}
\C_t$ -- the identity
\[
  \rk\big(\L\cap T^{\circ}\C\big)
   =\rk \L+\rk T^{\circ}\C-\rk\big(\L+T^{\circ}\C\big)
\]
exhibits $r$ as a constant minus a lower-semicontinuous function. Hence $r$ is upper-semicontinuous on $\C$: it can only jump upward, and a value constant on each
fiber $\C_t$ for $t\neq0$ can only increase as $t\to 0$. Excluding such an upward jump on the central fiber is the whole difficulty.
 
\medskip
\noindent\textbf{(ii) The comparison function $\iota$ is locally constant on $\C$.}
For each $t\neq 0$, the strong Dirac condition
$\Ker(d\Phib_t)\cap \Ker(\L_t)=0$ together with the clean intersection (H1) implies, by \cite[Proposition 2.17]{bualibanu2026reduction}, that
\begin{equation}\label{eq:star}
  r(p)=\iota(p):=\dim\A_{\S_t,\gammab_t,\varphi_t}\big|_{\Phib_t(p)},
  \qquad p\in \C_t,\ t\neq 0,
\end{equation}
the identification being the restriction of the anchor-plus-conormal map, so that
$\L_t\cap T^{\circ}\C_t$ at $p$ is the fiber of the stabilizer subalgebroid over $\Phib_t(p)\in \S_t$. By Definition \ref{defo level}, the fiberwise stabilizers $\A_{\S_t,\gammab_t,\varphi_t}$ assemble into a smooth subalgebroid $\A$ over the
total space $\S$; being a smooth subbundle, its fiber dimension $s\mapsto \dim\A_s$ is locally constant on $\S$. Composing with the smooth map $\Phib|_\C : \C\to \S$ shows
that $\iota=\dim\A\circ\Phib$ is locally constant on all of $\C$, including the central fiber.
 
\medskip
\noindent\textbf{(iii) Two inequalities on the central fiber.}
Fix $p_0\in \C_0$.
 
\emph{Lower bound $r(p_0)\ge\iota(p_0)$.} By~\eqref{eq:star}, $r\equiv\iota$ on the dense open set $\{t\neq0\}$, and $\iota$ is locally constant by (ii), so
$\iota(p)=\iota(p_0)$ for all $p$ near $p_0$; hence $r(p)=\iota(p_0)$ for $p$ near
$p_0$ with $t\neq0$. Upper-semicontinuity of $r$ from (i) then gives
\[
  r(p_0)\ \ge\ \limsup_{p\to p_0} r(p)\ =\ \iota(p_0).
\]
 
\emph{Upper bound $r(p_0)\le\iota(p_0)$.} This is exactly the inequality supplied by Lemma \ref{lem:central}, which identifies $\A_{\S_0,\gammab_0,\varphi_0}|_{\Phib_0(p_0)}$
with $\L_0\cap T^{\circ}_{p_0}\C_0$ on the central Poisson fiber.
 
Combining the two, $r(p_0)=\iota(p_0)$ for every $p_0\in \C_0$. Together with \eqref{eq:star} on $\{t\neq0\}$, this gives $r=\iota$ pointwise on all of $\C$.
Since $\iota$ is locally constant by (ii), so is $r$. This is precisely the cross-fiber constancy recorded by hypothesis (H3), and it is exactly the condition needed for $\jmapb^{\ast}\L$ to be a smooth subbundle of $\mathbb{T}\C$.
 
\medskip
\noindent\textbf{(iv) Assembly.}
The gauge transformation $\tau_{-\Phib_C^\ast\gammab}$ is smooth since $\gammab$ is a
smooth fiberwise $2$-form. Therefore
\[
  \L^{\C}:=\tau_{-\Phib_\C^\ast\gammab}\,\jmapb^{\ast}\L
\]
is a smooth fiberwise Lagrangian structure on $\C$.
 
By (H2), the fiberwise quotient maps $\pi_t: \C_t \to \Q_t$ assemble into a smooth map $\pi\colon \C \to \Q$. The orbit distribution of the stabilizer action on $\C$ has
locally constant rank: at each point its dimension equals the rank of $\A_{\S_t,\gammab_t,\varphi_t}$ (locally constant in $t$ by Definition 4) minus the isotropy dimension $\iota$, which is locally constant on $\C$ by (ii)--(iii) above.
(Thus the smoothness of the orbit distribution, and hence of $\Lred$ on $Q$, is not
independent of the $t=0$ analysis: it reuses the cross-fiber constancy of $\iota$
just established.) Combined with the smoothness of the family
$\{\A_{\S_t,\gammab_t,\varphi_t}\}$ from Definition \ref{defo level}, this yields a smooth distribution on $\C$. Therefore the fiberwise pushforward
\[
  \Lred:=\pi_\ast \L^{\C}
\]
is a smooth fiberwise Lagrangian structure on $\Q$. \qed

\noindent{\textbf{Step 3: Smoothness of $\Bar{{\widehat{\Phi}}}$ and the Dirac deformation structure.}}
The composition $\widehat{\Phi}_t \circ \varphi_t: \C_t \too \Z_t$ is smooth for each $t$ (both $\widehat{\Phi}$ and $\varphi$ are smooth). By \cite[\text{Proposition 2.20}]{bualibanu2026reduction}, its kernel contains the orbit distribution of $A_{\S,\widehat{\gamma}_t,\varphi_t}.$ Then, for each $t$, the composition $\widehat{\Phi}_t \circ \varphi_t$ descends to a smooth map $\Bar{\widehat{\Phi}}_t: \Q_t \too \Z_t$ and together assemble to a smooth map $\Bar{\widehat{\Phi}}: \Q \too \Z.$ 
The fiberwise twisting on $\Q$ is $\lambda_{Q,t} = (\pi_{t})_*(\widehat{\Phi}^{*}_{\C_t}(i_{t}^{*} \lambda_{X,t} + d\widehat{\gamma}_t))$ by Theorem \ref{reduction th}(1). Since all ingredients are smooth in $t$ and $\lambda_{\X,0} = 0$ and $d\widehat{\gamma}_0 = 0,$ we have $\lambda_{\Q,0} = 0.$ Now, we verify the conditions of Definition \ref{Dirac def}: \begin{itemize}
    \item For every $t\neq 0,$ $\Bar{\widehat{\Phi}}_t$ is a strong Dirac map by Theorem \ref{reduction th}(2). 
    \item At $t = 0,$ $\Bar{\widehat{\Phi}}_0$ is Poisson (strong Dirac map with zero twisting).
    \item $(\Q_1,\L^{\text{red}}_1) \cong (\Phi^{-1}(S)/A_{S,\gamma_1,\varphi_1})$ with the reduced Dirac structure and $(\Q_0,\L^{\text{red}}_0) \cong \varphi^{-1}(\widetilde{S})/A_{\widetilde{S},\tilde{\gamma},\tilde{\varphi}}$ as Poisson manifolds.
\end{itemize}
\end{proof}

\begin{Rem}
When $\Z_t = \{\text{pt}\}$ for every $t$ (ordinary reduction level), the reduced Dirac structure is Poisson at every fiber, and the Dirac deformation of reduced spaces becomes a deformation of Poisson manifolds. 
\end{Rem}

\begin{Rem}
    Throughout Section \ref{F example}, we work under the standard regularity hypothesis on moment map values that already underlie the reduction theorem of \cite[Section 2.1]{bualibanu2026reduction}. Under these conditions, hypothesis (H1), (H2) and (H3) are satisfied in each of the examples below. Example \ref{conjugacy examp} requires in addition the orbit-dimension constancy hypothesis
$\dim \mathfrak{g}_{e^{tx}} = \dim \mathfrak{g}_x$ for all $t$, which is the
condition that the conjugacy class $C_{e^{tx}}$ and the coadjoint orbit
$\mathcal{O}_x$ have the same dimension; this holds for $x$ regular semi-simple, and more generally on the regular locus.
\end{Rem}

\section{Fundamental examples} \label{F example}
Now we apply Theorem \ref{thm:main} to each class of reduction from \cite[\text{Section 2.4}]{bualibanu2026reduction}.

\begin{Exp}(Deformation of quasi-Poisson reduction at a point). \label{conjugacy examp} \label{ex 5.1} Let $\Phi: M\too G$ be a Hamiltonian quasi-Poisson moment map and $(\M,\G,\widehat{\Phi})$ a Dirac deformation to a Hamiltonian Poisson manifold $\nu: N \too \g^*.$ By \cite[\text{Example 2.5}]{bualibanu2026reduction}, any point $x\in G$ is a reduction level with stabilizer sub-algebroid $A_x = L_{G,x}$ (the isotropy Lie algebra of $L_G$ at $x$). Let $y$ be a regular value of $\nu.$ For $t$ sufficiently small, $e^{ty}$ is again a regular value of $\Phi$ such that $L_{G,e^{ty}} \cong \g_y$ and $L_{\g^{*},y} \cong \g_y.$ In the local coordinates near the central fiber of $\G = \D(G,\{1\}),$ $e^{ty}$ deforms to $y.$ Consider $\S = \{(e^{ty},t), t\neq 0\} \cup \{(y,0)\}.$ Then, $\S$ is a smooth submanifold of $\G.$ For each $t,$ The two-form $\widehat{\gamma}$ is defined to be  $\widehat{\gamma}_t = 0$ for all $t$ and $\varphi_t:\S_t \too \{\text{pt}\}.$  By \cite[\text{Example 2.26}]{bualibanu2026reduction}, for each $t\neq 0,$ the quotient $\widehat{\Phi}_{t}^{-1}(e^{ty})/\g_y$ is a smooth manifold equipped with its Poisson structure that arise from the quasi-Poisson reduction. At $t = 0,$ it is $\widehat{\Phi}_{0}^{-1}(y)/\g_y$ (Marsden-Weinstein reduction). Theorem \ref{thm:main} says: the quasi-Poisson reduction $\Phi^{-1}(e^{y})/\g_y$ admits a Dirac deformation to the Marsden-Weinstein reduction $\nu^{-1}(y)/\g_y.$
\end{Exp}

\begin{Exp}(Deformation of reduction along conjugacy classes to co-adjoints orbits.)
We keep the same source data $(M,\pi,\Phi)$ as in the previous example. Fix $x\in \g$ a regular value and let $\O = C_{e^x}\subset G$ be the conjugacy class of $e^{x}$ in $G$ Similarly, let $\O_x$ be the co-adjoint orbit of $x$ in $\g^*$. By \cite[\text{Example 2.7}]{bualibanu2026reduction}, $(\O,\omega_O)$ is a reduction level of $G$ with $A_{O,\omega_O} = L_{G}|\O \cong O \rtimes \g.$ The Source, the target and the strong Dirac map are the same as in example \ref{ex 5.1}. Define $\S$ to be \[
\S \coloneqq \{(ge^{tx}g^{-1},t), g\in G, t\in I^*\} \cup \{(\Ad_{g}x,0), g\in G\}.
\]

In the $\psi$-coordinates, we have $\psi^{-1}(\S) = \O_x \times I,$ where $\O_x = {\Ad_{g}x, g\in G}.$ Then, $\S$ is a smooth submanifold. By \cite[\text{Theorem 4.2}]{burelle2026deformations}, the conjugacy class $\S_t = C_{e^{tx}} \in \G_t$ deforms to the co-adjoint orbit $\S_0 = \O_x.$ Moreover, there exists a smooth two form $\widehat{\gamma}$ such that $\widehat{\gamma}_t = \dfrac{\omega_{e^{tx}}}{t}$ for $t\neq 0$ and it extends smoothly to the KKS-symplectic form $\omega_{\O_x}.$ Finally, $A_{\S_t} = \S_t \rtimes \g$ for each $t$ and the rank of the isotropy algebras is constant for each $t.$ Theorem \ref{thm:main} applied to a quasi-Poisson moment map $\Phi:M\too G:$ The quasi-Poisson reduction $\Phi^{-1}(\O)/\g$ \cite[\text{Example 2.29}] {bualibanu2026reduction} admits a Dirac deformation to the Marsden-Weinstein reduction $\nu^{-1}(\O_x)/\g$ along the coadjoint orbit \cite[\text{Example 2.28}]{bualibanu2026reduction}.  
\end{Exp}

\begin{Exp}(Steinberg and Kostant slices)
Let $G$ be a simply connected semi-simple complex Lie group of rank $r.$ The Steinberg cross-section is \[
\Sigma = \{e_1^{z_1}\cdot e_2^{z_2} \cdot \cdot \cdot e_{r}^{z_r}\cdot \dot{w}_0^{-1}: z_i \in \mathbb{C}\} \subset G,
\]
where $e_1,....,e_r$ are simple roots vectors and $w_0 = s_1\cdot \cdot \cdot \cdot s_r$ is a Coxeter element representative. Its additive counterpart is the Kostant slice \[
\Sigma_0 = \{e_{-\theta} + z_{1}e_{1} +....+ z_{r}e_r: z_i \in \mathbb{C}\} \subset \g^{*}.
\]

Both are Poisson transversal in the Dirac-geometric sense \cite[\text{Example 2.30}]{bursztyn2007dirac}: For any point $g\in G,$ the generalized tangent space splits as $\mathbb{T}_{g}G = L_{G,g} \oplus (T_{g}\Sigma \oplus T_{g}^{*}\Sigma),$ and similarly for $(\g^{*},\pi_{KKS})$ along $\Sigma_0.$ \\

As to the deformation, one sets $\S_t = \Sigma$ for all $t\neq 0$ and $\S_0 = \Sigma_0.$ The smoothness of this family inside $\D(G,\{1\})$ is the Sevostyanov interpolation \cite{sevostyanov2011algebraic} and the BCH formula: \[
\dfrac{1}{t}(\text{log}(e^{tz_1}\cdots e^{tz_r}\cdot \dot{w}_{0}^{-1})) \too e_{-\theta}  + \sum_{i} z_ie_i.
\]

Moreover, all data are trivial: $\widehat{\gamma}_t = 0, \varphi_t = \mathrm{Id}$ so the target $\Z_t = \S_t$ and most importantly $A_{\S_t} = 0$ by the Poisson transversal axioms itself. So no quotient is needed and the reduced space $\Q$ is already a manifold with $\Q_t = \Phi^{-1}(\Sigma), t\neq 0$ and $\Q_0 = \nu^{-1}(\Sigma_0).$ The reduced Dirac structure is the pure backward image $\L^{\text{red}}_t = j_{t}^{*}\L_t$, with no gauge transformation, and no push-forward quotient. The axioms $(H1)-(H3)$ are automatic. As a consequence, Theorem \ref{thm:main} says that $(\Phi^{-1}(\Sigma),\pi_{|\Sigma})$ admits a Dirac deformation to $(\nu^{-1}(\Sigma_0),(\pi_{0})_{|\Sigma_0}).$ 
\end{Exp}

\begin{Exp}(Deformation data for parabolic deformation) Let $G$ be a complex reductive group, $P = LU$ a parabolic with Levi $L$ and unipotent subgroup $U.$ We keep the same source data $(M,\pi,\Phi).$ By \cite[\text{Example 2.12}]{bualibanu2026reduction}, $U$ is an ordinary reduction level of $G$ with stabilizer subalgebroid $A_U = U \ltimes \mathfrak{p}.$ Let $\S_t = U, t\neq 0$ and $\S_0 = \mathfrak{u}.$ In the deformation space $\D(G,\{1\}),$ this familly is smooth: $\exp: \mathfrak{u} \too U$ is a diffeomorphism (simply connected unipotent group) and $\phi(\xi,t) = (\exp(t\xi),t)$ parametrizes $\S$ for $\xi \in \mathfrak{u}.$ Moreover $\widehat{\gamma}_t = 0, \forall t$ and $A_{\S_t} = U \ltimes \mathfrak{p}$ for $t\neq 0$ deforms to $A_{\S_0} = \mathfrak{u} \ltimes \mathfrak{p},$ with constant dimension for each fiber. Theorem \ref{thm:main} says: For any quasi-Poisson moment map $\Phi: M\too G,$ the quasi-Poisson reduction $\Phi^{-1}(U)/P$(multiplicative parabolic reduction \cite[\text{Example 4.5}]{bualibanu2026reduction})  admits a Dirac deformation to $\nu^{-1}(\mathfrak{u})/P$ (Hamiltonian parabolic reduction \cite[\text{Example 2.31}]{bualibanu2026reduction}). In particular, when $\mathbf{D}(G)$ (the internal fusion of the double $D(G)$), the multiplicative cotangent bundle $G \times_{U} P$ of the partial flag variety $G/P$ (\cite[\text{Example 4.5}]{bualibanu2026reduction}) deforms to the ordinary cotangent bundle $T^{*}(G/P) = G\times_{P} \mathfrak{u}.$
\end{Exp}

\begin{Exp}(Unipotent reduction)
We keep the same notations of the last example, and let $L = P/U$ be the Levi quotient with projection $\varphi: P \too L.$ By \cite[\text{Example 2.13}]{bualibanu2026reduction}, $(P,\varphi)$ is a generalized reduction level of $G$ with stabilizer algebroid $P \rtimes \mathfrak{u}.$ Define $\S$ as follows: $\S_t = P, t\neq 0$ and $\S_0 = \mathfrak{p}.$ Within the deformation space $\D(G,\{1\}),$ the restriction $\exp_{\mathfrak{p}}: \mathfrak{p} \too P$ is a local diffeomorphism which is needed to assure that $\S$ is smooth near the central fiber. \\  

Also, the projection $\varphi_t: P\too L$ deforms to $\varphi_0: \mathfrak{p} \too \mathfrak{l}$ (the linearized projection $\mathfrak{p} = \mathfrak{l}\oplus \mathfrak{u} \too \mathfrak{l}$.) The stabilizers $A_{\S_t} = P \ltimes \mathfrak{u}$ for $t \neq 0$ deform to $A_{\S_0} = \mathfrak{p} \ltimes \mathfrak{u}.$ Then, with $\widehat{\gamma}_t = 0$ for each $t$,  Theorem \ref{thm:main} says: For any quasi-Poisson moment map $(\mu,\Phi): M \too K\times G,$ the $K\times L$ quasi-Poisson structure on $\Phi^{-1}(P)/U$ (see \cite[\text{Example 4.6}]{bualibanu2026reduction}) admits a Dirac deformation to the $K\times L$ Hamiltonian Poisson structure on $\nu^{-1}(\mathfrak{p})/U$ \cite[\text{\text{Example 2.32}}]{bualibanu2026reduction}. In particular, applying this to the internal fusion of the double $D(G)$ yields the multiplicative cotangent bundle $G \times_{U}P$ of the affine base space $G/U$ (\cite[\text{Example 4.6}]{bualibanu2026reduction}) that deforms to $T^{*}(G/U) = G\times_{U} \mathfrak{p}$ (\cite[\text{Example 2.32}]{bualibanu2026reduction}).
\end{Exp}

\begin{Exp}(Deformation of quasi-Hamiltonian implosion to symplectic implosion.) Let $G$ be a simply connected compact Lie group with maximal torus $T$. For each subset $J\subset \Delta$ of simple roots, let $S_J$ be the face of the dominant alcove such that $0\in \bar{S}_{J}$ and $\sigma_{J} \subset \t^{*}$ the face of the Weyl chamber $\t^{*}_{+}.$ By \cite[\text{Example 2.15}]{bualibanu2026reduction}, $(S_{J},\varphi_J)$ is a generalized reduction level of $G$ with stabilizer algebroid $S_{J} \rtimes [\g_{J},\g_{J}],$ and by \cite[\text{Example 2.14}]{bualibanu2026reduction}, $(\sigma_J,\varphi_{0,J})$ is a generalized reduction level of $\g^{*}$ with stabilizer algebroid $\sigma_{J} \rtimes [\g_{J},\g_{J}].$ In the deformation space $D(G,{1}),$ the face $S_{J}$ deforms to $\sigma_{J}$(\cite[\text{Lemma 4.8}]{maiza2026multiplicative}). We define $\S_J$ to be $(\S_J)_t = S_J, t\neq 0$ and $(\S_J)_0 = \sigma_J$. Then, $\S_J$ is smooth. Moreover, the stablizer subalgebroid $S_J \rtimes [\g_J,\g_J]$ is constant in $t$. Then, Theorem \ref{thm:main} applied stratum by stratum gives: If $\Phi: M\too G$ is a quasi-Hamiltonian moment map, then the quasi-Hamiltonian imploded cross section \[
M_{\imp} = \bigsqcup_{J} \Phi^{-1}(J)/[G_J,G_J],
\]  
admits a generalized Dirac deformation to the imploded cross-section of the symplectic manifold $N$ \[
N_{\imp} = \bigsqcup_{J} \nu^{-1}(\sigma_J)/[G_J,G_J].
\] 
This recovers \cite[\text{Theorem 5.3}]{maiza2026multiplicative} from a Dirac geometry perspective.
\end{Exp}

\begin{Exp}(Sevostyanov slices and Whittaker reduction). Let $G$ be a complex reductive semi-simple Lie group, $w\in W$ an element in the Weyl group and $\dot{w} \in Z_{G}(T)$ a lift. The Sevostyanov slice is $\Sigma_{w} = U^{+}\cdot w\cdot Z_{\dot{w}^{-1}}.$ By \cite[\text{Proposition 4.16}]{bualibanu2026reduction}, $\Sigma_w$ is a generalized reduction level with $\gamma = 0,$ $\varphi: \Sigma_w \too Z_{\dot{w}^{-1}}$ is the projection and trivial stabilizer.\\

The additive counterpart is the Slodowy slice $\Sigma_{w,0} = e_{-w} + \g^{e_w} \subset \g^{*}$ where $e_{-w} = \sum_{\alpha \in R^{+}\cap w(R^{-1})} e_{-\alpha}.$ One sets $\S_t = \Sigma_w, t\neq 0$ and $\S_0 = \Sigma_{w,0}.$ The deformation within $\G = \D(G,\{1\})$ is the content of Sevostyanov \cite{sevostyanov2011algebraic} and the BCH formula: the exponential parametrization $\xi \too e^{t\xi}$ maps diffeomorphically $\Sigma_{w,0}$ onto $\Sigma_{w}$ near $\dot{w}^{-1}$ for each $t\neq 0,$ and $Z_{\dot{w}^{-1}}$ deforms to $Z_0 = \g^{e_{-w}} \cap \t$ with all trivial data $\gamma = 0$ and the subalgebroid stabilizer is trivial. Now, consider the $U^{+}$ saturation \[ \Theta_{w} = U^{+}\cdot \Sigma_w \subset G,\] which is a larger generalized reduction level with stabilizer $\Theta_w \ltimes \mathfrak{u}^{+}.$ This deforms to \[
\Theta_{w,0} = \mathfrak{u}^{+} + \Sigma_{w,0} \subset \g^{*},
\]  
with stabilizer $\Theta_w \ltimes \mathfrak{u}^{+},$ of constant rank. Moreover, the non-trivial 2-form $\gamma_{\Theta_w}$ of \cite[Proposition 4.19]{bualibanu2026reduction} defined by \[
\gamma_{\Theta_w} = \gamma_g(x^R_1 + z^R_1 + y^L_1, x^R_2 + z^R_2 + y^L_2) = \dfrac{1}{2} \langle x_2 + z_2, \Ad_{g}y_1 \rangle_{\g} - \dfrac{1}{2} \langle x_1 + z_1, \Ad_{g}y_2 \rangle_{\g} 
\] 
deforms to $\gamma_{\Theta_{w,0}}$ defined by \[
\gamma_{\Theta_{w,0}} = \gamma_0(x_1 + y_1 + z_1, x_2 + y_2 + z_2) = \dfrac{1}{2} \langle x_1 + z_2, [y,y_1] \rangle_{\g} - \dfrac{1}{2} x_2 + z_2, [y,y_2] \langle \rangle_{\g},
\]
through the family $\widehat{\gamma}_t = \dfrac{\gamma_{e^{ty}}}{t}.$
Here, reducing along $\Sigma_w$ directly and reducing along $\Theta_w$ then quotienting by $U^{+}$ is the same \[
\Phi^{-1}(\Sigma_w) \cong \Phi^{-1}(\Theta_w)/U^{+}, 
\]
known as the Whittaker reduction. Then, Theorem \ref{thm:main} says that for any quasi-Poisson map $\Phi: M\too G,$ the $Z_{\dot{w}^{-1}}$ quasi-Poisson structure on $\Phi^{-1}(\Sigma_w) \cong \Phi^{-1}(\Theta_w)/U^{+}$ deforms to $\nu^{-1}(\Sigma_{w,0}) \cong \nu^{-1}(\Theta_{w,0})/U^{+}$ (Hamiltonian $Z_0$-space). 
\end{Exp}

\begin{Exp}(Deformation of a quasi-Poisson reduction by reducible group.) Let $H\subset G$ be a reducible subgroup (\cite[\text{Definition 4.1}]{bualibanu2026reduction}), i.e., $\h^{\perp}$ is also a Lie subalgebra and consider the projection $\varphi: H \too \Bar{H} = H/(H\cap H^{\perp}).$ By \cite[\text{Lemma 4.3}]{bualibanu2026reduction}, $(H,\varphi)$ is a generalized reduction level of $G$ with stabilizer subalgebroid $H \ltimes \h^{\perp}.$ Define $\S_t = H, t\neq 0$ and $\S_0 = \h.$ $\S$ is a smooth submanifold, parametrized by $\psi$-coordinates restricted to $\h \times \R.$ Moreover, $\gamma_H = 0$ so we consider $\widehat{\gamma}$ to be identically zero. Within the deformation space $\D(G,\{1\}),$ $H$ deforms to $\h$ and $\h^{\perp}$ to itself ( $\h^{\perp}$ depends only on the scalar product $\langle \rangle_{\g}.$). Finally, the projection $\varphi: H \too \Bar{H}$ deforms to the linear projection $\h \too \Bar{h} = \h/h\cap \h^{\perp}.$ Theorem \ref{thm:main} gives: For any quasi-Poisson moment map $(\mu,\Phi): M \too K\times G,$ the $K\times \Bar{H}$ quasi-Poisson structure on $\Phi^{-1}(H)/H^{\perp}$ \cite[\text{Proposition 4.4}]{bualibanu2026reduction} admits a Dirac deformation to the $K\times\Bar{h}$-Hamiltonian structure on $\nu^{-1}(\h)/\h^{\perp}.$
\end{Exp}

\begin{Exp}(Universal reduced space.) Let $\Gamma \rightrightarrows X$ be a pre-symplectic groupoid, $(S,\gamma,\varphi)$ a reduction level of $X$ and $\H$ a source-connected subgroupoid that integrates the stabilizer subalgebroid $A_{S,\gamma,\varphi}.$ By \cite[\text{Proposition 3.7}]{bualibanu2026reduction}, the reduction of $\Gamma$ at $(S,\gamma,\varphi)$ gives a universal reduced space $\Gamma_S$ such that any reduction of a strong Dirac map $\Phi: M\too X$ at $S$ is Dirac diffeomorphic to $(M\times_{X} \Gamma_S)/\Gamma.$ When $X = G$ with its Cartan-Dirac structure and $\Gamma = D(G),$ the universal reduced space is $D(G) = (G\times S)/\H$ (\cite[\text{Corollary 3.9}]{bualibanu2026reduction}). The deformation of $D(G)$ to the cotangent bundle $T^{*}G$ induces a deformation of $D(G)_S$ to $(T^{*}G)_{S_0},$ where $S_0$ is the additive reduction level. Moreover, the deformation is universal : $\Phi^{-1}(S)/\H \cong (M\times_{G} D(G)_S)/D(G)$ deforms to $\nu^{-1}(S_0)/\mathcal{H}_0 \cong (N\times_{\g^{*}} (T^{*}G)_{S_0})/T^{*}G.$ Here $\mathcal{H}_0$ is the source-connected subgroupoid that integrates $A_{S_0,\widehat{\gamma}_0,\varphi_0}.$ 
\end{Exp}

\section{Dirac deformation and infinitesimal Morita equivalence of Lie algebroids} \label{section 6}

In this section, we define a deformation theory for quasi-symplectic groupoids compatible with a Dirac deformation. Given a smooth family of Dirac manifolds $(\X,\K,\lambda_{\X},\pi_\X)$, we show that it integrates to a smooth Lie groupoid under specific conditions. Finally,
We prove that the Dirac deformation theory (Definition \ref{Dirac def}) is compatible with the infinitesimal Morita equivalence theory in the sense of \cite{xu2004momentum}. Throughout, $G$ denotes a compact Lie group with Lie algebra $\g$ equipped with an $\Ad$-invariant scalar product $\langle \cdot,\cdot \rangle_{\g}.$

\subsection{Quasi-symplectic groupoids.} We recall the definitions following \cite[\text{Section 1.1}]{bualibanu2026reduction} and \cite{xu2004momentum,bursztyn2004integration}.

\begin{Def}
    \textit{A quasi-symplectic groupoid} (or a twisted presymplectic groupoid) over an $\eta$-twisted Dirac manifold $(X,L_X)$ is a Lie groupoid $(s,t): \Gamma \rightrightarrows X$ together with a multiplicative two form $\Omega \in \Omega^{2}(\Gamma)$ such that: \begin{enumerate}
        \item $d\Omega = s^{*}\eta - t^{*}\eta.$
        \item $\dim(\Gamma) = 2\dim(X).$
        \item $\Ker (\Gamma) \cap \Ker (ds) \cap \Ker (dt) = 0.$
    \end{enumerate}
    The Dirac structure $L_X$ is recovered from $\Gamma$ as the image of the map \[
    (\rho,\varphi): \Lie(\Gamma)\too TX \oplus T^{*}X, \quad a \too (\rho(a),\Omega^{\flat}(a)|TX),
    \]
    where $\rho: \Lie(\Gamma) \too TX$ is the anchor map. The induced map $\Lie(\Gamma)\too L_X$ is an isomorphism of Lie algebroids.
\end{Def}

\begin{Exp}(\cite[\text{Example 1.4}]{bualibanu2026reduction})
The Cartan Dirac structure $L_G = \{((\xi^{L} - \xi^{R}),\sigma(\xi))\}$ on $G$ is integrated by the double $D(G) = G\times G \rightrightarrows G$ with source $s(a,b) = \Ad_{a}b$ and $t(a,b) = b^{-1},$ and presymplectic form \[
\omega_{D(G)} = \dfrac{1}{2} \Big\{ \langle \Ad_{b}a^{*}\theta^{L} \wedge \theta^{L} \rangle + \langle a^{*}\theta^{L} \wedge (b^{*}\theta^{L} + b^{*}\theta^{R}) \rangle \Big\}
\]
\end{Exp}

\begin{Exp}
The Kirillov-Kostant-Souriau Poisson structure $\pi_{\text{KKS}}$ on $\g^{*}$ is integrated by the symplectic groupoid $T^{*}G \rightrightarrows \g^{*}$ with the canonical symplectic form $\omega_{T^{*}G}$ and the source/target map 
\[
s(a,\xi) = \Ad_{a}^{*}\xi, \quad t(a,\xi) = -\xi,    
\] 
where we use left trivialization $T^{*}G \cong G\times \g^{*}.$
\end{Exp}

\subsection{Hamiltonian groupoid spaces}
\begin{Def}(\cite{xu2004momentum},\cite[\text{Section 5.1}]{bualibanu2026reduction}) \label{def:ham} 
Let $\Gamma \rightrightarrows X$ be a quasi-symplectic groupoid with presymplectic form $\Omega.$ \textit{A Hamiltonian} $\Gamma$-\textit{space} is a non-degenerate Dirac manifold $(M,L_{\omega})$ with a $\Gamma-$action along $\Phi: M\too X$ such that: \begin{enumerate}
    \item The graph $\{(g,m,g\cdot m): s(g) = \Phi(m)\}$ is isotropic with respect to $(\Omega,\omega,-\omega).$ 
    \item $d\omega = -\Phi^{*}\eta.$
    \item $\Ker(\Phi_{*}) \cap \Ker(\omega) = 0.$
\end{enumerate}
Conditions (2) and (3) state precisely that $\Phi$ is a strong Dirac map.    
\end{Def}

\begin{Exp}
For $\G = T^{*}G \rightrightarrows \g^{*},$ a $\G$-Hamiltonian space is the same as a $G$-Hamiltonian space. When $\G = D(G) \rightrightarrows G$, a $\G$-Hamiltonian space is the same as a quasi-Hamiltonian $G$-space.
\end{Exp}

\subsection{Deformation of quasi-symplectic groupoids}
Let $(\X,\K,\lambda_{X},\pi_{\X},)$ be a smooth family of twisted Dirac structures over $\R$ (Definition \ref{Dirac def}), so that $(\X_t,\K_t,\lambda_{\X,t})$ is a twisted Dirac manifold for $t\neq 0$ and a Poisson manifold for $t = 0.$

\begin{Def} \label{def:family}
    \textit{A deformation of quasi-symplectic groupoids} compatible with $(\X,\K,\lambda_{X})$ is a smooth Lie groupoid $\mathbf{\Gamma} \rightrightarrows \X$ over $\R$(i.e., the structure maps commute with $\pi_\X$) together with a smooth fiberwise multiplicative two-form $\mathbf{\Omega}$ on $\mathbf{G}$ such that \begin{enumerate}
        \item For each $t\neq 0,$ the fiber $\mathbf{\Gamma}_t \rightrightarrows \X_t$ is a quasi-symplectic groupoid integrating $(\X_t,\K_t).$
        \item At $t = 0,$ $(\mathbf{\Gamma}_0\rightrightarrows \X_0)$ is a symplectic groupoid integrating the Poisson structure $(\X_0,\K_0).$
        \item The Lie algebroid $\Lie(\Gamma_t)$ maps diffeomorphically to $\K_t$ via $(\rho_t,\varphi_t)$ for each $t,$ and these isomorphisms vary smoothly in $t$
    \end{enumerate}
\end{Def}

The following theorem shows that a smooth family of Dirac manifolds $(\X,\K,\lambda_\X,\pi_{\X})$ integrates to a smooth Lie groupoid $\mathbf{G}$ with respect to the following hypothesis, called \textbf{The uniform integrability.}\\

\noindent\textbf{Hypothesis}(Uniform integrability). \label{Uniform inte} For each $t\in \R,$ $\K_t \too \X_t$ is integrable in the sense of Crainic-Fernandes, with monodromy subgroups $N_{x}(\K_t)\subset \g_{x,t},$ where $\g_{x,t} = \Ker(\rho_{\K})_x  \subset \K_{x,t}$ is the isotropy Lie algebra at $x\in \X_t.$\\
The family $(\N_{x}(\K_t))_{(x,t)\in \X }$ is locally uniformly discrete in the following sense: for any $(x_0,t_0) \in \X,$ there exists an open neighborhood $U\subset \X$ of $(x_0,t_0),$ a continuous family of norms $\|\|_{x,t}$ over $U$ (which exists because $\K$ is a smooth vector bundle over $\X$), and a constant $\epsilon$ such that $N_{x}(\K_t) \cap B_{\epsilon}(0) = \{0\}, \forall (x,t)\in U.$ The intersection with a ball in the larger space $\K_{x,t}$ is well-defined even when the isotropy $\g_{x,t}$ jumps in dimension because $N_{x}(\K_t) \subset \g_{x,t} \subset \K_{x,t}.$

\begin{Lem}[Uniform Lie-algebroid exponential]\label{lem:uniform-exp}
Assume the Uniform integrability hypothesis. Fix $(x_0,t_0)\in \X$ and a vertical trivialization $\Phi\colon U\xrightarrow{\ \sim\ }V_0\times I$ of a
neighborhood of $(x_0,t_0)$, with $V_0\subset \X_{t_0}$ open and
$I\subset\mathbb{R}$ an open interval. For every compact $C\subset V_0$ there exist a constant $\delta > 0$ and an open zero-section neighborhood
\[
  \mathcal{O}_C:=\bigl\{(x,t,\xi)\in \K|_{C\times I}:\|\xi\|_{x,t}<\delta\bigr\}
\]
such that:
\begin{itemize}
\item[(i)] for each $t\in I$ the fiberwise exponential
  $\expt\colon \O_C\cap\Kt\to \mathbf{G}_t$ is defined, as the time-$1$ evaluation of the $\Kt$-path issuing from $\xi$;
\item[(ii)] the assembled map $\exp\colon \O_C \to \mathbf{G}$,
  $(x,t,\xi)\mapsto\expt(\xi)$, is smooth in $(x,t,\xi)$;
\item[(iii)] for each $t\in I$, $\expt$ is a diffeomorphism onto an open neighborhood of $\varepsilon_t(C)$ in $\mathbf{G}_t$, and the radius $\delta$ is uniform in $t\in I$.
\end{itemize}
\end{Lem}

\begin{proof}
We treat the three assertions in turn; (i) and (ii) are the routine half and (iii) is where the hypothesis is used.
 
\textbf{Existence and smoothness of the flow (i)--(ii).}
Under $\Phi$ the algebroid $K|_U$ pulls back to a smooth family
$\{\Kt|_{V_0}\}_{t\in I}$ of Lagrangian sub-bundles of $TV_0$, and the anchor $\rho_{\K_t}$, the Dorfman bracket $[\cdot,\cdot]_t$, and the twist $\lambda_{\X,t}$ all vary smoothly in $(x,t)$. By definition the value $\expt(\xi)$ is the endpoint at time $1$ of the unique $\Kt$-path $a\colon[0,1]\to\Kt$ with $a(0)$ determined by $\xi$; concretely $a$ solves the
algebroid ODE
\begin{equation}\label{eq:algebroid-ode}
  \frac{d}{ds}\gamma(s)=\rho_{\K_t}\bigl(a(s)\bigr),\qquad
  \frac{D}{ds}a(s)=0,\qquad a(0)=\xi,
\end{equation}
where $\gamma = \pi\circ a$ is the base path and $D/ds$ is the $\Kt$-path derivative determined by $[\cdot,\cdot]_t$. The right-hand side of \eqref{eq:algebroid-ode} is a smooth function of $(x,t,\xi,s)$, since all the structure data are. By the smooth-dependence theorem for solutions of ODEs on
parameters and initial conditions, there is a zero-section neighborhood on which the time-$1$ solution exists and depends smoothly on $(x,t,\xi)$; over the
compact $C\times I$ (with $I$ shrunk to a relatively compact subinterval if necessary) this neighborhood contains a set of the form $\{\|\xi\|_{x,t}<\delta_0\}$ for some $\delta_0>0$. On the quotient $\mathbf{G}_t = P(\Kt)/\!\sim_t$ the endpoint descends to $\expt(\xi)$, and the descent is smooth because the equivalence is generated by the smooth $\Kt$-homotopies of
\cite{crainic2003integrability}. This gives (i) and the smoothness (ii), on $\{\|\xi\|_{x,t}<\delta_0\}$.\\ 

It remains to upgrade local injectivity to genuine injectivity on a
$t$-uniform ball, and this is the only step that uses the hypothesis. We do \emph{not} argue by subtracting fiber elements: because $\dim \g_{x,t}$ jumps across $t=0$, the isotropy does not sit inside $\Kt$ as a smooth sub-bundle near the central fiber, and there is no $t$-uniform splitting
$\Kt \cong \mathfrak{g}_{x,t}\oplus(\text{transverse})$ to subtract along.
Instead we follow the groupoid-multiplication argument of
Crainic--Fernandes, which localizes the coincidence of two exponentials to a single isotropy fiber \emph{without} choosing any such splitting; the uniform-discreteness hypothesis is then applied there.

Fix $t \in I$ and suppose
\[
  \exp_t(\xi) = \exp_t(\xi'), \qquad
  \|\xi\|_{x,t},\ \|\xi'\|_{x',t} < \delta_1 ,
\]
with $\xi \in \K_t|_x$ and $\xi' \in \K_t|_{x'}$.

\smallskip
\noindent\textbf{Step 1 (the base points coincide).}
The source and target of $\exp_t(\xi)$ in $G_t$ are $x$ and the endpoint of the base path of the $\K_t$-geodesic issuing from $\xi$. Writing
$g := \exp_t(\xi)\,\exp_t(\xi')^{-1} \in \mathbf{G}_t$, the hypothesis $\exp_t(\xi)=\exp_t(\xi')$ forces $g = \varepsilon_t(\text{pt})$ to be a unit, so in particular $x=x'$ and the two geodesics share their endpoint. We may
therefore assume $\xi,\xi' \in \K_t|_x$ over the common base point $x$, and $g$ lies in the isotropy $\mathbf{G}_t(x) := s^{-1}(x)\cap t^{-1}(x)$.

\smallskip
\noindent\textbf{Step 2 (the discrepancy lies over the isotropy Lie algebra).}
Consider instead the composite $\K_t$-path
$a := \exp_t(\xi)\cdot\exp_t(\xi')^{-1}$ representing $g$ before passing to $\mathbf{G}_t$; its class $[a]=g$ is a unit, hence $a$ is $\K_t$-homotopic to the trivial $\K_t$-path $0_x$. In particular the base path of $a$ is a loop at $x$.
We claim that, on a small enough $t$-uniform ball, the only $\K_t$-geodesic with closed base path is the constant one. In local coordinates over the compact $C\times I$ the geodesic system reads
\[
  \dot\gamma^i \;=\; \sum_p b^i_p(\gamma)\,a^p, \qquad \dot a^p = 0 ,
\]
so along a geodesic the fiber component $a$ is constant and the base path is an integral curve of the vector field $X_a = \sum_{i,p} b^i_p(\gamma)\,a^p\,
\partial_{x^i}$. By the period-bounding lemma \cite[Lemma]{yorke1969periods}, any
nonconstant periodic orbit of $X_a$ with data in $C\times I$ has period
\[
  T \;\ge\; \frac{2\pi}{M_{C\times I}(a)}, \qquad
  M_{C\times I}(a) \;:=\;
  \sup_{\substack{x\in C,\ t\in I \\ 1\le j,k\le \operatorname{rk}K}}
  \Bigl|\sum_p \tfrac{\partial b^j_p}{\partial x^k}(x,t)\,a^p\Bigr| .
\]
Because the family $(\K_t\to \X_t)_{t\in I}$ and its anchor vary smoothly in $(x,t)$, the coefficients $b^j_p$ and their first $x$-derivatives are bounded uniformly in $t$ on the compact $C\times I$; hence $M_{C\times I}(a)\to 0$ linearly as $\|a\|_{x,t}\to 0$, uniformly in $(x,t)$. Choose $\delta_1'>0$ so small that
\[
  \|a\|_{x,t} < \delta_1' \ \Longrightarrow\ M_{C\times I}(a) < 2\pi
  \qquad\text{for all } (x,t)\in C\times I ,
\]
which is possible by that uniform linear bound. Then every nonconstant periodic base orbit issuing from $\{\|\xi\|_{x,t}<\delta_1'\}$ would have period $T > 1$, so no $\K_t$-geodesic in this ball has a nonconstant closed base path of length $\le 1$; the base loop of $a$ is therefore constant.

Consequently $a$ lies over the single point $x$, its initial datum is tangent to the isotropy, and $a$ is $\K_t$-homotopic to a constant $\mathfrak{g}_{x,t}$-path $n \in \mathfrak{g}_{x,t}=\ker(\rho_{\K_t})_x$. Replace $\delta_1$ by
$\min\{\delta_1,\delta_1'\}$, still called $\delta_1$.

\smallskip
\noindent\textbf{Step 3 (coset description and discreteness).}
Over the single fiber $\mathfrak{g}_{x,t}$ the exponential $\exp_t$ restricts to the ordinary Lie-group exponential of the isotropy, and the fiber of $\exp_t$ over the unit is exactly the monodromy coset: the constant
$\mathfrak{g}_{x,t}$-path $n$ is $\K_t$-homotopic to the trivial path iff
\[
  n \in \Nx \subset Z(\mathfrak{g}_{x,t}) \subset \mathfrak{g}_{x,t}
      \subset \K_t|_x ,
\]
by Definition 3.1 and Lemma 3.3 of \cite{crainic2003integrability}, which identify $\Nx$ with the
subgroup $\exp^{-1}(\widetilde N_x)$ of $Z(\mathfrak{g}_{x,t})$. We stress the distinction between the two ambient spaces at play here. The monodromy group $\Nx$ is intrinsically a subgroup of the \emph{center} of the isotropy Lie algebra, and it is there that the Crainic--Fernandes discreteness obstruction
is formulated; the Uniform integrability hypothesis, by contrast, measures $\|n\|_{x,t}$ and enforces the exclusion in the \emph{ambient} fiber $\K_t|_x$.

These are consistent: since $\Nx\subset Z(\mathfrak{g}_{x,t})\subset\Kt|_x$, the
ambient ball $B_\varepsilon(0)\subset\K_t|_x$ restricts to a ball in $Z(\mathfrak{g}_{x,t})$, and $\Nx\cap B_\varepsilon(0)=\{0\}$ in $\K_t|_x$ is therefore \emph{a priori} at least as strong as discreteness of $\Nx$ inside
$Z(\mathfrak{g}_{x,t})$. The hypothesis thus does not redefine which group $\Nx$ is; it only requires the ambient norm to detect its discreteness uniformly, which is exactly what a fixed sub-bundle norm on $\K_t$ cannot be assumed to do intrinsically once $\dim\mathfrak{g}_{x,t}$ varies.

This is where genuine injectivity is bought. By construction $n$ is
represented in the ambient fiber $\Kt|_x$ with
\[
  \|n\|_{x,t} \;\le\; \|\xi\|_{x,t} + \|\xi'\|_{x,t} \;<\; 2\delta_1 .
\]
Now set
\[
  \delta \;:=\; \tfrac12\,\min\{\delta_0,\ \delta_1,\ \varepsilon\},
\]
with $\varepsilon$ the uniform-discreteness constant of the hypothesis on
$U \supset C\times I$ (shrinking $I$ so that $C\times I\subset U$). If both
$\|\xi\|_{x,t},\|\xi'\|_{x,t} < \delta$ then $\|n\|_{x,t} < 2\delta \le
\varepsilon$, so
\[
  n \in \Nx \cap B_\varepsilon(0) = \{0\}
\]
by uniform discreteness. Hence $n=0$, the constant path is trivial, and therefore $\xi=\xi'$.

\smallskip
This proves $\exp_t$ is injective on $\{\|\xi\|_{x,t}<\delta\}$ for every $t\in I$. A local diffeomorphism that is injective is a diffeomorphism onto its open image, which contains $\varepsilon_t(C)$; this is (iii), with $\delta$ manifestly independent of $t$.
\end{proof}

\begin{Rem}
The crucial point is that the ball $B_\varepsilon(0)$ in the hypothesis is taken in the ambient fiber $\K_t|_x$, not in $\mathfrak{g}_{x,t}$. Step 3 produces $n$ as an element of $\mathfrak{g}_{x,t}$, but its norm is measured
and its exclusion enforced in $\K_t|_x$; this is what keeps the exclusion $\Nx \cap B_\varepsilon(0) = \{0\}$ meaningful and $t$-uniform across values
of $t$ where $\dim\mathfrak{g}_{x,t}$ jumps, and in particular across the central fiber $t=0$, where $\K_0$ is Poisson and $\mathfrak{g}_{x,0}$ may drop
in dimension. Note that no fiberwise splitting off $\mathfrak{g}_{x,t}$ was
used at any stage: the localization to the isotropy in Step 2 is effected by the period bound, not by a projection, which is exactly what allows the argument to survive the dimension jump.
\end{Rem}

\begin{Thm} \label{groupoid integra}
Let $(\X,\K,\pi_{\X})$ be a smooth family of twisted Dirac manifolds satisfying Hypothesis \ref{Uniform inte}. Then, there exists a smooth Lie groupoid $\mathbf{G} \rightrightarrows \X$ over $\R$ and a smooth section $\Omega$ of $\bigwedge^{2} (\Ker(d\pi_{\mathbf{G}})^{*})$ such that \begin{enumerate}
    \item $(\mathbf{G}_t,\Omega_t)$ is a quasi-symplectic groupoid integrating $(\K_t,\lambda_{\X,t})$
    \item $\Lie(\mathbf{G}_t) \cong \K_t$ varies smoothly in $t.$
    \item For $t = 0,$ $(\mathbf{G}_0,\omega_0)$ is the symplectic groupoid integrating the Poisson structure $\K_0.$
\end{enumerate}
$(\mathbf{G},\Omega)$ is unique up to an isomorphism to a smooth family of source-simply-connected quasi-symplectic groupoid.
\end{Thm}

We prove the theorem through a sequence of steps and lemmas.\\

\noindent\textbf{Part 1: Smooth assembly of the groupoid family}
\begin{Prop} \label{integrability prop}
Under Hypothesis \ref{Uniform inte}, the Weinstein groupoid $\mathbf{G}_t = P(\K_t)/\sim_t$ assemble into a smooth Lie groupoid $\mathbf{G} \rightrightarrows \X$ over $\R,$ with $\Lie(\mathbf{G}_t) \cong \K_t,$ for every $t.$    
\end{Prop}

\begin{proof}[Proof of Proposition \ref{integrability prop}]
By \cite[Theorem 4.1]{crainic2003integrability}, for each $t$ the Weinstein groupoid $\mathcal{G}_t = P(\K_t)/\!\sim_t$ is a finite-dimensional source-simply-connected Lie groupoid, since $\K_t$ is integrable by Hypothesis \ref{Uniform inte}.
Write $\epsilon_t : \X_t \to \mathcal{G}_t$ for the unit section and set $\mathbf{G} = \bigsqcup_{t \in \mathbb{R}} \mathcal{G}_t$. We construct a smooth structure on $\mathbf{G}$ in two steps: first a chart around the unit bisection
via the fiberwise Lie algebroid exponential of Lemma \ref{lem:uniform-exp}; then a global atlas via right multiplication by local bisections.

\medskip
\noindent\textbf{Step 1: A local chart near the unit bisection.}
Fix $(x_0, t_0) \in \X$ and a vertical trivialization
$\Phi : U \xrightarrow{\ \sim\ } V_0 \times I$ of a neighborhood of $(x_0, t_0)$, with $V_0 \subset \X_{t_0}$ open and $I \subset \mathbb{R}$ an open interval.
Because $\K \subset T\X$ is a smooth Lagrangian subbundle of the
fiberwise generalized tangent bundle (part of the Dirac deformation data of Definition \ref{Dirac def}), under $\Phi$ it pulls back to a smooth
family $\{\K_t|_{V_0}\}_{t \in I}$ of Lagrangian subbundles of $\mathbb{T}V_0$; concretely, choosing a smooth local frame of $\K$ over $U$ and applying $d\Phi$ yields, for each $t$, a frame of $\K_t|_{V_0}$ depending smoothly on $(x,t)$, and the anchor $\rho_{\K_t}$, the Dorfman bracket $[\cdot,\cdot]_t$, and
the twist $\lambda_{\X,t}$ all vary smoothly in $(x,t)$. This is exactly the input required by Lemma \ref{lem:uniform-exp}.

Applying Lemma \ref{lem:uniform-exp} to the family $\K \to \X$
over $U$: for every compact $C \subset V_0$ there exist $\delta > 0$ and an open zero-section neighborhood
\[
  O_C := \bigl\{(x,t,\xi) \in \K|_{C \times I} : \|\xi\|_{x,t} < \delta \bigr\}
\]
such that the assembled map $\exp : O_C \to \mathbf{G}$, $(x,t,\xi) \mapsto \exp_t(\xi)$,
is smooth in $(x,t,\xi)$ (Lemma \ref{lem:uniform-exp}(i)--(ii)), and for each $t \in I$ the fiberwise map $\exp_t : O_C \cap \K_t \to \mathbf{G}_t$ is a diffeomorphism onto an open neighborhood of $\epsilon_t(C)$, with radius $\delta$ uniform in $t$ (Lemma \ref{lem:uniform-exp}(iii)). This produces a smooth
chart on a neighborhood of the unit bisection, valid uniformly across $t = 0$, where $\dim \mathfrak{g}_{x,t}$ may jump.

\medskip
\noindent\textbf{Step 2: Propagation by right-translation via local bisections.}
Let $g_0 \in \mathbf{G}_{t_0}$ with $s(g_0) = x_0^s$ and $t(g_0) = x_0^t$. Since $\mathbf{G}_{t_0}$ is source-simply-connected and $s_{t_0}$ is a surjective submersion, there is a local bisection of $\mathbf{G}_{t_0}$ through $g_0$: an open neighborhood $W_0 \subset \X_{t_0}$ of $x_0^s$ and a smooth section
$\sigma_0 : W_0 \to \mathbf{G}_{t_0}$ of $s_{t_0}$ with $\sigma_0(x_0^s) = g_0$, such that $t_{t_0} \circ \sigma_0 : W_0 \to \X_{t_0}$ is an open embedding. By the implicit function theorem and the uniform smoothness of the groupoid operations
near the unit bisection (Step 1), $\sigma_0$ extends to a smooth local bisection
\[
  \sigma : W \longrightarrow \mathbf{G}
\]
over an open neighborhood $W \subset \X$ of $(x_0^s, t_0)$, whose
restriction to each fiber $W \cap \X_t$ is a local bisection of $\mathbf{G}_t$ for $t$ sufficiently close to $t_0$.

Right-translation by $\sigma$ defines a fiberwise diffeomorphism
\begin{equation}
  R_\sigma : t^{-1}(W) \longrightarrow t^{-1}\bigl(t(\sigma(W))\bigr),
  \qquad h \longmapsto \sigma(t(h)) \cdot h.
\end{equation}
For the product $\sigma(t(h)) \cdot h$ to be defined we need
$s(\sigma(t(h))) = t(h)$, which holds because $\sigma$ is a section of $s$: setting $x = t(h)$ in $s(\sigma(x)) = x$ gives the requirement. The target of the product
is $t(\sigma(t(h)))$, which lies in $t(\sigma(W))$ as claimed. The inverse of $R_\sigma$ is left-translation by the pointwise inverse $\sigma^{-1}$, so $R_\sigma$
is a diffeomorphism.

The map $R_\sigma$ sends a neighborhood of the unit bisection $\epsilon(W)$ (where
$h = \epsilon(t(h))$, so $\sigma(t(h)) \cdot \epsilon(t(h)) = \sigma(t(h))$) to a
neighborhood of $\sigma(W) \subset \mathcal{G}$, and in particular
$R_\sigma(\epsilon(x_0^t)) = \sigma(x_0^t) = g_0$. Pulling back the Step 1 chart by $R_\sigma^{-1}$ gives a smooth chart on a neighborhood of $g_0$.

Since local bisections through any two points of $\mathbf{G}_t$ can be concatenated (by source-simple-connectivity), this procedure covers $\mathbf{G}$ by smooth charts. Compatibility on overlaps follows from smoothness of multiplication on the
unit-bisection chart combined with uniqueness of the source-simply-connected integration.

\medskip
\noindent\textbf{Smoothness of the structure maps.}
The unit section $\epsilon : \X \to \mathbf{G}$ is smooth by construction.
In the unit-bisection chart, $s$ and $t$ are smooth because $t \circ \exp$ is the base exponential composed with the anchor, and $s \circ \exp$ is its source analogue, both smooth by Step 1; in a right-translated chart they are compositions
with the diffeomorphism $R_\sigma^{\pm 1}$. Multiplication and inversion are smooth in the unit-bisection chart by smoothness of the fiberwise exponential and propagate to the global atlas via the same right-translation argument of Step 2.

\medskip
\noindent\textbf{Smooth Lie algebroid isomorphism.}
The identification $\mathrm{Lie}(\mathbf{G}_t) \cong \K_t$ from
\cite[Theorem 4.1]{crainic2003integrability} is realized by the anchor-exponential correspondence of Step 1. Since the exponential family is smooth in $t$
(Lemma \ref{lem:uniform-exp}(ii)), so is the identification.
\end{proof}

\noindent\textbf{Part 2: The smooth IM-two form.} Recall from \cite{bursztyn2012multiplicative,bursztyn2004integration} that \textit{an infinitesimally multiplicative} (IM) 2-form on a Lie algebroid $A\too M$ with a $\eta$-twist (a closed 3-form on $M$) is a bundle map $\mu: A\too T^{*}M$ such that for any $a,b \in \Gamma(A)$ \begin{enumerate}
    \item $\langle \mu(a),\rho(b) \rangle = -\langle \mu(b),\rho(a) \rangle.$ \label{condition 1}
    \item $\mu([a,b]_A) = \L_{\rho(a)}\mu(b) - i_{\rho(b)}d\mu(a) + i_{\rho(a)\wedge \rho(b)}\eta.$ \label{condition 2}
\end{enumerate}

The sign convention in \ref{condition 2} is that of \cite[\text{Eq. (2.6)}]{bursztyn2004integration}. For a $\eta$-twisted Dirac structure $L \subset \mathbb{T}M$, viewed as a Lie algebroid with respect to the restriction of the Dorfman bracket, the map defined by $\mu_L: L \too T^{*}M, \quad (v,\alpha) \too \alpha$ is an IM 2-form with twist $\eta$.

\begin{Lem} \label{IM 2 form}
    The map \[
    \mathbf{\mu}: \K \too (T^{\text{vert}}X)^{*}, \quad (v,\alpha) \too \alpha,
    \]
    
obtained by restricting the projection $\mathbb{T}^{\text{vert}}\X \too (T^{\text{vert}}\X)^{*}$ to $\K$, is a smooth bundle map over $\X.$ For each $t$, its fiberwise restriction $\mu_t: \K_t \too T^{*}\X_t$ is an IM 2-form for the Lie algebroid $\K_t \too \X_t$ with twist $\lambda_{\X,t},$ integrated by the quasi-symplectic groupoid $\mathbf{G}_t \rightrightarrows \X_t.$
\end{Lem}

\begin{proof}
Smoothness of $\mathbf{\mu}$ is immediate: $\K$ is a smooth sub-bundle of $\mathbb{T}\X$ by the definition of a Dirac deformation, and the projection is smooth. The fiberwise identification is \cite{bursztyn2004integration}: $\mu_t$ satisfies \ref{condition 1} and \ref{condition 2} with $\rho = \rho_{\K_t},$ $[,]_{A}$ is the Dorfman bracket on $\K_t$ and $\eta = \lambda_{\X,t}.$
\end{proof}

\noindent\textbf{Part 3: smooth integration of $\mathbf{\mu}$ to $\Omega$.}
The key step is the smooth dependence of the integrated multiplicative two form on the IM-data. 

\begin{Prop}\label{prop:Omega-smooth}
There exists a unique smooth fiberwise multiplicative 2-form
$\boldsymbol{\Omega} \in \Omega^2_{\mathrm{vert}}(\mathbf{G})$ such that, for
each $t$, $\Omega_t := \boldsymbol{\Omega}|_{\mathbf{G}_t}$ is the
multiplicative 2-form on $\mathbf{G}_t$ associated by
\cite{bursztyn2004integration} to $\mu_t$. It satisfies the fiberwise identity
\begin{equation}\label{eq:twist-fiberwise}
d\Omega_t = \mathbf{s}_t^* \lambda_{\mathcal{X}, t}
- \mathbf{t}_t^* \lambda_{\mathcal{X}, t}
\quad \text{for every } t \in \R,
\end{equation}
which assembles into the identity
$d\boldsymbol{\Omega}
= \mathbf{s}^* \lambda_{\mathcal{X}} - \mathbf{t}^* \lambda_{\mathcal{X}}$
on $\mathbf{G}$, and the non-degeneracy condition
\begin{equation} \label{non-degen}
\ker \Omega_t \cap \ker d\mathbf{s}_t \cap \ker d\mathbf{t}_t = 0
\quad \text{for every } t \in \R.
\end{equation}
\end{Prop}

\begin{proof}[Proof of Proposition \ref{prop:Omega-smooth}]
For each $t$, \cite{bursztyn2004integration} applied to the IM 2-form $\mu_t$ on $(\K_t, \lambda_{\X,t})$---integrable by
Proposition \ref{integrability prop}---produces a unique multiplicative 2-form $\Omega_t$ on $\mathbf{G}_t$ satisfying the twist and non-degeneracy conditions \eqref{eq:twist-fiberwise}, \eqref{non-degen}.

It remains to show that the family $(\Omega_t)_{t \in \mathbb{R}}$ assembles into a multiplicative section $\Omega$ of $\bigwedge^2 (\ker d\pi_{\mathbf{G}})^*$ depending smoothly on $t$ through the atlas of Proposition \ref{integrability prop}. For $g \in \mathbf{G}$ there is a unique $t$ with $g \in \mathbf{G}_t$; define
\[
  \Omega_g : T_g\mathbf{G} \times T_g\mathbf{G} \longrightarrow \mathbb{R},
  \qquad \Omega_g(u_1, u_2) := \Omega_t(g)(u_1, u_2),
\]
where $T_g\mathbf{G} \cong T_g\mathbf{G}_t$.

\medskip
\noindent\textbf{Step 1: Smoothness on the unit bisection.}
At a point $\epsilon_t(x) \in \mathbf{G}_t$, the tangent space splits as $T_{\epsilon_t(x)}\mathbf{G}_t = T_x \X_t \oplus \K_{t,x}$ via
$v \mapsto (ds_t\, v,\ v - ds_t\, v)$. By \cite[Eq. (3.14)]{bursztyn2004integration}, in this splitting
\begin{equation}
  \Omega_{\epsilon_t(x)}\bigl((X,\alpha),(Y,\beta)\bigr)
  = \mu_t(\alpha)(Y) - \mu_t(\beta)(X) + \langle \mu_t(\alpha), \rho_{K_t}(\beta) \rangle,
\end{equation}
for $X, Y \in T_x \X_t$ and $\alpha, \beta \in \K_{t,x}$. The right-hand side is
smooth in $(x, t, X, Y, \alpha, \beta)$:
\begin{itemize}
  \item $\mu$ is smooth by Lemma \ref{IM 2 form};
  \item $\rho_{\K_t}$ is smooth in $t$ by the definition of a Dirac deformation (the anchor is the restriction to $\K$ of the smooth projection $\mathbb{T}\X \to T\X$);
  \item the splitting $T\mathbf{G}_t = T\X_t \oplus \K_t$ at the unit bisection is smooth in $t$ by Proposition~\ref{integrability prop}.
\end{itemize}
Therefore $\Omega$ restricted to the unit bisection $\epsilon(\mathcal{X})$ is a
smooth fiberwise 2-form.

\medskip
\noindent\textbf{Step 2: Smoothness on the rest of $\mathbf{G}$ via right-translation.}
Let $g \in \mathbf{G}_t$ with $t_t(g) = y$. By the right-translation reduction
\cite[Eq.(3.4)]{bursztyn2004integration}, for
$\alpha_y \in \ker(ds_t)_y$ and $v_g \in T_g\mathbf{G}_t$,
\begin{equation}
  \Omega_t(g)\bigl((dR_g)_y(\alpha_y), v_g\bigr)
  = \Omega_t(y)\bigl(\alpha_y, (dt_t)_g v_g\bigr).
\end{equation}
The right-hand side is the value of $\Omega_t$ at the unit corresponding to the source of $g$ via the unit bisection, which is smooth by Step 1. Finally,
$T_g\mathbf{G}_t \cong dR_g(T_y\mathbf{G}_t) = dR_g(T_x\X \oplus \K_{t,x})$, so for any $u_1, u_2 \in T_g\mathbf{G}_t$ there exist $X_i \in T_x \X_t$ and $\alpha_{i,y} \in \K_{x,t}$ with $u_i = dR_g(X_i + \alpha_{i,y})$, $i \in \{1,2\}$.
Then
\[
  \Omega_t(g)(u_1, u_2)
  = \Omega_t(g)\bigl(dR_g(X_1), \alpha_{2,y}\bigr)
  + \Omega_t(g)\bigl(\alpha_{1,y}, dR_g(X_2)\bigr)
  + \Omega_t(g)\bigl(\alpha_{1,y}, \alpha_{2,y}\bigr)
  + \Omega_t(g)\bigl(dR_g(X_1), dR_g(X_2)\bigr).
\]
The first three terms are smooth by the argument above. For the last term,
\[
  \Omega_t(g)\bigl((dR_g)_y(X_1), (dR_g)_y(X_2)\bigr) = (\sigma^* \Omega_t)_y(X_1, X_2),
\]
and $\sigma^* \Omega_t$ varies smoothly in $(y, t)$ because $\Omega_t$ is already smooth on the unit bisection and the bisection $\sigma$ depends smoothly on
$(g, t)$ by Proposition \ref{integrability prop} (Step 2 of its proof).
\end{proof}

Now, we are ready to give the proof of Theorem \ref{groupoid integra}.

\begin{proof}[Proof of Theorem \ref{groupoid integra}]
The first two points of Theorem \ref{groupoid integra} are a direct consequence of Proposition \ref{integrability prop}, Lemma \ref{IM 2 form} and Proposition \ref{prop:Omega-smooth}. When $t = 0,$ $\lambda_{X,0} = 0$ so $d\mathbf{\Omega}_0 = 0$ and $\K_0$ is a Poisson structure.  By \cite[Theorem 2.5]{bursztyn2004integration}, The IM map $\mu_0: \K_0 \too T^{*}\X_0$ is the cotangent algebroid identification, hence an isomorphism, and the unit bisection formula \cite[Eq. (3.14)]{bursztyn2004integration} together with the right translation shows that $\Omega_0$ is non-degenerate as a 2-form. Hence, $(\mathbf{G}_0,\mathbf{\Omega}_0)$ is the symplectic groupoid integrating the Poisson structure $\K_0.$ 
\end{proof}

\subsection{The fundamental example}
\begin{Prop} \label{prop:fund}
    The deformation of the Cartan-Dirac structure $L_G$ to the KKS Poisson structure $L_{\g^{*}}$ lifts to a smooth deformation of the quasi symplectic groupoid $D(G) \rightrightarrows G$ to $T^{*}G \rightrightarrows \g^{*}.$
\end{Prop}

\begin{proof}
Define the groupoid $\mathbf{G} = \D =  G \times \D(G,\{1\}) = (D(G) \times \R^{*}) \bigsqcup (T^{*}G \times \{0\})$ over $\G = \D(G,\{1\})$ By \cite[Theorem 4.1]{burelle2026deformations}, $\mathbf{G}$ is a smooth manifold whose source and target maps are smooth maps given by 
    \begin{align*}
        \mathbf{s}: \mathbf{G} \too \G, \quad m \too \begin{cases}
            (\Ad_{a}b,t), & \text{ if } m \in \D_t, t\neq 0,\\
            (\Ad_{a}x,0),& \text{ if } m = ((a,x),0) \in D_0,
        \end{cases}
    \end{align*}

    \begin{align*}
        \mathbf{t}: \mathbf{G} \too \G, \quad m\mtoo \begin{cases}
            (b^{-1},t), & \text{ if } m = (a,b,t)\in \D_{t}, t\neq 0\\
            (-x,0),& \text{ if } m = ((a,x),0) \in \D_0.
        \end{cases}
    \end{align*}

Recall that the multiplication in $D(G)$ is $(a_1,b_1)\cdot (a_2,b_2) = (a_1a_2,b_2)$ when $s(a_1,b_1) = t(a_2,b_2),$ i.e., $\Ad_{a_1}b_1 = b_{2}^{-1}.$ The multiplication deforms smoothly: at $t \neq 0$ is a group multiplication, and at $t = 0$ it becomes \[
(a_1,y_1)\cdot (a_2,y_2) = (a_1a_2,y_2),
\]
which is the groupoid multiplication in $T^{*}G.$ Finally, define the multiplicative two-form $\mathbf{\Omega}$ to be $\widehat{\omega}.$ For $t \neq 0, \mathbf{\Omega}_t = \dfrac{\omega_{D(G)}}{t}$ and the same argument used in \cite[Theorem 4.3]{burelle2026deformations} shows that the rescaled form extends smoothly to $t = 0$ to the canonical symplectic form $\mathbf{\Omega}_0 = \omega_{T^{*}G}.$  Finally, at $t \neq 0,$ $(\mathbf{G}_t,\mathbf{\Omega}_t) = (D(G),\dfrac{\omega_{D(G)}}{t})$ integrates the Cartan-Dirac structure $\K_t$ on $\G_t \cong G.$ At $t = 0,$ $(\mathbf{G}_0,\mathbf{\Omega}_0) = (T^{*}G,\omega_{T^{*}G})$ is the symplectic groupoid integrating $(\g^{*},\pi_{\mathrm{KKS}}).$
\end{proof}

\subsection{Dirac deformation and infinitesimal Morita equivalence}

The Morita equivalence of Lie algebroids provides an infinitesimal criterion that is particularly well-suited to Dirac deformations. We use the dual-pair characterization: two Lie algebroids $A_1 \to X_1$ and $A_2 \to X_2$ are Morita equivalent if there exists a manifold $P$ with surjective submersions $J_1 \colon P \to X_1$ and $J_2 \colon P \to X_2$, whose fibers are connected and simply-connected, together with an isomorphism of pullback Lie algebroids $$J_1^! A_1 \;\cong\; J_2^! A_2 \quad \text{over } P,$$ such that the resulting Lie algebroid actions on $P$ are complete. When $A_1 = K_1$ and $A_2 = K_2$ are Dirac structures, this is equivalent, under transversality (Lemma \ref{lem:smooth-pullback}) to the Dirac identity $K_1 = (J_1)_{*}J^{*}_{2}K_2.$

\begin{Def}\label{def:family-dual}
Let $(\X^{(1)}, \K^{(1)})$ and $(\X^{(2)}, \K^{(2)})$ be two Dirac deformations. A \emph{smooth family of dual pairs} between them is a smooth manifold $\PP \to \R$ with two smooth submersions $J_i \colon \PP \to \X^{(i)}$ commuting with the projections to $\R$, such that for each $t \in \R$:
\begin{enumerate}
  \item $J_{i,t} \colon P_t \to X_t^{(i)}$ is a complete surjective submersion with connected, simply-connected fibers;
  \item $J_{1,t}$ is transverse to $K_t^{(2)}$ and $J_{2,t}$ is transverse to $K_t^{(1)}$ in the sense of Dirac pullback;
  \item the dimension of the Dirac pullback $J_{i,t}^* K_t^{(j)}$ is locally constant in $t$ on $P_t$ (equivalently, the rank of $K_t^{(j)} \cap (J_i)_* TP_t$ is locally constant).
\end{enumerate}
\end{Def}
 
The role of (3) is to ensure that the Dirac pullback is a smooth Dirac structure on $\PP$, varying smoothly in $t$ across $t = 0$. We now make this explicit.
 
\begin{Lem}\label{lem:smooth-pullback}
Under the hypotheses of Definition~\ref{def:family-dual}, the Dirac pullback $J_{i,t}^* K_t^{(j)}$ is a smooth Dirac structure on $P_t$ for each $t \in \R$, and the assignment $t \mapsto J_{i,t}^* K_t^{(j)}$ is smooth as a section of the Grassmannian bundle of subspaces of constant dimension in $TP \oplus T^*P$.
\end{Lem}
 
\begin{proof}
The Dirac pullback under a smooth map $f \colon N \to M$ is defined as
\[
  f^* L = \{ (X, f^* \alpha) \in TN \oplus T^*N : (df(X), \alpha) \in L \}.
\]
When $f$ is transverse to $L$ (in the sense that $df(TN) + \mathrm{pr}_{TM}(L) = TM$), the pullback $f^*L$ is a smooth sub-bundle of $TN \oplus T^*N$ of the same rank as $L$. This is \cite[Proposition 5.6]{bursztyn2013brief} and is the standard transversality criterion for smooth Dirac pullback.
 
In our setting, $f = J_{i,t}$ varies smoothly in $t$ together with $L = K_t^{(j)}$. Hypothesis (2) of Definition~\ref{def:family-dual} provides transversality at each $t$. Hypothesis (3) ensures the rank of the pullback is locally constant in $t$. Together these imply that $J_{i,t}^* K_t^{(j)}$ is a smooth section of a Grassmannian bundle of fixed-dimensional subspaces of $TP \oplus T^*P$, smooth in $t$.
\end{proof}
 
\begin{Prop}\label{prop:morita-LA}
Let $(\X^{(1)}, \K^{(1)})$ and $(\X^{(2)}, \K^{(2)})$ be two Dirac deformations, and let $\PP$ be a smooth family of dual pairs between them. If the Dirac algebroids $K_t^{(1)}$ and $K_t^{(2)}$ are Morita equivalent via $(P_t, J_{1,t}, J_{2,t})$ for each $t \neq 0$, then the Poisson Lie algebroids $K_0^{(1)}$ and $K_0^{(2)}$ at $t = 0$ are Morita equivalent via $(P_0, J_{1,0}, J_{2,0})$.
\end{Prop}
 
\begin{proof}
By Lemma~\ref{lem:smooth-pullback}, the family $J_{i,t}^* K_t^{(j)}$ is a smooth Dirac structure on $P_t$ varying smoothly in $t$. The pushforward $(J_{1,t})_* J_{2,t}^* K_t^{(2)}$ is well-defined as a Dirac structure on $X_t^{(1)}$ because $J_{1,t}$ is a complete submersion with connected, simply-connected fibers (hypothesis (1) of Definition~\ref{def:family-dual}). The dimension of the pushforward equals $\dim X_t^{(1)}$, which is locally constant in $t$ (this is part of the smooth family of Dirac-deformations data).
 
For $t \neq 0$, Morita equivalence holds by hypothesis, and is witnessed by $(P_t, J_{1,t}, J_{2,t})$ together with the identity $K_t^{(1)} = (J_{1,t})_* J_{2,t}^* K_t^{(2)}$. Both sides are smooth sections of the Grassmannian bundle of constant dimensional subspaces of $T\X^{(1)} \oplus T^*\X^{(1)}$, smooth in $t$. They agree on $\{t \neq 0\}$, which is dense in $\R$. By continuity of sections of a Grassmannian bundle, ensured by the local constancy rank condition (Definition \ref{def:family-dual}(3)), they agree at $t = 0$:
\[
  K_0^{(1)} = (J_{1,0})_* J_{2,0}^* K_0^{(2)}.
\]
This is the dual-pair characterization of Morita equivalence for $K_0^{(1)}$ and $K_0^{(2)}$, which are Poisson by the Dirac deformation hypothesis.
\end{proof}
 
\begin{Rem}
Proposition~\ref{prop:morita-LA} does not claim that Morita equivalence of Dirac structures is preserved under arbitrary deformations: it requires a smooth family of dual pairs witnessing the equivalence, including the transversality, completeness, and rank-constancy conditions of Definition~\ref{def:family-dual} at $t = 0$. These conditions can genuinely fail under degenerations.
\end{Rem}

\subsection{Compatibility with Hamiltonian spaces}
For a Hamiltonian $\G$-space, the integration of the infinitesimal action to a groupoid action requires a completeness hypothesis. We state this explicitly.

\begin{Def} \label{Completeness}
    Let $\G \rightrightarrows \X$ be a deformation of quasi-symplectic groupoids, and let $\widehat{\varphi}: \M \too \X$ be a smooth family of strong Dirac maps over $\R.$ The associated Lie algebroid action $\rho: \K \times_{\X} T\M$ is \textit{complete} at \textit{parameter} $t$ if every $K_t$-path starting in $M_t$ can be integrated on all of $[0,1].$ The family is \textit{complete} if this holds for every $t.$
\end{Def}

\begin{Thm}\label{thm:ham-deform}
Let $\mathbb{G} \rightrightarrows \X$ be a deformation of a quasi-symplectic groupoids with each $\G_t$ source-simply connected, and let $(M_1,\omega_1,\Phi_1)$ be a Hamiltonian $\G_1$-space. Suppose: \begin{itemize}
\item [(i)] $(M_1,L_{\omega_1})$ admits a smooth Dirac deformation $(\M,L_{\M})$ of graph type, with $L_{M,t} = L_{\omega_t}$ for a smooth family $\omega_t$ of non-degenerate 2-forms such that $\omega_0$ is symplectic;
\item [(ii)] There exists a smooth family of strong Dirac maps $\widehat{\varphi}_t: M_t \too X_t$ deforming $\Phi_1$;
\item [(iii)] The associated Lie algebroid action (of Definition \ref{Completeness}) is complete.
\end{itemize}

Then, \begin{enumerate}
    \item For each $t\neq 0,$ $(\M_t,\omega_t,\widehat{\varphi}_t)$ is a Hamiltonian $\G_t$-space.
    \item At $t = 0,$ $(M_0,\omega_0,\widehat{\varphi}_0)$ is a Hamiltonian $\G_0-$space; equivalently, denoting $\pi_{M_0}$ the Poisson structure with $L_{\pi_{\M_0}} = L_{\omega_0},$ the map $\widehat{\varphi}_0: (\M_0, \pi_0) \too (\X_0,\K_0)$ is a Poisson map.
    \item  The $\G_t$-action on $\M_t$ varies smoothly in $t,$ and the action at $t =0$ is the symplectic-groupoid integration of the Poisson map $\widehat{\varphi}_0.$
\end{enumerate}
\end{Thm}

\begin{proof}
We proceed in six steps.
 
\medskip
\noindent\textbf{Step 1 (Lie algebroid action; smoothness in $t$).}
By \cite[Corollary 3.12]{bursztyn2005dirac}, for each $t \in \R$ the strong Dirac map $\hat\varphi_t \colon (\M_t, \L_{\omega_t}) \to (\X_t, \K_t)$
induces a Lie algebroid action
\[
   \rho_t \colon \K_t \times_{\X_t} \M_t \longrightarrow T\M_t,
   \qquad
   (w, \alpha) \longmapsto \rho_t(w, \alpha),
\]
where $\rho_t(w,\alpha) \in T_m \M_t$ is the unique vector satisfying
$d\hat\varphi_t(\rho_t(w,\alpha)) = w$ and
$\bigl(\rho_t(w,\alpha), \hat\varphi_t^{*}\alpha\bigr) \in L_{\omega_t}$.\\

To see that $\rho: \K \times_{\X} \M \too T\M$ is smooth, define the bundle map $A: T\M \too \hat\varphi^{*}T\X \oplus T^{*}\M, \quad v \mtoo (d\hat\varphi(v),\omega^{\flat}(v)),$ where everything is fiberwise vertical (relative to the projection $\pi_{\M}: \M \too \R$). The strong Dirac condition $\ker d\hat\varphi_t \cap \ker L_{\omega_t} = 0$ that holds for every $t$ means that $A$ is fiberwise injective, hence of constant rank $\text{dim}(\M_t).$ Therefore, $\mathrm{Im}(A) \subset \hat\varphi^{*}TX \oplus T\M^{*}$ is a smooth sub-bundle and $A$ admits a smooth left inverse $A^{-1}$ on its image. Moreover, $\A \circ \rho = (\mathrm{pr}_{\X}, \hat\varphi^{*}\mathrm{pr}_{T^{*}\X})$, where $((\mathrm{pr}_{\X},\mathrm{pr}_{T^{*}\X}))$ is the inclusion $\K \too T\X \oplus T^{*}\X.$ Therefore, $\rho = A^{-1} \circ (\mathrm{pr}_{\X},\hat\varphi^{*}\mathrm{pr}_{T^{*}\X})$ is the composition of two smooth maps, hence smooth. 
 
\medskip
\noindent\textbf{Step 2 (Groupoid action for $t \neq 0$).}
For each $t \neq 0$, the Lie algebroid action $\rho_t$ integrates to a local $\mathcal{G}_t$-action by source-simple-connectedness of $\mathcal{G}_t$ \cite[Theorem 4.1]{crainic2003integrability}. Hypothesis (iii) ensures that this local action is
defined on all of $\G_t \times_{\X_t} \M_t$.
 
\medskip
\noindent\textbf{Step 3 (Groupoid action at $t = 0$).}
At $t = 0$, $\mathcal{G}_0$ is the source-simply-connected symplectic groupoid integrating $(\mathcal{X}_0, \mathcal{K}_0)$ by Theorem \ref{groupoid integra}.
Since $\omega_0$ is symplectic by (i), we have $\L_{\omega_0} = \L_{\pi_{\M_0}}$ with $\pi_{\M_0} := \omega_0^{-1}$, and the strong Dirac map $\hat\varphi_0$
(with vanishing twist $\lambda_{\mathcal{X},0} = 0$) is exactly a Poisson map
$(\M_0, \pi_{\M_0}) \to (\X_0, \K_0)$. The induced cotangent
algebroid action $\rho_0$ is the standard Poisson-map action
$\alpha \mapsto \pi_{\M_0}^{\sharp}(\hat\varphi_0^{*}\alpha)$. It is complete by (iii), so it integrates to a global $\mathcal{G}_0$-action.
 
\medskip
\noindent\textbf{Step 4 (Smoothness of the integrated action in $t$).}
We show that the global action map
\[
   \Theta \colon \mathcal{G} \times_{\mathcal{X}} \mathcal{M} \longrightarrow \mathcal{M},
   \qquad (g, m) \longmapsto g \cdot m,
\]
is smooth.
 
Fix $(t_0, g_0, m_0)$. By Proposition~\ref{integrability prop},
a neighborhood of $g_0$ in $\mathcal{G}$ is smoothly parametrized via the Lie algebroid exponential chart of \cite{mackenzie2005general}: for each $t$
near $t_0$, the exponential map
$\exp_t \colon \mathcal{O}_t \subset \mathcal{K}_t \to \mathcal{G}_t$
is a diffeomorphism onto a neighborhood of the unit bisection, and the family of maps $\exp$ is smooth in $(t, \xi)$.
 
Given $\xi \in \O_t \subset \K_t$, the action of
$\exp_t(\xi) \in \G_t$ on $m \in \M_t$ is, by the construction of integrated Lie algebroid actions, the
time-$1$ flow $\Psi_t(\xi, m)$ of the time-dependent vector field
\[
   X_t^{\xi} \colon [0,1] \times \M_t \longrightarrow T\M_t,
   \qquad (s, m') \longmapsto \rho_t\bigl(\xi(s), m'\bigr),
\]
where $s \mapsto \xi(s)$ is the canonical $\mathcal{K}_t$-path from $0$ to $\xi$.
By Step 1, the map $(t, \xi, s, m') \mapsto X_t^{\xi}(s, m')$ is smooth. By hypothesis (iii), for each $(t, \xi, m)$ the flow of $X_t^{\xi}$ exists on $[0, 1]$. Fix a compact set $K \subset \G \times_{\X} \M$.
By smooth dependence of ODE solutions on parameters and openness of the
maximal existence interval, the set of $(t, \xi, m) \in K$ for which the flow of $X_t^{\xi}$ extends to a neighborhood of $[0, 1]$ is open in $K$; by fiberwise completeness it equals all of $K$. Hence the flow exists on $[0, 1]$ uniformly on $K$, and smooth dependence on $(t, \xi, m)$ follows.
This gives smoothness of $\Psi$, and hence of $\Theta$, on a neighborhood of the unit bisection.
 
Propagation to all of $\mathcal{G}$ is by right-translation along smooth local bisections, exactly as in Step~2 of the proof of
Proposition \ref{integrability prop}: smoothness of multiplication
on the unit-bisection chart, combined with smooth dependence of local
bisections on $(t, g)$, transports smoothness of $\Theta$ to a neighborhood of $g_0$. This is the smooth-family upgrade of Crainic–Fernandes \cite{crainic2003integrability},  combined with the bisection tools of Lie groupoids of Mackenzie \cite{mackenzie2005general}. Pointwise existence of the integrated action is the content of  Moerdijk–Mrčun \cite{moerdijk2002integrability}; the smooth-family work on which we insert parameter dependence follows the convention of Crainic–Mestre–Struchiner \cite{crainic2020deformations}. 
 
\medskip
\noindent\textbf{Step 5 (Hamiltonian conditions).}
We verify each condition of Definition \ref{def:ham} for every $t$.

\begin{itemize}
\item \emph{(a) Non-degeneracy of $\L_{\omega_t}$.} By hypothesis (i), $\omega_t$ is non-degenerate for every $t$, so $\L_{\omega_t} = L_{\pi_{\M_t}}$ with
$\pi_{M_t} := \omega_t^{-1}$; in particular $\ker \L_{\omega_t} = 0$.
 
\item \emph{(b) Twisting compatibility.} For graph-type Dirac structures
$L_{\mathcal{M},t} = L_{\omega_t}$, the twisting is
$\lambda_{\mathcal{M},t} = -d\omega_t$ by the definition of graph-type
twisted Dirac structures. The strong-Dirac-map condition
$\hat\varphi_t^{*}\lambda_{\mathcal{X},t} = \lambda_{\mathcal{M},t}$ from~(ii)
yields
\[
   d\omega_t \;=\; -\hat\varphi_t^{*}\lambda_{\mathcal{X},t}.
\]
At $t = 0$, $\lambda_{\mathcal{X},0} = 0$ gives $d\omega_0 = 0$, consistent
with $\omega_0$ symplectic.
 
\item \emph{(c) Minimal non-degeneracy.} The condition
$\ker d\hat\varphi_t \cap \ker \omega_t = 0$ is the second axiom of a strong
Dirac map and holds by~(ii).
 
\item \emph{(d) Isotropy of the action graph.} The action graph
\[
   \Gamma_t \;=\; \{(g, m, g \cdot m) : s(g) = \hat\varphi_t(m)\}
   \;\subset\; \mathcal{G}_t \times \M_t \times \M_t
\]
must be isotropic with respect to $\Omega_t \oplus \omega_t \oplus (-\omega_t)$.
Infinitesimally, this isotropy is the conjunction of two conditions: the
$f$-Dirac property of $\hat\varphi_t$, the weaker piece of the strong-Dirac hypothesis~(ii) and the compatibility between the multiplicative form $\Omega_t$ and the action $\rho_t$ via the IM-form realization
$\mu_t \colon \mathcal{K}_t \to T^{*}\mathcal{X}_t$ of Lemma \ref{IM 2 form},
$(v,\alpha) \mapsto \alpha$. By source-simple-connectedness of $\mathcal{G}_t$, the infinitesimal isotropy integrates to global isotropy of $\Gamma_t$ on
$\mathcal{G}_t$; this is the IM-to-multiplicative correspondence of
\cite{bursztyn2004integration} (see also \cite{bursztyn2012multiplicative}), applied to the IM-data
$(\mathcal{K}_t, \mu_t, \lambda_{\mathcal{X},t})$ and its
$\hat\varphi_t$-pullback to $\M_t$.
\end{itemize}

\medskip
\noindent\textbf{Step 6 (Identification of the $t=0$ action with the symplectic-groupoid integration).}
By Theorem \ref{groupoid integra}, $\mathcal{G}_0$ is the
source-simply-connected symplectic groupoid integrating
$(\mathcal{X}_0, \mathcal{K}_0)$. By Step~3, the algebroid action $\rho_0$ is
the cotangent action
$\alpha \mapsto \pi_{\M_0}^{\sharp}(\hat\varphi_0^{*}\alpha)$ associated to
the Poisson map $\hat\varphi_0 \colon (\M_0, \pi_{\M_0}) \to (\X_0, \K_0)$.
The integration of this action to $\G_0$ is, by construction, the
standard Hamiltonian action of the symplectic groupoid on $(\M_0, \omega_0)$ induced by a Poisson map; see \cite{mikami1988moments,xu2004momentum} for the identification of this action with the Hamiltonian $\G_0$-space structure induced by a Poisson map and \cite[Section~5.1]{bualibanu2026reduction} for the Dirac-geometric formulation.
\end{proof}

\section{Shifted symplectic presentation via the presentation model} \label{Section 7}

\subsection{The presentation model} Full derived algebraic geometry (DAG) over $\R$ in the sense of \cite{calaque2021derived} and \cite{pantev2013shifted} is beyond the scope of this paper. Instead, we use the following presentation model, which encodes 1-shifted symplectic stacks via their presenting data and is adequate for all applications of this paper.

\begin{Def}
\textit{A smooth family of presented} $1$-\textit{shifted symplectic stacks} over $\R$ is a deformation of quasi-symplectic groupoids $(\G,\Omega,\lambda_X)$ in the sense of Definition \ref{def:family}. We denote it $[\X/\G]$ and interpret it ''smooth family of 1-shifted symplectic stacks.''
\end{Def}

Before stating Definition \ref{def:family-lag}, we recall the explicit form of the Lagrangian complex $C^{\bullet}$ of \cite[Theorem 5.6]{bualibanu2026reduction} and give the geometric interpretation of each of its terms and differentials. This makes the homological reformulation of the Lagrangian condition transparent and motivates the locally-constant-rank hypothesis below.

\medskip
\noindent\textbf{The complex.} Let
$(\mathcal{G}_t \rightrightarrows X_t, \Omega_t, \lambda_{X_t})$ be a
quasi-symplectic groupoid presenting the $1$-shifted symplectic stack
$[\X_t / \G_t]$, and let $i_t \colon \S_t \hookrightarrow \X_t$
be a generalized reduction level with $2$-form $\gamma_t \in \Omega^2(\S_t)$
and stabilizer subgroupoid $\H_t \rightrightarrows \S_t$ inside
$\G_t|_{\S_t}$. Following \cite[Equation (5.9)]{bualibanu2026reduction},
the Lagrangian complex is
\begin{equation}\label{eq:lagrangian-complex}
   C_t^\bullet \;:\;
   \Lie(\H_t)
   \xrightarrow{\;\alpha_t\;}
   L_{\X_t}|_{\S_t} \oplus T\S_t
   \xrightarrow{\;\beta_t\;}
   T\X_t|_{S_t} \oplus T^*\S_t
   \xrightarrow{\;\delta_t\;}
   \Lie(\H_t)^*,
\end{equation}
with differentials
\begin{align*}
   \alpha_t(\xi) &= \bigl(\xi,\, \rho_{\H_t}(\xi)\bigr),\\
   \beta_t(a, X) &= \bigl(\,di_t(X) - \rho_{L_{\X_t}}(a),\;\;
                          i_t^{*}\,\pr_{T^*\X_t}(a) - \gamma_t^\flat(X)\,\bigr),\\
   \delta_t(w, \eta)(\xi) &= \langle \pr_{T^*X_t}(\xi),\, w\rangle
                            - \langle \eta,\, \rho_{H_t}(\xi)\rangle,
\end{align*}
where $\rho_{L_{\X_t}}\colon L_{\X_t} \to T\X_t$ and
$\pr_{T^*\X_t}\colon L_{\X_t} \to T^*\X_t$ are the two projections in $L_{\X_t} \subset T\X_t \oplus T^*\X_t$, and
$\rho_{\H_t} = \rho_{L_{\X_t}}|_{\Lie(\H_t)}$. The complex is self-dual under the $1$-shifted symplectic pairing: $\delta_t$ is dual to $\alpha_t$ and
$\beta_t$ is self-dual. We write
$\beta_t = (\beta_t^{(1)}, \beta_t^{(2)})$ for its two components, agreeing with the pair $(\beta, \gamma)$ of \cite[(5.9)]{bualibanu2026reduction}.
 
\medskip
\noindent\textbf{Geometric content.} The standard analogue from ordinary symplectic geometry is \emph{Hamiltonian reduction}, where $X$ is a
symplectic manifold with a Hamiltonian $G$-action, moment map
$\mu \colon X \to \g^*$, regular value $0$, and $G$ acting freely
on $\mu^{-1}(0)$. In that case, the morphism
$[\mu^{-1}(0)/G] \to [X/G]$ is Lagrangian, and the acyclicity of the complex encodes precisely the Marsden-Weinstein hypotheses. The four cohomology groups have the following meaning, in general and specialized to this case.
 
\smallskip
$\bullet$\; $H^{-1}(C_t^\bullet) = \ker \alpha_t$ vanishes iff
$\Lie(\H_t)$ injects into $L_{\X_t}|_{\S_t} \oplus T\S_t$. Geometrically this is the \emph{effectiveness} (infinitesimal local freeness) of the $\H_t$-action. In the Hamiltonian-reduction example, this is freeness of
the $G$-action on $\mu^{-1}(0)$.
 
\smallskip
$\bullet$\; $H^{0}(C_t^\bullet) = \ker \beta_t / \text{im} \alpha_t$ vanishes iff
every infinitesimal direction $(a, X) \in L_{\X_t}|_{\S_t} \oplus T\S_t$
which is annihilated by $\beta_t$ comes from $\Lie(\H_t)$ via $\alpha_t$.
This is the \emph{moment-map / coisotropy} condition: the stabilizer
subalgebroid $A_{\S_t,\gamma_t,\phi_t}$ exhausts the kernel of $\beta_t$,
in the precise form
\[
   A_{\S_t,\gamma_t,\phi_t}
   \;=\;
   L_{\X_t} \cap i_t^{*}L_{\X_t} \cap (\ker d\phi_t \oplus T^{*}\X_t)
\]
of \cite[Definition 2.1]{bualibanu2026reduction}. In the Hamiltonian-reduction
example, this is the moment-map equation
$T\mu^{-1}(0)^{\perp_{\omega_X}} = \mathfrak{g}\text{-orbits}$.
 
\smallskip
$\bullet$\; $H^{+1}(C_t^\bullet)$ is dual to $H^{0}(C_t^\bullet)$ under the
$1$-shifted symplectic pairing; its vanishing is automatic given $H^0 = 0$.
 
\smallskip
$\bullet$\; $H^{+2}(C_t^\bullet) = \mathrm{coker}\,\delta_t$ is dual to
$H^{-1}(C_t^\bullet)$, and its vanishing is the dual effectiveness
statement.
 
\medskip
\noindent\textbf{Half-dimensionality is bundled into $H^0 = 0$.}
A naive analogy would identify the Lagrangian condition with the claim ''isotropic $+$ half-dimensional.'' This is not true. In the
1-shifted setting, isotropy is built into the self-dual structure of
$C_t^\bullet$, and the role of half-dimensionality is played by the
dimension constraint forced by acyclicity at degree $0$. Concretely, in
the Hamiltonian-reduction example, $H^0 = 0$ forces the symplectic
orthogonal of $T\mu^{-1}(0)$ to equal the $G$-orbit distribution,
which has dimension $\dim \mathfrak{g}$. Combined with effectiveness
($\dim$-stabilizer $=0$), this gives
\[
   \dim \mu^{-1}(0) \;=\; \dim X - \dim \mathfrak{g},
\]
so $\dim(\mu^{-1}(0)/G) = \dim X - 2\dim\mathfrak{g}$, the expected
dimension of a symplectic reduction. The half-dimensionality content of
the classical Lagrangian condition is bundled into the $H^0$-vanishing,
not into a separate dimension count.
 
\begin{Rem}\label{rem:trivial-stack}
When $\mathcal{G}_t = \X_t$ is the trivial unit groupoid (so
$\Lie(\H_t) = 0$ and the stack $[\X_t/\mathcal{G}_t] = \X_t$ is
$0$-shifted), the complex degenerates: acyclicity at $H^0$ forces
$\Ker(i^{*}\omega_{X_t} = T\S_t \cap T\S_t^{\perp \omega_X}) = 0$, i.e., $\S_t$ is a symplectic submanifold of $\X_t$, not Lagrangian. The 1-shifted formalism is
designed for genuinely 1-shifted targets, where the gauge symmetry
$\G_t$ provides the dimension reduction.
\end{Rem}

\subsection{Lagrangian morphisms}
Now we state the main definition.
\begin{Def} \label{def:family-lag}
\textit{A smooth family of Lagrangian morphisms} $[\S/\G] \too [\X/\H]$ is a smooth family of ordinary reduction levels $(\S,\gamma)$ together with a smooth family of stabilizer subgroupoids $\H \rightrightarrows \S$ (i.e. a smooth subgroupoid of $\G|\S$), such that for each $t \in \R$ the pair $(\S_t,\H_t,\gamma_t)$ is a Lagrangian morphism in the sense of \cite[Theorem 5.6]{bualibanu2026reduction}. We assume: 

\begin{enumerate}[]
  \item the ranks of the terms in the Lagrangian complex
  \begin{equation}\label{eq:Lag-complex}
    C^\bullet_t \colon \Lie(\H_t) \xrightarrow{\alpha_t} L_{\X_t}|_{\S_t} \oplus T\S_t \xrightarrow{\beta_t} T\X_t|_{\S_t} \oplus T^*\S_t \xrightarrow{\delta_t} \Lie(\H_t)^*
  \end{equation}
  are locally constant in $t$;
  \item the differentials $\alpha_t, \beta_t, \delta_t$ have locally constant rank in $t$ (equivalently, the cohomology bundles $H^i(C^\bullet_t)$ have locally constant rank).
\end{enumerate}
\end{Def}

\begin{Rem}
Condition $(2)$ is essential. Without it, acyclicity at $t \neq 0$ does not propagate to $t = 0:$ cohomology rank is upper-semi continuous in general, so vanishing on a dense subset is not preserved. A simple example is the family of complexes $\R \xrightarrow{\cdot t} \R$ on $\R$, which has trivial cohomology for $t\neq 0$ but $H^{0} = H^{1} = \R$ at  $t = 0.$
\end{Rem}

\subsection{The main theorem}
 
\begin{Thm}\label{thm:shifted}
Let $\G \rightrightarrows \X$ be a deformation of quasi-symplectic groupoids. Then:
\begin{enumerate}
  \item The presenting data $(\Omega, \lambda_\X)$ of the $1$-shifted symplectic structure on $[X_t / \G_t]$ varies smoothly in $t$; at $t = 0$, it is the standard $1$-shifted symplectic data of the symplectic groupoid $\G_0$.
  \item Let $(\S, \H, \gamma)$ be a smooth family of Lagrangian morphisms $[\S_t / \H_t] \to [\X_t / \G_t]$ (Definition \ref{def:family-lag}). Then the Lagrangian condition is preserved at $t = 0$.
  \item Let $\hat\varphi \colon \M \to \X$ be a smooth family of strong Dirac maps satisfying the hypotheses of Theorem \ref{thm:ham-deform} (in particular, completeness). Then $(\M, \omega)$ gives a smooth family of Lagrangian morphisms $[\M_t / \G_t] \to [\X_t / \G_t]$.
  \item Suppose further that, for each $t \in \mathbb{R}$:
  \begin{itemize}
    \item the fiber product
    \[
      \S_t \times_{\X_t} \M_t \;:=\; \{(s,m)\in \S_t\times \M_t : \phi_t(s)=\hat\varphi_t(m)\}; \quad \phi_t = i_t: \S_t \hookrightarrow \X_t \text{ is the injection}.
    \]
    is a clean intersection of $\phi_t$ and $\hat\varphi_t$  (a smooth manifold of the expected dimension
    $\dim \S_t+\dim \M_t-\dim \X_t$);
    \item $\H_t$ acts locally freely on $\S_t\times_{\X_t}\M_t$ via the
    \emph{diagonal action}
    \[
      h\cdot(s,m)\;=\;(h\cdot s,\;h\cdot m),
    \]
    where $h\cdot s$ is given by the groupoid action $\H_t\rightrightarrows \S_t$
    and $h\cdot m$ is given by the restriction to $\H_t\subset
    \G_t|_{\S_t}$ of the $\G_t$-action on $\M_t$.
    Equivalently, under the diffeomorphism $\S_t\times_{\X_t}\M_t\cong
    \hat\varphi_t^{-1}(\S_t)$, $(s,m)\mapsto m$, this is the action of
    $\H_t$ on $\hat\varphi_t^{-1}(\S_t)\subset \M_t$ inherited from
    $\G_t\curvearrowright \M_t$.
  \end{itemize}
  Then the reduced fiber product
  \[
    Z_t \;:=\; \frac{\S_t\times_{\X_t}\M_t}{\H_t}
  \]
  is a smooth manifold for each $t$, is $0$-shifted symplectic
  (a smooth symplectic manifold in the ordinary sense), varies smoothly in $t$. Moreover, $Z_t$ presents the stack-level fiber
  product $[\S_t/\H_t]\times_{[\X_t/\G_t]}[\M_t/\G_t]$.
\end{enumerate}
\end{Thm}

We decompose the proof of Theorem \ref{thm:shifted} into three statements: smoothness of the family fiber product, closedness of the reduced 2-form, and non-degeneracy. 

\begin{Lem}[Smooth family fiber product]\label{lem:smooth-family-fibre-product}
Let $\G \rightrightarrows \X$ be a deformation of
quasi-symplectic groupoids, let $\S \to \R$ be a smooth family of generalized reduction levels with stabilizer subgroupoids $\H \rightrightarrows \S$, and let
$\hat\varphi \colon \M \to \X$ be a smooth family of
strong Dirac maps satisfying the hypotheses of
Theorem~\ref{thm:ham-deform}. Suppose that for every
$t \in \mathbb{R}$:
\begin{enumerate}[label=\textnormal{(\roman*)}]
  \item $\S_t$ intersects $\hat\varphi_t$ cleanly in $\X_t$;
  \item $\H_t$ acts locally freely on the fiber product
        $\S_t \times_{\X_t} \M_t$ via the diagonal action.
\end{enumerate}
Then:
\begin{enumerate}
  \item the total space
        $\S \times_{\X} \M
            := \bigsqcup_t \bigl(\S_t \times_{\X_t} \M_t\bigr)$
        is a smooth manifold over $\R$, with each fiber a smooth
        submanifold of dimension
        $\dim \S_t + \dim \M_t - \dim \X_t$;
  \item the diagonal $\mathcal{H}$-action on
        $\mathcal{S} \times_{\mathcal{X}} \mathcal{M}$ is smooth and locally
        free, and the quotient
        \[
           \mathcal{Z} \;:=\;
           \bigsqcup_t\, Z_t
           \;=\; \bigsqcup_t \frac{\S_t \times_{\X_t} \M_t}{\H_t}
        \]
        is a smooth manifold over $\mathbb{R}$;
  \item the projection $\mathrm{pr}_M \colon
        \S_t \times_{\X_t} \M_t \to \hat\varphi_t^{-1}(\S_t)$,
        $(s, m) \mapsto m$, is a diffeomorphism, and under this
        identification the diagonal $\H_t$-action transports to the
        restricted $\H_t \subset \mathcal{G}_t|_{\S_t}$-action on
        $\hat\varphi_t^{-1}(\S_t)$ inherited from
        $\mathcal{G}_t \curvearrowright \M_t$.
\end{enumerate}
\end{Lem}
 
\begin{proof}
\emph{(1)} The maps $\phi \colon \mathcal{S} \to \mathcal{X}$ (the family
of inclusions) and $\hat\varphi \colon \mathcal{M} \to \mathcal{X}$ are
smooth families over $\mathbb{R}$, all submersive over $\mathbb{R}$ by
construction. Fiberwise cleanness from~(i), combined with smoothness of
$\phi$ and $\hat\varphi$ as families and smoothness of the projections
$\pi_{\mathcal{S}}, \pi_{\mathcal{M}}, \pi_{\mathcal{X}}$ over $\mathbb{R}$,
gives that the fiber product
$\mathcal{S} \times_{\mathcal{X}} \mathcal{M}$ is a smooth submanifold of
$\mathcal{S} \times_{\mathbb{R}} \mathcal{M}$ of fiberwise dimension
$\dim \S_t + \dim \M_t - \dim \X_t$.
 
\emph{(2)} The $\H$-action on the total space is smooth because
$\H \to \R$ is a smooth family of subgroupoids of
$\G|_{\S}$ (Definition \ref{def:family-lag}),
$\G$ acts smoothly on $\M$ as families
(Theorem~\ref{thm:ham-deform}(3)), and the diagonal action
$h \cdot (s, m) = (h \cdot s,\, h \cdot m)$ is a composition of smooth operations. Local freeness is an open condition on $(t, s, m)$: the set
\[
   \mathcal{F}
   \;=\; \{(t, s, m) \in \mathcal{S} \times_{\mathcal{X}} \mathcal{M}
           \;:\; \H_t \curvearrowright (s, m)
           \text{ is locally free at } (s, m)\}
\]
is open, and by~(ii) it contains every fiber. Hence $\mathcal{F}$ is all
of $\mathcal{S} \times_{\mathcal{X}} \mathcal{M}$, and the quotient
$\mathcal{Z}$ is a smooth manifold of fiberwise dimension
$\dim \S_t + \dim \M_t - \dim X_t - \dim \Lie(\H_t)$.
 
\emph{(3)} For each $t$, the inclusion $\phi_t = i_t$ is an embedding and
the fiber product $\S_t \times_{\X_t} \M_t$ collapses to
$\hat\varphi_t^{-1}(\S_t) \subset \M_t$ via $(s, m) \mapsto m$, which is a
diffeomorphism by cleanness. The diagonal $\H_t$-action on
$\S_t \times_{\X_t} \M_t$ transports to the standard restriction of the
$\G_t$-action: since $\H_t \rightrightarrows \S_t$ is a subgroupoid
of $\G_t|_{\S_t}$, both source and target of every $h \in \H_t$
lie in $\S_t$, so $\H_t$ preserves $\hat\varphi_t^{-1}(\S_t)$, and the action on the fiber product agrees with this restriction under the diffeomorphism. The result is the family version of
\cite[Proposition 3.7]{bualibanu2026reduction}; smoothness in $t$ follows from the smooth-family data already established.
\end{proof}

\begin{Prop}[Closedness via twist cancellation]\label{prop:closedness}
Under the hypotheses of Lemma~\ref{lem:smooth-family-fibre-product},
define
\begin{equation}\label{eq:reduced-form-upstairs}
   \widetilde\omega_t
   \;:=\;
   \mathrm{pr}_M^{*} \omega_t \;-\; \mathrm{pr}_S^{*} \gamma_t
\in \Omega^2\bigl(\S_t \times_{\X_t} \M_t\bigr).
\end{equation}
Then $d\widetilde\omega_t = 0$ for every $t \in \R$. Equivalently,
on $\hat\varphi_t^{-1}(\S_t)$ under the diffeomorphism of
Lemma~\ref{lem:smooth-family-fibre-product} (3), the form
$\widetilde\omega_t = \jmath_t^{*}\omega_t - \hat\varphi_t^{*}\gamma_t$
is closed, where $\jmath_t \colon \hat\varphi_t^{-1}(\S_t) \hookrightarrow \M_t$
is the inclusion.
\end{Prop}
 
\begin{proof}
By the Hamiltonian-space condition,
$d\omega_t = -\hat\varphi_t^{*} \lambda_{\mathcal{X},t}$ on $\M_t$. By the reduction-level condition $i_t^{*}\eta_{\mathcal{X},t} + d\gamma_t = 0$
of \cite[Definition 2.1]{bualibanu2026reduction},
$d\gamma_t = -\phi_t^{*}\lambda_{\mathcal{X},t}$ on $\S_t$. Pulling back to
$\S_t \times_{\X_t} \M_t$:
\[
   d(\mathrm{pr}_M^{*}\omega_t)
   = -(\hat\varphi_t \circ \mathrm{pr}_M)^{*} \lambda_{\mathcal{X},t},
   \qquad
   d(\mathrm{pr}_S^{*}\gamma_t)
   = -(\phi_t \circ \mathrm{pr}_S)^{*} \lambda_{\mathcal{X},t}.
\]
By the defining identity of the fiber product,
$\hat\varphi_t \circ \mathrm{pr}_M = \phi_t \circ \mathrm{pr}_S$, so the
two pullbacks of $\lambda_{\mathcal{X},t}$ coincide and
$d\widetilde\omega_t = 0$. The two twists cancel because the fiber
product forces the two maps into $X_t$ to agree.
\end{proof}

\begin{Prop}[Non-degeneracy]\label{prop:nondegeneracy}
Under the hypotheses of Lemma~\ref{lem:smooth-family-fibre-product},
the closed $2$-form $\widetilde\omega_t$ of
Proposition~\ref{prop:closedness} is $H_t$-basic, and the induced
$2$-form $\omega_{Z_t} \in \Omega^2(Z_t)$ characterized by
$\pi_t^{*}\omega_{Z_t} = \widetilde\omega_t$
(where $\pi_t \colon \S_t \times_{\X_t} \M_t \to Z_t$ is the quotient map)
is non-degenerate.
\end{Prop}
 
\begin{proof}
The basiness has two parts: $\H_t$-invariance and vanishing on the orbit distribution.\\

\noindent\textbf{Invariance.} $\omega_t$ is $\mathcal{G}_t$-invariant by the
Hamiltonian-space condition (Definition \ref{def:ham}(1)),
hence $\H_t$-invariant. The form $\gamma_t$ is $\H_t$-invariant by the Lagrangian-morphism condition (Definition~\ref{def:family-lag}).
The projections $\mathrm{pr}_S$ and $\mathrm{pr}_M$ are $\H_t$-equivariant for the diagonal action; hence $\widetilde\omega_t$ is $\H_t$-invariant.\\
 
\noindent\textbf{Vanishing on the orbit distribution.}
For $\xi \in \Lie(\H_t)$ viewed as a section of the stabilizer
subalgebroid $A_{\S_t,\gamma_t} \subset L_{\X_t}|_{\S_t}$ of
\cite[Definition 2.1]{bualibanu2026reduction}, write $\xi = (\hat\xi_S, \alpha_\xi)$
with $\hat\xi_S \in TS_t$ and $\alpha_\xi \in T^{*}\X_t|_{\S_t}$. The defining conditions $\xi \in L_{\X_t}$, $\hat\xi_S \in T\S_t$, and
$i_t^{*}\alpha_\xi = \gamma_t^\flat(\hat\xi_S)$ yield the $\S_t$-side
moment-map identity
\begin{equation}\label{eq:star-S}
   \gamma_t(\hat\xi_S, v) = \alpha_\xi(i_{t*}v),
   \qquad v \in T_s \S_t.
   \tag{$\star$}
\end{equation}
The strong-Dirac-map property of $\hat\varphi_t$
\cite[Corollary~3.12]{bursztyn2005dirac} produces a unique fundamental vector field $\hat\xi_M \in T\M_t$ characterized by
$\hat\varphi_{t*}\hat\xi_M = \hat\xi_S$ and
$(\hat\xi_M, \hat\varphi_t^{*}\alpha_\xi) \in L_{\M_t}$. For the
non-degenerate Dirac structure $L_{\M_t} = L_{\omega_t}$, the second condition reads $\omega_t^\flat(\hat\xi_M) = \hat\varphi_t^{*}\alpha_\xi$,
i.e., the $\M_t$-side moment-map identity
\begin{equation}\label{eq:star-M}
   \omega_t(\hat\xi_M, w) = \alpha_\xi(\hat\varphi_{t*}w),
   \qquad w \in T_m \M_t.
   \tag{$\star\star$}
\end{equation}
This $\hat\xi_M$ is precisely the fundamental vector field on $M_t$ for
the action of $\xi \in \Lie(\H_t)$ via the restricted $\G_t$-action; in particular it is generically non-zero.
 
The fundamental vector field on $\S_t \times_{\X_t} \M_t$ for the diagonal $\H_t$-action is therefore $(\hat\xi_S, \hat\xi_M)$. For a tangent vector
$(v, w) \in T_{(s, m)}(\S_t \times_{\X_t} \M_t)$, the fiber-product
condition reads $i_{t*}v = \hat\varphi_{t*}w$. Combining
\eqref{eq:reduced-form-upstairs}, \eqref{eq:star-S}, and
\eqref{eq:star-M}:
\[
   \iota_{(\hat\xi_S, \hat\xi_M)}\widetilde\omega_t(v, w)
   \;=\; \omega_t(\hat\xi_M, w) \;-\; \gamma_t(\hat\xi_S, v)
   \;\stackrel{\eqref{eq:star-M},\eqref{eq:star-S}}{=}\;
   \alpha_\xi(\hat\varphi_{t*}w) \;-\; \alpha_\xi(i_{t*}v)
   \;=\; 0,
\]
where the final equality is the \textbf{fiber-product identity}
$i_{t*}v = \hat\varphi_{t*}w$. Hence $\widetilde\omega_t$ is basic.
 
\medskip

\noindent\textbf{Non-degeneracy.}
We show that $\ker \widetilde\omega_t$ at a point
$(s, m) \in \S_t \times_{\X_t} \M_t$ is exactly the tangent space to the
diagonal $\H_t$-orbit through $(s, m)$. After quotienting by $H_t$, the
kernel descends to zero on $Z_t$, proving non-degeneracy of
$\omega_{Z_t}$.
 
\noindent\textbf{Step 1. $H_t$-orbits are in} $\ker \widetilde\omega_t$. This is
the basiness calculation above: for every $\xi \in \Lie(\H_t)$, the fundamental vector field $(\hat\xi_S, \hat\xi_M)$ is in
$\ker \widetilde\omega_t$.
 
\noindent\textbf{Step 2.} Every element of $\ker \widetilde\omega_t$ comes from
$\Lie(\H_t)$. Let
$(\hat v, \hat w) \in \ker \widetilde\omega_t$ at $(s, m)$. By
definition of $\widetilde\omega_t$ on tangent vectors satisfying the fiber-product condition $i_{t*}v = \hat\varphi_{t*}w$:
\[
   \omega_t(\hat w, w) - \gamma_t(\hat v, v) = 0
   \qquad \forall\, (v, w) \text{ with } i_{t*}v = \hat\varphi_{t*}w.
\]
Setting $v = 0$, $w \in \ker \hat\varphi_{t*}$:
$\omega_t(\hat w, w) = 0$ for all $w \in \ker \hat\varphi_{t*}$. By
minimal non-degeneracy
$\ker \hat\varphi_{t*} \cap \ker \omega_t = 0$ and the equivalent dual
formulation, this constrains $\hat w$ to lie in
$\omega_t^\flat{}^{-1}(\text{im} \hat\varphi_t^{*})$, equivalently in the
image of the algebroid action $\rho_{M_t}$. So
$\hat w = \rho_{\M_t}(a)$ for a unique
$a = (\hat a_S, \alpha_a) \in L_{\X_t}|_{\S_t}$ with
$\hat\varphi_{t*}\hat w = \hat a_S = i_{t*}\hat v$ (the latter by the
fiber-product condition for the kernel vector itself, applied
diagonally). Hence $\hat a_S = i_{t*}\hat v$ lies in $T\S_t$.
 
Returning to the kernel equation with general $(v, w)$ satisfying
$i_{t*}v = \hat\varphi_{t*}w$:
\[
   \omega_t(\hat w, w) - \gamma_t(\hat v, v)
   = \alpha_a(\hat\varphi_{t*}w) - \gamma_t(\hat v, v)
   = \alpha_a(i_{t*}v) - \gamma_t(\hat v, v),
\]
using \eqref{eq:star-M} for $\hat w = \rho_{\M_t}(a)$ with $1$-form
component $\alpha_a$, and the fiber-product identity. This must vanish
for all $v \in T\S_t$, so
\[
   \alpha_a(i_{t*}v) = \gamma_t(\hat v, v) \qquad \forall v \in T_s S_t.
\]
This is precisely \eqref{eq:star-S} with
$(\hat\xi_S, \alpha_\xi) = (\hat v, \alpha_a)$, identifying
$a = (\hat v, \alpha_a)$ with an element of the stabilizer subalgebroid
$A_{\S_t,\gamma_t} \subset L_{\X_t}|_{\S_t}$. Local freeness of the
$\H_t$-action at $(s, m)$ ensures that
$A_{\S_t,\gamma_t}$ at $s$ is exactly $\Lie(\H_t)$ at $s$. Hence
$a \in \Lie(\H_t)$, and $(\hat v, \hat w) = (\hat a_S, \rho_{\M_t}(a))
 = (\hat\xi_S, \hat\xi_M)$ for $\xi = a$, the fundamental vector field
of the $\H_t$-action. Then, combining Steps 1 and 2, $\ker \widetilde\omega_t$ at $(s, m)$ equals the
tangent space to the $\H_t$-orbit. After quotienting,
$\omega_{Z_t}$ is non-degenerate on $Z_t$.\\

The non-degeneracy is smooth in $t$ because all the family data
(strong Dirac map, IM form, stabilizer subalgebroid, $\H_t$-orbits) is smooth in $t$ by Theorem~\ref{thm:ham-deform} and
Definition~\ref{def:family-lag}.
\end{proof}

\begin{proof}[of Theorem \ref{thm:shifted}]
The pair $(\Omega_t, \lambda_{X,t})$ is smooth in $t$ by Definition \ref{def:family}. At $t = 0$, $\lambda_{X,0} = 0$ and $\Omega_0$ is symplectic, giving the standard $1$-shifted symplectic structure of the symplectic groupoid $\G_0$.\\

By \cite[Theorem 5.6]{bualibanu2026reduction}, the Lagrangian condition for $[\S_t / \H_t] \to [\X_t / \G_t]$ is equivalent to the acyclicity of the complex $C^\bullet_t$ in \eqref{eq:Lag-complex}. By Definition \ref{def:family-lag}(L1)--(L2), the terms of $C^\bullet_t$ have locally constant rank, and the differentials have locally constant rank. Together these imply that the kernel and image bundles $\ker \alpha_t, \text{im} \alpha_t, \ker \beta_t, \text{im} \beta_t, \ker \delta_t$ have locally constant rank in $t$, so the cohomology bundles $H^i(C^\bullet_t)$ are smooth vector bundles over $\S_t$ of locally constant rank in $t$. Therefore, the function $t \mapsto \mathrm{rk}\, H^i(C^\bullet_t)$ is locally constant on $\R$. Since $\R$ is connected, this function is globally constant. By hypothesis, $H^i(C^\bullet_t) = 0$ for $t \neq 0$. Hence $\mathrm{rk}\, H^i(C^\bullet_t) = 0$ for all $t \in \R$, and in particular $H^i(C^\bullet_0) = 0$. So $C^\bullet_0$ is acyclic, and the map $[\S_0 / \H_0] \to [\X_0 / \G_0]$ is a Lagrangian morphism for the symplectic groupoid $\G_0$, i.e.\ a reduction level in the Poisson sense.\\

By Theorem \ref{thm:ham-deform}, $(\M_t, \omega_t, \hat\varphi_t)$ is a Hamiltonian $\G_t$-space for each $t$. By \cite[Proposition 5.5]{bualibanu2026reduction}, this is equivalent to a Lagrangian morphism $[\M_t / \G_t] \to [\X_t / \G_t]$. Smoothness in $t$ is direct from the smoothness of the Hamiltonian-$\G$-space data in Theorem \ref{thm:ham-deform}.\\

Smoothness of $\mathcal{Z} \to \mathbb{R}$ and the
diffeomorphism $\S_t \times_{\X_t} \M_t \cong \hat\varphi_t^{-1}(\S_t)$ are
established in Lemma~\ref{lem:smooth-family-fibre-product}.
Closedness of the upstairs form $\widetilde\omega_t$, hence of the
descended form $\omega_{Z_t}$, is
Proposition~\ref{prop:closedness}. Non-degeneracy of $\omega_{Z_t}$ is
Proposition~\ref{prop:nondegeneracy}. Combining, $\omega_{Z_t}$ is
symplectic in the ordinary sense (closed and non-degenerate) for every
$t \in \mathbb{R}$; in particular at $t = 0$ this is
Marsden-Weinstein symplectic reduction.
 
By \cite[Proposition 3.7 and Remark 5.12]{bualibanu2026reduction}, applied fiberwise, $Z_t$ is diffeomorphic to
$(\M_t \times_{\X_t} (\G_t)_{\S_t})/\G_t$ and presents
the stack-level fiber product
$[\S_t/\H_t] \times_{[\X_t/\G_t]} [\M_t/\G_t]$.
Smoothness in $t$ of this identification follows from
Lemma~\ref{lem:smooth-family-fibre-product} together with the
smooth-family data of Theorem~\ref{thm:ham-deform} (3).

\end{proof}

\begin{Rem}\label{rem:rank-constancy}
In the principal applications (see the next section), condition 2 is verified by direct construction or by appealing to ambient regularity hypotheses (e.g.\ regular moment values, locally free actions).
\end{Rem}
 
\begin{Rem}
The clean-intersection and local-freeness hypotheses of Theorem \ref{thm:shifted}(4) cannot be dropped. Already in the simplest example of reduction at a point for a circle action, the condition $\mu^{-1}(0)$ being regular / $S^1$ acting freely can fail at discrete values of a deformation parameter, causing the reduced space to become singular. We emphasize that our statement is a deformation of \emph{smooth} reductions, not a claim that smooth reductions exist for arbitrary deformations.\\

In the DAG setting of \cite{pantev2013shifted,calaque2021derived}, the derived fiber product $Z_t^{\mathrm{der}}$ is always $0$-shifted symplectic, and its classical truncation is $Z_t$ (smooth) precisely when the clean-intersection hypothesis holds. Part (4) of our theorem would be the classical shadow of a derived statement, which we don't establish here.
\end{Rem}

\begin{Exp}
    The fundamental example is the deformation $D(G)\rightrightarrows G \too T^{*}G\rightrightarrows \g^{*}$ of Proposition \ref{prop:fund}. In the presentation model of Section \ref{Section 7}: \begin{itemize}
        \item The presenting data of $[G/D(G)]$ deforms to the presenting data of $[\g^{*}/T^{*}G]$
        \item This is ''the groupoid lift'' of the Dirac deformation $L_G \rightsquigarrow L_{\g^*}$ of Proposition \ref{Dirac G}. 
    \end{itemize}
\end{Exp}

\begin{Exp}
    Fix $x\in \g$ such that $\O_x \subset \g^{*}$ is a regular co-adjoint orbit and $\exp_{|\O_x}$ is an embedding of $\O_x$ into a conjugacy class $\C_{\exp(x)} \subset G$ (this holds for $x$ in a neighborhood of $0\in \g^{*}$). Let $(M,\mu,\omega,G)$ be a quasi-Hamiltonian $G$-manifold that deforms to a Hamiltonian $G$-manifold $(N,\nu,\sigma,G).$ Under the conditions of Theorem \cite[Theorem 3.3]{burelle2026deformations}, by Theorem \ref{thm:shifted}, the reduction $M//_{\C_{\exp(tx)}}^{\text{qH}}G \coloneqq \widehat{\mu}^{-1}_{t}(\C_{\exp(tx)})/G$ is a smooth symplectic manifold for $t\neq 0,$ and deforms smoothly to the Marsden-Weinstein reduction $\nu^{-1}(\O_x)/G$ at $t = 0.$ 
\end{Exp}

\begin{Exp}(Implosion and Steinberg slice.)

\noindent\textbf{Implosion.} The quasi-Hamiltonian implosion $D(G)_{\exp(\bar{A})} \too G$ and the symplectic implosion $(T^{*}G)_{\t^{*}}$ (\cite[Example 3.8(1)]{bualibanu2026reduction}) are both singular stratified spaces (see \cite{guillemin2002symplectic,hurtubise2006group}). The deformation $D(G) \rightsquigarrow T^{*}G$ restricts to a stratified deformation of the imploded spaces (see \cite{maiza2026multiplicative}). Theorem \ref{thm:shifted}(2) applies stratum by stratum (with Condition (L2) verified on each stratum). The statement ''$[\exp(\bar{A})/\H_{\bar{A}}] \too [G/D(G)]$ deforms to $[\t^{*}_{+}/\H_{\t^{*}_{+}}]$'' is an equality of families of stratified Lagrangian morphisms.\\

\noindent\textbf{Steinberg slice.} Assume $G$ is a complex reductive group with Steinberg cross-section $\Sigma \subset G$ and Kostant cross-section $\Sigma_0 \subset \g^{*}.$ The one-sided Steinberg slice $D(G)_{\Sigma}$ deforms to the one-sided Kostant slice $(T^{*}G)_{\Sigma_0}.$ The Lagrangian $[\Sigma/\{e\}]\too [G/D(G)]$ deforms to $[\Sigma_{0}/\{0\}].$ Here both $\Sigma$ and $\Sigma_0$ are smooth slices and the Lagrangian condition is the transversality-cross section property, which is open and stable. Conditions (L1) and (L2) are verified by the constancy of slice dimensions along the family.
\end{Exp}

\begin{Exp}(The Grothendieck-Springer resolution.)
This example requires $G$ to be a complex reductive group with Borel subgroup $B$ and maximal torus $T.$ Let $\mathfrak{b} = \Lie(B), \t = \Lie(T), U = [B,B], \mathfrak{u} = \Lie(U).$ \\ By \cite[Example 4.8]{bualibanu2026reduction}, the multiplicative Grothendieck-Springer resolution $G \times^{B} B \too G, \quad [g:b] \mtoo gbg^{-1},$ is a Hamiltonian quasi-Poisson $G$-manifold. Viewed as a Hamiltonian $D(G)$ space, it gives a Lagrangian $[(G \times^{B}B)/D(G)] \too [G/D(G)].$\\
The deformation $B \rightsquigarrow \mathfrak{b}$ (The Dirac deformation of Proposition \ref{Dirac G} applied to $B$ with its inherited inner product) induces a deformation $G\times^{B} B \rightsquigarrow G\times^{B} \mathfrak{b} = \tilde{\g},$ the ordinary Grothendieck-Springer resolution. By Theorem \ref{thm:ham-deform} applied to this family, the Hamiltonian $D(G)$-space structure on $G\times^{B} B$ deforms to the Hamiltonian $T^{*}G$-structure on $\tilde{\g}.$ The Lagrangian \[
[(G\times^{B}B)/D(G)] \too [G/D(G)]
\]
deforms to 
\[
[\tilde{\g}/T^{*}G] \too [\g^{*}/T^{*}G].
\]
Fiberwise, the quasi-Hamiltonian leaves $G\times^{B} tU$ (indexed by $t\in T$) deform to the Hamiltonian leaves $G \times^{B}(\xi + u)$ (indexed by $\xi \in \t$).
\end{Exp}

\bibliographystyle{plain}
\bibliography{dirac.bib}
\end{document}